\documentclass{amsart}
\usepackage{amssymb,amsmath,amsthm,amscd,bm}
\usepackage{leftidx}
\usepackage{graphicx}
\usepackage{color}
\usepackage{microtype}
\usepackage{hyperref}
\usepackage[usenames,dvipsnames,svgnames,table]{xcolor}
\usepackage{tikz}
\usepackage{comment}
\usepackage{epstopdf} 
\usepackage{float}
\usepackage{tikz-cd}
\usepackage[all]{xy}
\newtheorem{lemma}{Lemma}[section]
\newtheorem{theorem}[lemma]{Theorem}
\newtheorem{corollary}[lemma]{Corollary}
\newtheorem{proposition}[lemma]{Proposition}

\theoremstyle{definition}
\newtheorem{definition}[lemma]{Definition}
\theoremstyle{remark}
\newtheorem{remark}[lemma]{Remark}
\newcommand{\C}{\mathbb{C}}
\newcommand{\D}{\mathbb{D}}
\newcommand{\N}{\mathbb{N}}
\newcommand{\Q}{\mathbb{Q}}
\newcommand{\R}{\mathbb{R}}
\newcommand{\Z}{\mathbb{Z}}

\newcommand{\cC}{\mathcal{C}}

\DeclareMathOperator{\re}{Re}
\DeclareMathOperator{\Int}{int}

\renewcommand{\epsilon}{\varepsilon}
\renewcommand{\phi}{\varphi}

\newcommand*\circcled[1]{\tikz[baseline=(char.base)]{
\node[shape=circle,draw,inner sep=1.3pt] (char) {#1};}}
\newcommand{\ciq}{\circcled{?}}

\date{\today}

\begin{document}

\title[Dynamics of Schwarz reflections]{Dynamics of Schwarz reflections:\\ The Mating Phenomena}

\author[S.-Y. Lee]{Seung-Yeop Lee}
\address{Department of Mathematics and Statistics, University of South Florida, Tampa, FL 33620, USA}
\email{lees3@usf.edu}

\author[M. Lyubich]{Mikhail Lyubich}
\address{Institute for Mathematical Sciences, Stony Brook University, NY, 11794, USA}
\email{mlyubich@math.stonybrook.edu}

\author[N. G. Makarov]{Nikolai G. Makarov}
\address{Department of Mathematics, California Institute of Technology, Pasadena, California 91125, USA}
\email{makarov@caltech.edu}

\author[S. Mukherjee]{Sabyasachi Mukherjee}
\address{School of Mathematics, Tata Institute of Fundamental Research, 1 Homi Bhabha Road, Mumbai 400005, India}
\email{sabya@math.tifr.res.in}


\maketitle

\begin{abstract}
We initiate the exploration of a new class of anti-holomorphic dynamical systems generated by Schwarz reflection maps associated with quadrature domains. More precisely, we study Schwarz reflection with respect to the deltoid, and Schwarz reflections with respect to the cardioid and a family of circumscribing circles. We describe the dynamical planes of the maps in question, and show that in many cases, they arise as unique conformal matings of quadratic anti-holomorphic polynomials and the ideal triangle group. \\

{\sc Titre.} Dynamique des r{\'e}flexions de Schwarz: un ph{\'e}nom{\`e}ne d'accouplement.\\

{\sc R\'esum\'e.} Nous entamons l'exploration d'une nouvelle classe de syst{\`e}mes dynamiques anti-holomorphes engendr{\'e}s par des r{\'e}flexions de Schwarz associ{\'e}s {\`a} des domaines {\`a} quadrature. Plus pr{\'e}cis{\'e}ment, nous {\'e}tudions des r{\'e}flexions de Schwarz par rapport {\`a} une delto{\"i}de, par rapport {\`a} une cardioide et {\`a} une famille de cercle circonscrits. Nous d{\'e}crivons le plan dynamique des applications en questions, et montrons que dans beaucoup de cas, elles sont obtenues {\`a} partir d'un unique accouplement conforme d'un polyn{\^o}me anti-holomorphe quadratique avec le groupe de r{\'e}flexion d'un triangle id{\'e}al.
\end{abstract}

\section{Introduction}\label{intro}

Schwarz reflections associated with quadrature domains (or disjoint unions of quadrature domains) provide an interesting class of dynamical systems. In some cases such systems combine the features of the dynamics of rational maps and reflection groups.

A domain in the complex plane is called a quadrature domain if the Schwarz reflection map with respect to its boundary extends meromorphically to its interior. They first appeared in the work of Davis \cite{Dav74}, and independently in the work of Aharonov and Shapiro \cite{AS1,AS2,AS}. Since then, quadrature domains have played an important role in various areas of complex analysis and fluid dynamics (see \cite{QD} and the references therein).

It is well known that except for a finite number of {\it singular} points (cusps and double points), the boundary of a quadrature domain consists of finitely many disjoint real analytic curves. Every non-singular boundary point has a neighborhood where the local reflection in $\partial\Omega$ is well-defined. The (global) Schwarz reflection $\sigma$ is an anti-holomorphic continuation of all such local reflections.

\begin{figure}[ht!]
\centering
\includegraphics[scale=0.18]{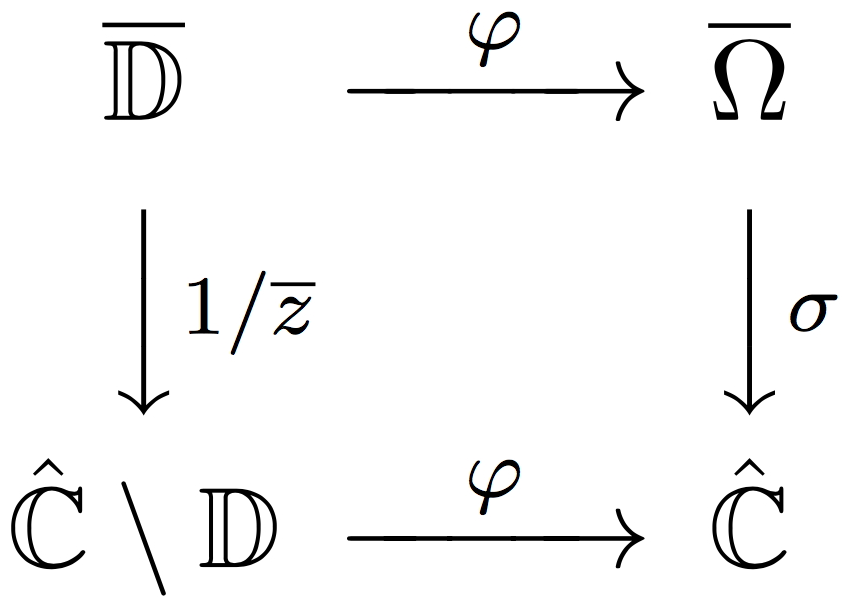}
\caption{The rational map $\phi$ semi-conjugates the reflection map $1/\overline{z}$ of $\D$ to the Schwarz reflection map 
$\sigma$ of $\Omega$ .}
\label{comm_diag_schwarz}
\end{figure}

Round discs on the Riemann sphere are the simplest examples of quadrature domains. Their Schwarz reflections are just the usual circle reflections. Further examples can be constructed using univalent polynomials or rational functions. Namely, if $\Omega$ is a {\it simply connected} domain and $\phi : \D \to\Omega$ is a univalent map from the unit disc onto $\Omega$, then $\Omega$ is a quadrature domain if and only if $\phi$ is a rational function. In this case, the Schwarz reflection $\sigma$ associated with $\Omega$ is semi-conjugate by $\phi$ to reflection in the unit circle.

Let us mention two specific examples: the interior of the {\it cardioid} curve and the exterior of the {\it deltoid} curve,
$$\left\{\frac z2-\frac{z^2}4:~|z|<1\right\}\quad{\rm and}\quad
\left\{\frac{1}{z}+\frac{z^2}{2}:~|z|<1\right\}.$$

In \cite{LM}, questions on equilibrium states of certain $2$-dimensional Coulomb gas models were answered using iteration of Schwarz reflection maps associated with quadrature domains (see \cite[\S 1]{LLMM2} for a brief account of this connection). It transpired from their work that these maps give rise to dynamical systems that are interesting in their own right. One of the principal goals of the current paper is to take a closer look at this class of maps and develop a general method of producing conformal matings between groups and anti-polynomials using Schwarz reflection maps associated with disjoint union of quadrature domains. In particular, we will prove that the Schwarz reflection map of the deltoid is a mating of the \emph{ideal triangle group} and the anti-polynomial $\overline{z}^2$.

The \emph{ideal triangle group} $\mathcal{G}$ is generated by the reflections in the sides of a hyperbolic triangle $\Pi$ in the open unit disk $\D$ with zero angles. Denoting the anti-M{\"o}bius reflection maps in the three sides of $\Pi$ by $\rho_1$, $\rho_2$, and $\rho_3$, we have $$\mathcal{G}=\langle\rho_1, \rho_2, \rho_3: \rho_1^2=\rho_2^2=\rho_3^2=\mathrm{id}\rangle<\mathrm{Aut}(\D).$$ $\Pi$ is a fundamental domain of the group. The tessellation of $\D$ by images of the fundamental domain under the group elements are shown in Figure~\ref{tessellation_pic}.
In order to model the dynamics of Schwarz reflection maps, we define a map $$\rho:\D\setminus\Int{\Pi}\to\D$$ by setting it equal to $\rho_k$ in the connected component of $\D\setminus\Int{\Pi}$ containing $\rho_k(\Pi)$ (for $k=1,2,3$). The map $\rho$ extends to an orientation-reversing double covering of $\mathbb{T}=\partial\D$ admitting a Markov partition $\mathbb{T}=[1,e^{2\pi i/3}]\cup[e^{2\pi i/3},e^{4\pi i/3}]\cup[e^{4\pi i/3},1]$ with transition matrix $$M:=\begin{bmatrix} 0 & 1 & 1\\ 1 & 0 & 1\\ 1 & 1 & 0\end{bmatrix}.$$  

\begin{figure}[ht!]
\begin{center}
\includegraphics[scale=0.06]{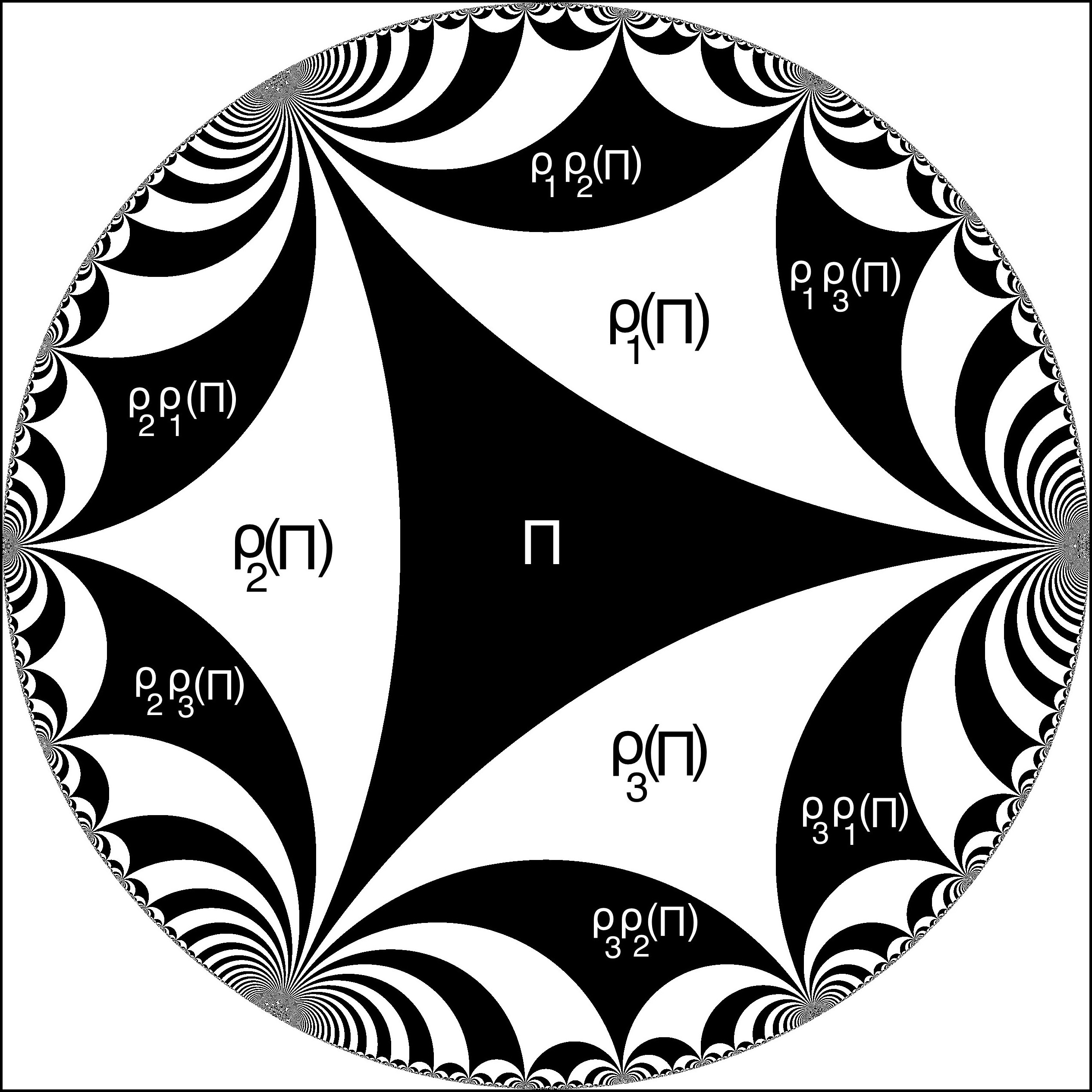} \includegraphics[scale=0.06]{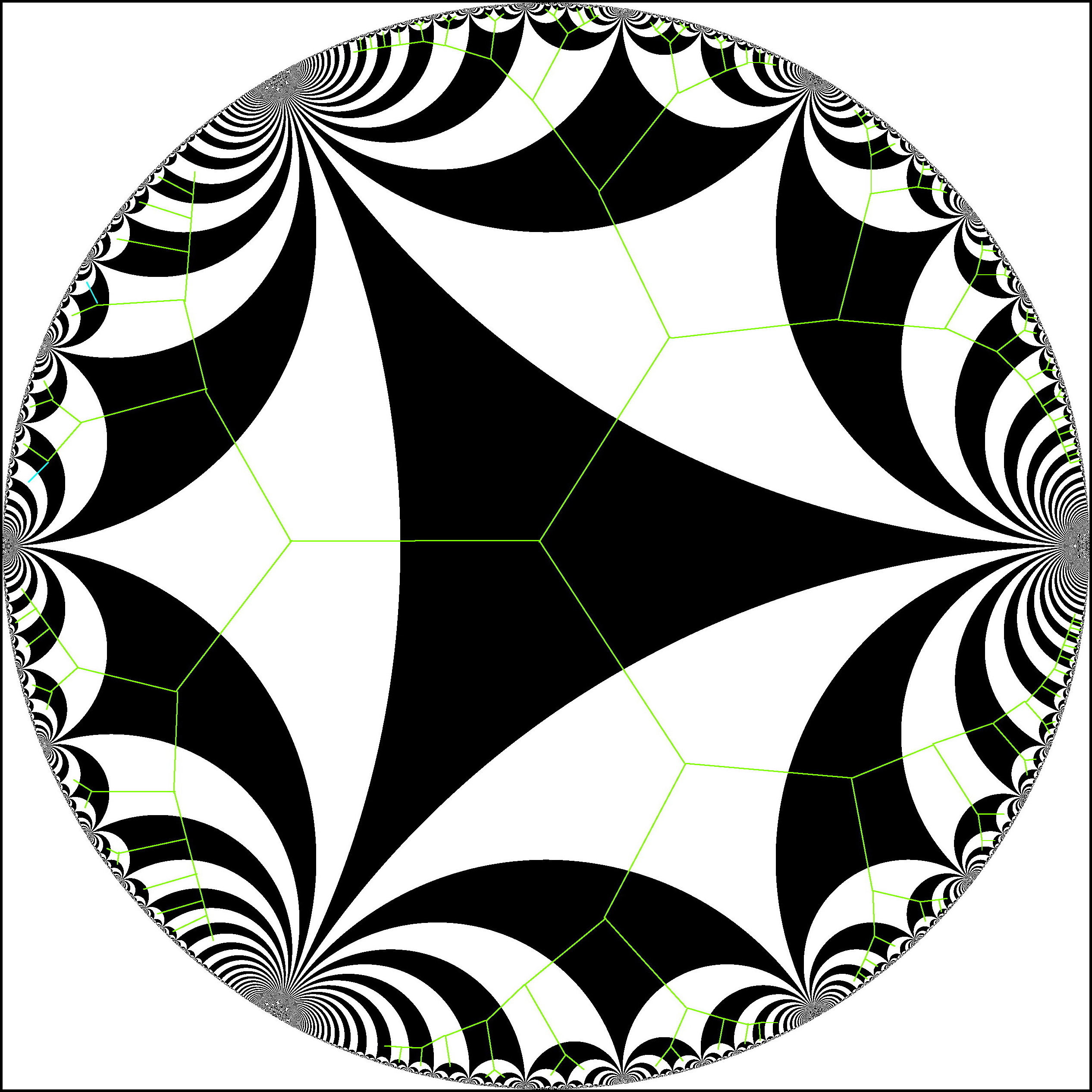}
\end{center}
\caption{Left: The $\mathcal{G}$-tessellation of $\D$ and the formation of a few initial tiles are shown. Right: The dual tree to the $\mathcal{G}$-tessellation of $\D$.}
\label{tessellation_pic}
\end{figure}

Since the Schwarz reflection maps (which are anti-holomorphic) studied in this paper have a unique, simple critical point, it is not surprising that their dynamics is closely related to the dynamics of quadratic anti-holomorphic polynomials (anti-polynomials for short). The dynamics of quadratic anti-polynomials and their connectedness locus, the Tricorn, was first studied in \cite{CHRS} (note that they called it the Mandelbar set). Their numerical experiments showed structural differences between the Mandelbrot set and the Tricorn. However, it was Milnor who first observed the importance of the Tricorn; he found little Tricorn-like sets as prototypical objects in the parameter space of real cubic polynomials \cite{M3}, and in the real slices of rational maps with two critical points \cite{M4}. Since then, dynamics of anti-holomorphic polynomials and the topological structure of the associated connectedness loci (in particular, the Tricorn) have been studied by various people. We refer the readers to \cite[\S 2]{LLMM2} for a survey on this topic.

The connection between quadratic anti-polynomials and the ideal triangle group comes from the fact that the anti-doubling map $$m_{-2}:\R/\Z\to\R/\Z,\ \theta\mapsto-2\theta$$ (which models the `external' dynamics of quadratic anti-polynomials) and the map $\rho$ described above admit the same Markov partition with the same transition matrix. This allows one to construct a circle homeomorphism $\mathcal{E}:\mathbb{T}\to\mathbb{T}$ that conjugates the reflection map $\rho$ to the anti-doubling map $m_{-2}$. The conjugacy $\mathcal{E}$, which is a version of the Minkowski question mark function, serves as a connecting link between the dynamics of Schwarz reflections and that of quadratic anti-polynomials (see the article by Shaun Bullett in \cite[\S 7.8]{BrFa} for a detailed exposition of the Minkowski question mark function, and Subsection~\ref{question_mark_subsec} for an explicit relation between $\mathcal{E}$ and the Minkowski question mark function). The conjugacy $\mathcal{E}$ plays a crucial role in the paper (see Section~\ref{ideal_triangle} for details).

Let us now describe the basic dynamical objects associated with iteration of Schwarz reflection maps. Given a disjoint collection of quadrature domains, we call the complement of their union a \emph{droplet}. Removing the finitely many singular points from the boundary of a droplet yields the \emph{desingularized droplet} or the \emph{fundamental tile}. One can then look at a partially defined anti-holomorphic dynamical system $\sigma$ that acts on the closure of each quadrature domain as its Schwarz reflection map. Under this dynamical system, the Riemann sphere $\widehat{\C}$ admits a dynamically invariant partition. The first one is an open set called the {\it escaping/tiling set}, it is the set of all points that eventually escape to the fundamental tile (on the interior of which $\sigma$ is not defined). Alternatively, the tiling set is the union of all ``tiles", the {\it fundamental tile} $T$ and the components of all its preimages under the iterations of $\sigma$ (the tiling structure is reminiscent of tessellations of the unit disk under reflection groups). The second invariant set is the {\it non-escaping} set, the complement of the tiling set; or equivalently, the set of all points on which $\sigma$ can be iterated forever (the non-escaping set is analogous to the filled Julia set in polynomial dynamics; i.e., the set of points with bounded forward orbits under a polynomial). When the tiling set contains no critical points of $\sigma$, it is often the case that the dynamics of $\sigma$ on its non-escaping set resembles that of an anti-polynomial on its filled Julia set, while the $\sigma-$action on the tiling set exhibits features of reflection groups.

This is precisely the case if $T$ is the deltoid. Figure~\ref{deltoid_reflection_pic} (right) shows the tiling and the non-escaping sets as well as their common boundary, which is simultaneously analogous to the Julia set of an anti-polynomial (i.e., the boundary of the filled Julia set) and to the limit set of a group. In fact, the Schwarz reflection $\sigma$ of the deltoid is the ``mating" of the anti-polynomial $z\mapsto \overline{z}^2$ and the reflection map $\rho$ in the following sense. The conformal dynamical systems
$$\rho: \overline{\mathbb D}\setminus\Int{\Pi}\to  \overline{\mathbb D}\quad \mathrm{and}\quad f_0: \widehat\C\setminus\mathbb D\to \widehat\C\setminus\mathbb D,\ z\mapsto\overline{z}^2$$ can be glued together by the circle homeomorphism $\mathcal{E}$ (which conjugates $\rho$ to $f_0$ on $\mathbb{T}$) to yield a partially defined topological map $\eta$ on a topological $2$-sphere. There exists a unique conformal structure on this $2$-sphere which makes $\eta$ an anti-holomorphic map conformally conjugate to $\sigma$.

In Section~\ref{deltoid_reflection}, we study in detail the dynamics of the Schwarz reflection map $\sigma$ of the deltoid (which is one of the simplest non-trivial dynamical systems generated by Schwarz reflections), and prove the following theorem:

\begin{theorem}[Dynamics of deltoid reflection]\label{deltoid_mating_dynamical_partition} 
1) The dynamical plane of the Schwarz reflection $\sigma$ of the deltoid can be partitioned as
$$\widehat {\mathbb C}=T^\infty\sqcup \Gamma\sqcup A(\infty),$$
where $T^\infty$ is the tiling set, $A(\infty)$ is the basin of infinity, and $\Gamma$ is their common boundary (which we call the limit set). Moreover, $\Gamma$ is a conformally removable Jordan curve. 

2) $\sigma$ is the unique conformal mating of the reflection map $\rho: \overline{\mathbb D}\setminus\Int{\Pi}\to  \overline{\mathbb D}$ and the anti-polynomial $f_0: \widehat\C\setminus\mathbb D\to \widehat\C\setminus\mathbb D,\ z\mapsto\overline{z}^2$. 
\end{theorem}

This is a new occasion of a phenomenon discovered by Bullett and Penrose \cite{BP} (and more recently studied by Bullett and Lomonaco \cite{BuLo1}), where matings of holomorphic quadratic polynomials and the modular group were realized as {\it holomorphic correspondences}. 

Schwarz reflections associated with quadrature domains provide us with a general method of constructing such matings. In this paper, we initiate the study of the following one-parameter family of Schwarz reflection maps, which give rise to conformal matings between the ideal triangle group and quadratic anti-polynomials $\overline{z}^2+c$. We consider a fixed cardioid, and for each complex number $a$, we consider the circle centered at $a$ circumscribing the cardioid (see Figure~\ref{real_slit_double} and Figure~\ref{various_limit_sets}). Let $T_a$ denote the resulting droplet (the closed disc minus the open cardioid), and let $F_a$ denote the corresponding Schwarz reflection (the circle reflection $\sigma_a$ in its exterior, and the reflection $\sigma$ with respect to the cardioid in its interior). We denote this family of Schwarz reflections maps $F_a$ by $\mathcal{S}$ and call it the circle-and-cardioid family.

Note that the droplet $T_a$ has two singular points on its boundary. Removing these two singular points from $T_a$, we obtain the \emph{desingularized droplet}/\emph{fundamental tile} $T_a^0$. Recall that the \emph{non-escaping set} of $F_a$ (denoted by $K_a$) consists of all points that do not escape to the fundamental tile $T_a^0$ under iterates of $F_a$, while the \emph{tiling set} of $F_a$ (denoted by $T_a^\infty$) is the set of points that eventually escape to $T_a^0$. We call the components of the iterated preimages of $T_a^0$ \emph{tiles} of $F_a$. The boundary of the tiling set is called the \emph{limit set}, and is denoted by $\Gamma_a$.

The Schwarz reflection map $F_a$ is unicritical; indeed, the circle reflection map $\sigma_a$ is univalent, while the cardioid reflection map $\sigma$ has a unique critical point at the origin. As in the case of quadratic polynomials, the non-escaping set of $F_a$ is {\it connected} if and only if it contains the unique critical point of $F_a$; i.e. the critical point does not escape to the fundamental tile. On the other hand, if the critical point escapes to the fundamental tile, the corresponding non-escaping set is totally disconnected (see Figure~\ref{various_limit_sets} for a connected non-escaping set and a totally disconnected non-escaping set).

\begin{theorem}[Connectivity of the non-escaping set]\label{non_escaping_connected_intro}
1) If the critical point of $F_a$ does not escape to the fundamental tile $T_a^0$, then the conformal map $\psi_a$ from $T_a^0$ onto $\Pi$ extends to a biholomorphism between the tiling set $T_a^\infty$ and the unit disk $\D$. Moreover, the extended map $\psi_a$ conjugates $F_a$ to the reflection map $\rho$. In particular, $K_a$ is connected.

2) If the critical point of $F_a$ escapes to the fundamental tile, then the corresponding non-escaping set $K_a$ is a Cantor set.
\end{theorem}

This leads to the notion of the {\it connectedness locus} $\cC(\mathcal{S})$ as the set of parameters with connected non-escaping sets. Equivalently, $\cC(\mathcal{S})$ is exactly the set of parameters for which the tiling set is unramified; i.e., the map $F_a$ restricted to each tile is an unramified covering. As a slight abuse of terminology, we will often refer to a tile of $F_a$ as ramified/unramified if the restriction of $F_a$ to that tile is a ramified/unramified covering.

In order to study the maps $F_a$ in the connectedness locus $\cC(\mathcal{S})$, we carry out a detailed analysis of their Fatou components, non-repelling periodic points and their relationship with critical orbits. In his context, the next theorem shows that many aspects of the classical Fatou-Julia theory carry over to these partially defined dynamical systems.

\begin{theorem}[Fatou components and critical orbits]\label{fatou_mega_thm}
Let $a\in\C\setminus(-\infty,-\frac{1}{12}]$. Then the following hold true.
\begin{enumerate}
\item Every Fatou component of $F_a$ is eventually preperiodic. Every periodic Fatou component of $F_a$ is either the (immediate) basin of attraction of an attracting cycle, or the (immediate) basin of attraction of a parabolic cycle, or a Siegel disk (i.e., a Fatou component containing an irrationally neutral, linearizable periodic point).

\item If $F_a$ has an attracting or parabolic cycle, then the forward orbit of the critical point $0$ converges to this cycle. Moreover, the basin of attraction of this attracting or parabolic cycle is equal to $\Int{K_a}$.

\item If $U$ is a Siegel disk of $F_a$, then $\partial U\subset\overline{\left\{F_a^{\circ n}(0)\right\}_{n\geq0}}$. Moreover, every Fatou component of $F_a$ eventually maps to this cycle of Siegel disks.

\item Every Cremer point (i.e. an irrationally neutral, non-linearizable periodic point) of $F_a$ is also contained in $\overline{\left\{F_a^{\circ n}(0)\right\}_{n\geq0}}$. Moreover, if $F_a$ has a Cremer point, then $\Int{K_a}=\emptyset$; i.e. $K_a=\Gamma_a$.
\end{enumerate}
\end{theorem}

The geometrically finite maps (i.e. maps with attracting/parabolic cycles, and maps with non-escaping, strictly preperiodic critical point) of $\mathcal{S}$ are of particular importance. They belong to the connectedness locus $\cC(\mathcal{S})$, and their topological and analytic properties are more tractable.

\begin{theorem}[Limit sets of geometrically finite maps]\label{geom_finite_limit_set}
Let $F_a$ be geometrically finite. 

1) The limit set $\Gamma_a$ of $F_a$ is locally connected. Moreover, the area of $\Gamma_a$ is zero.

2) The set of iterated preimages of the cardioid cusp is dense in the limit set $\Gamma_a$. Moreover, the set of repelling periodic points of $F_a$ is dense in $\Gamma_a$.
\end{theorem}

A good understanding of the dynamics of geometrically finite maps in the circle-and-cardioid family allow us to produce plenty of examples of conformal matings between the reflection map $\rho$ and quadratic anti-polynomials. While this matter will be pursued in greater detail in a sequel to this work \cite{LLMM2}, here we will elucidate how the simplest map in the circle-and-cardioid family arises as a conformal mating of a quadratic anti-polynomial and the reflection map $\rho$. Let us briefly mention the precise meaning of conformal mating in this context. For any $c_0$ in the Tricorn with a locally connected Julia set, one can glue the conformal dynamical systems
$$\rho: \overline{\mathbb D}\setminus\Int{\Pi}\to  \overline{\mathbb D}\quad \mathrm{and}\quad f_{c_0}: \mathcal{K}_{c_0}\to\mathcal{K}_{c_0},\ z\mapsto\overline{z}^2+c_0$$ (where $\mathcal{K}_{c_0}$ is the filled Julia set of $f_{c_0}$; i.e., the set of points with bounded forward orbits under $f_{c_0}$) by a factor of the circle homeomorphism $\mathcal{E}$ yielding a partially defined topological map $\eta$ on a topological $2$-sphere. We say that $F_{a_0}$ is a conformal mating of $\rho$ and the quadratic anti-polynomial $f_{c_0}$ if this topological $2$-sphere admits a conformal structure that turns $\eta$ into an anti-holomorphic map conformally conjugate to $F_{a_0}$. If such a conformal structure is unique up to a M{\"o}bius map, then $F_{a_0}$ will be called the unique conformal mating of $\rho$ and $f_{c_0}$.

We show in \cite{LLMM2} that there exists a natural combinatorial bijection between the geometrically finite maps in the family $\mathcal{S}$ and those in the {\it basilica limb} $\mathcal{L}$ of the Tricorn (see \cite[\S 2.2.11]{LLMM2}). The combinatorial models of the corresponding maps are related by the circle homeomorphism $\mathcal{E}$. This allows us to demonstrate that every geometrically finite map in $\mathcal{S}$ is a conformal mating of a unique geometrically finite quadratic anti-polynomial and the reflection map $\rho$. Using the combinatorial bijection between geometrically finite maps mentioned above, we further establish that the locally connected topological model of $\cC(\mathcal{S})$ is naturally homeomorphic to that of the basilica limb $\mathcal{L}$ (where the homeomorphism is induced by the map $\mathcal{E}$).

Let us now detail the organization of the paper. In Section~\ref{ideal_triangle}, we give a self-contained description of the ideal triangle group, the associated tessellation of the unit disk, and the reflection map $\rho$. Here we also define the topological conjugacy $\mathcal{E}$ between $\rho$ and the anti-doubling map $\theta\mapsto-2\theta$ on the circle. 
In Section~\ref{sec_quad_domain}, we briefly review some general properties of quadrature domains and Schwarz reflection maps. Section~\ref{deltoid_reflection} is devoted to the study of the dynamics of Schwarz reflection with respect to the deltoid. The principal goal of this section is to prove Theorem~\ref{deltoid_mating_dynamical_partition} by interpreting the deltoid reflection map as the unique conformal mating of the anti-polynomial $\overline{z}^2$ and the reflection map $\rho$. 
In Section~\ref{sec_main}, we turn our attention to the circle-and-cardioid family $\mathcal{S}$. After establishing some basic mapping properties of the Schwarz reflection map associated to the cardioid in Subsection~\ref{sec_cardioid}, we carry out a detailed discussion of the elementary dynamical properties of $F_a$ and the associated dynamically invariant sets in Subsection~\ref{dynamical_plane}. Here we prove Theorem~\ref{fatou_mega_thm} by establishing a classification theorem for Fatou components (connected components of the interior of the non-escaping set) and studying the interaction between various types of Fatou components and the post-critical orbit of $F_a$ (see Propositions~\ref{fatou_classification}, \ref{fatou_critical} and Corollaries~\ref{att_para_count}, \ref{siegel_cremer_complete}). Subsection~\ref{escape_set_group} concerns the dynamics of $F_a$ on its tiling set; we show in Proposition~\ref{schwarz_group} that for maps in the connectedness locus, the dynamics on the tiling set (which is simply connected) is conformally conjugate to the reflection map $\rho$, while for maps outside $\cC(\mathcal{S})$, such a conjugacy still exists on a subset of the tiling set containing the critical value. Finally we prove in Proposition~\ref{cantor_outside} that maps outside $\cC(\mathcal{S})$ have totally disconnected non-escaping sets. This completes the proof of Theorem~\ref{non_escaping_connected_intro}. It is worth mentioning that for $a\in\cC(\mathcal{S})$, the dynamical uniformization of the tiling set of $F_a$ leads to a combinatorial model (quotient of the unit disk by a geodesic lamination) of the non-escaping set. In Section~\ref{geom_fin_sec}, we study geometrically finite maps in the family $\mathcal{S}$. We discuss some basic topological, analytic and measure-theoretic properties of hyperbolic and parabolic maps in Subsection~\ref{hyp_para_sec} and of Misiurewicz maps in Subsection~\ref{misi_maps}. These results immediately imply Theorem~\ref{geom_finite_limit_set}. In the final Section~\ref{basilica_mating_sec}, we use our knowledge of geometrically finite maps in $\mathcal{S}$ to illustrate the mating phenomena in the circle-and-cardioid family with a concrete example.
\vspace{2mm}

\noindent\textbf{Acknowledgements.} The second author was partially supported by NSF grants DMS-1600519 and 1901357. The fourth author was supported by the Institute for Mathematical Sciences at Stony Brook
University, an endowment from Infosys Foundation, and SERB research grant SRG/2020/000018 during parts of the work on this project. The second and the fourth author would also like to acknowledge the support of the Institute for Theoretical Studies at ETH Z{\"u}rich.

We thank the anonymous referees for their instructive comments that improved the exposition of the paper. All pictures of Schwarz reflection dynamical planes appearing in this paper were produced using the Wolfram Mathematica software.

\section{Ideal triangle group}\label{ideal_triangle}

The goal of this section is to review some basic properties of the ideal triangle group and its boundary extension. This will play an important role in our study of the ``escaping'' dynamics of Schwarz reflection maps.

Consider the open unit disk $\D$ in the complex plane. Let $\widetilde{C}_1$, $\widetilde{C}_2$, and $\widetilde{C}_3$ be the hyperbolic geodesics in $\D$ connecting the third roots of unity. These geodesics bound a closed ideal triangle (in the topology of $\D$), which we call $\Pi$.

Reflections with respect to the circular arcs $\widetilde{C}_i$ are anti-conformal involutions (hence automorphisms) of $\D$, and we call them $\rho_1$, $\rho_2$, and $\rho_3$. The maps $\rho_1$, $\rho_2$, and $\rho_3$ generate a subgroup $\mathcal{G}$ of $\mathrm{Aut}(\D)$. The group $\mathcal{G}$ is called the \emph{ideal triangle group}. As an abstract group, it is given by the generators and relations $\langle\rho_1, \rho_2, \rho_3: \rho_1^2=\rho_2^2=\rho_3^2=\mathrm{id}\rangle.$

We will denote the connected component of $\D\setminus \Pi$ containing $\Int{\rho_i(\Pi)}$ by $\D_i$. Note that $\D_1\cup \D_2\cup \D_3=\D\setminus\Pi$.

\subsection{The reflection map $\rho$, and symbolic dynamics}\label{rho_sec}
We now define the reflection map $\rho:\D\setminus\Int{\Pi}\to\D$ as:
$$ z \mapsto         \rho_i(z)\ \mathrm{if}\ z\in \D_i\cup\widetilde{C}_i,\ \mathrm{for}\ i=1,2,3.$$
Then $\rho(\D_1)\supset\D_2\cup\D_3$, $\rho(\D_2)\supset\D_3\cup\D_1$, $\rho(\D_3)\supset\D_2\cup\D_1$, and $\rho(\D_i)\cap\D_i=\emptyset$, for $i=1,2,3$. Thus, the dynamics of $\rho$ on $\D\setminus\Int{\Pi}$ is encoded by the transition matrix $M=\mathbf{1} -\mathrm{Id}$, where $\mathbf{1}$ is the $3\times 3$ matrix with all entries equal to $1$, and $\mathrm{Id}$ is the $3\times 3$ identity matrix. (The Markov map $\rho$ is analogous to the so-called Bowen-Series boundary maps associated with Fuchsian groups, compare \cite{BoSe}.)

Let $W:=\lbrace 1,2,3\rbrace$. An element $(i_1,i_2,\cdots)\in W^{\mathbb{N}}$ is called $M$-admissible if $M_{i_k,i_{k+1}}=1$, for all $k\in\mathbb{N}$. We denote the set of all $M$-admissible words in $W^{\mathbb{N}}$ by $M^\infty$. One can similarly define $M$-admissibility of finite words.

Note that $\Pi$ is a fundamental domain of $\mathcal{G}$. The tessellation of $\D$ arising from $\mathcal{G}$ will play an important role in this paper (see Figure~\ref{tessellation_pic}). 

\begin{definition}[Tiles]\label{def_tiles}
The images of the fundamental domain $\Pi$ under the elements of $\mathcal{G}$ are called \emph{tiles}. More precisely, for any $M$-admissible word $(i_1,\cdots,i_k)$, we define the tile $$T^{i_1,\cdots,i_k}:=\rho_{i_1}\circ\cdots\circ\rho_{i_k}(\Pi).$$
\end{definition}

It follows from the definition that $T^{i_1,\cdots,i_k}$ consists of all those $z\in\D$ such that $\rho^{\circ (n-1)}(z)\in\D_{i_n}$, for $n=1,\cdots,k$ and $\rho^{\circ k}(z)\in\Pi$. In other words, $$T^{i_1,\cdots,i_k}=\bigcap_{n=1}^{k}\rho^{-(n-1)}(\D_{i_n})\cap\rho^{-k}(\Pi).$$

Clearly, $\rho$ extends as an orientation-reversing $C^1$ double covering of $\mathbb{T}$ (real-analytic away from the fixed points $1$, $e^{2\pi i/3}$, and $e^{4\pi i/3}$) with associated Markov partition $\mathbb{T}=(\partial\D_1\cap\mathbb{T})\cup(\partial\D_2\cap\mathbb{T})\cup(\partial\D_3\cap\mathbb{T})$. The corresponding transition matrix is $M$ as above. Since $\vert\rho'\vert\geq 1$ on $\mathbb{T}$ with equality only at the third roots of unity, it follows that $\rho\vert_{\mathbb{T}}$ is an expansive map.

Recall that the set of all $M$-admissible words in $W^{\mathbb{N}}$ is denoted by $M^\infty$. It follows from expansiveness of $\rho\vert_{\mathbb{T}}$ that for any element of $M^\infty$, the corresponding infinite sequence of tiles shrinks to a single point of $\partial\D$. This allows us to define a continuous surjection 
$$Q:M^\infty\to\mathbb{T},\quad (i_1,i_2\cdots)\mapsto \bigcap_{n\in\mathbb{N}}\ \rho^{-(n-1)}(\partial\D_{i_n}\cap\mathbb{T})$$ 
which semi-conjugates the (left-)shift map on $M^\infty$ to the map $\rho$ on $\mathbb{T}$.

For any $M$-admissible sequence $(i_1,i_2,\cdots)$, let us consider the sequence $\{0, \rho_{i_1}(0), \rho_{i_1}\circ\rho_{i_2}(0), \cdots\}$. Since $d_{\D}(0,\rho_{1}(0))=d_{\D}(0,\rho_{2}(0))=d_{\D}(0,\rho_{3}(0))$, the hyperbolic distance (in $\D$) between any two consecutive points in this sequence is constant. Connecting consecutive points of this sequence by hyperbolic geodesics of $\D$, we obtain a curve in $\D$ that lands at $Q((i_1,i_2,\cdots))$.

\begin{definition}[Rays]\label{rays_triangle}
The curve constructed above is called a $\mathcal{G}$-ray at angle $Q((i_1,i_2,\cdots))$ (here, we identify $\mathbb{T}$ with $\R/\Z$).
\end{definition}

\begin{remark}\label{dual_tree}
The set of all $\mathcal{G}$-rays form a dual tree to the $\mathcal{G}$-tessellation of $\D$. As an abstract graph, it is isomorphic to the undirected Cayley graph of $\mathcal{G}$ with respect to the generating set $\{\rho_1,\rho_2,\rho_3\}.$
\end{remark}

\subsection{The conjugacy $\mathcal{E}$}\label{conjugacy_anti_doubling_rho_sec}

The expanding double covering of the circle $m_{-2}:\R/\Z\to\R/\Z,\ \theta\mapsto-2\theta,$ which is the action of quadratic anti-polynomials on angles of external dynamical rays (compare \cite[\S 2]{Sa}), is closely related to $\rho$. The general fact that two homotopic expansive maps of compact manifolds are topologically conjugate yields a topological conjugacy between $\rho\vert_{\mathbb{T}}$ and $\overline{z}^2\vert_{\mathbb{T}}$ (see \cite{CR}). However, in the current setup, one can construct such a conjugacy explicitly using the symbolic dynamics of the maps.

The map $m_{-2}$ admits the same Markov partition as $\rho$; namely $\R/\Z=I_1\cup I_2\cup I_3$, where $\displaystyle I_1=[0,1/3], I_2=[1/3,2/3], I_3=[2/3,1]$. Moreover, the associated transition matrix coincides with that of $\rho$. This allows us to define a continuous surjection 
$$P:M^\infty\to\R/\Z,\quad (i_1,i_2\cdots)\mapsto \bigcap_{n\in\mathbb{N}}\ m_{-2}^{-(n-1)}(I_{i_n})$$ 
which semi-conjugates the (left-)shift map on $M^\infty$ to the map $m_{-2}$ on $\R/\Z$. It is easy to see that the two coding maps $Q$ and $P$ have precisely the same fibers. It follows that $P\circ Q^{-1}$ induces a homeomorphism of the circle conjugating $\rho$ to $m_{-2}$. We will denote this conjugacy by $\mathcal{E}$.

\section{Quadrature domains and Schwarz reflections}\label{sec_quad_domain}

We now come to a general description of the main objects of this paper. While Schwarz reflections with respect to real-analytic curves are only locally defined, it is possible to extend this reflection map to a semi-global map in certain cases (e.g. reflection with respect to a circle extends to $\widehat{\C}$). 

Throughout this section, we let $\Omega\subsetneq\widehat{\C}$ be a domain such that $\infty\notin\partial\Omega$ and $\Int{\overline{\Omega}}=\Omega$. 
We will denote the complex conjugation map by $\iota$. 

\begin{definition}[Schwarz functions and quadrature domains]\label{quad_domain_def}
A \emph{Schwarz function} of $\Omega$ is a meromorphic extension of $\iota\vert_{\partial\Omega}$ to all of $\Omega$. More precisely, a continuous function $S:\overline{\Omega}\to\widehat{\C}$ of $\Omega$ is called a Schwarz function of $\Omega$ if it satisfies the following two properties:
\begin{enumerate}
\item $S$ is meromorphic on $\Omega$,

\item $S=\iota$ on $\partial \Omega$.
\end{enumerate}
The domain $\Omega$ is called a \emph{quadrature domain} if it admits a Schwarz function.
\end{definition}

It follows from the definition that a Schwarz function of a quadrature domain is unique. 
Therefore, for a quadrature domain $\Omega$, the map $\sigma:=\iota\circ S:\overline{\Omega}\to\widehat{\C}$ is the unique anti-meromorphic extension of the Schwarz reflection map with respect to $\partial \Omega$ (the reflection map fixes $\partial\Omega$ pointwise). We will call $\sigma$ the \emph{Schwarz reflection map of} $\Omega$.

We now define the notion of a quadrature function for a domain. 

\begin{definition}[Quadrature functions]\label{quad_func_def}
Let $\Omega$ be a quadrature domain. Functions in $H(\Omega)\cap C(\overline{\Omega})$ are called \emph{test functions} for $\Omega$ (if $\Omega$ is unbounded, we further require test functions to vanish at $\infty$). A rational map $R_\Omega$ is called a \emph{quadrature function} of $\Omega$ if all poles of $R_\Omega$ are inside $\Omega$ (with $R_\Omega(\infty)=0$ if $\Omega$ is bounded), and the identity $$\displaystyle \int_{\Omega} f dA=\frac{1}{2i} \oint_{\partial\Omega} f(z) R_{\Omega}(z) dz$$ holds for every test function of $\Omega$.
\end{definition}

The following theorem is classical (see \cite[Lemma~2.3]{AS}, \cite[Lemma~3.1]{LM}).

\begin{theorem}[Characterization of quadrature domains]\label{characterization}
The following are equivalent.
\begin{enumerate}
\item $\Omega$ is a quadrature domain; i.e. it admits a Schwarz function,

\item $\Omega$ admits a quadrature function $R_\Omega$,

\item The Cauchy transform $\widehat{\chi}_{\Omega}$ of the characteristic function $\chi_{\Omega}$ (of $\Omega$) is rational outside $\Omega$. 
\end{enumerate}
\end{theorem}

We call $d_\Omega:=\mathrm{deg}(R_\Omega)$ the \emph{order} of the quadrature domain $\Omega$. The poles of $R_\Omega$ are called \emph{nodes} of $\Omega$.

We now state an important theorem which guarantees that boundaries of quadrature domains are real-algebraic \cite[Theorem~5]{Gus}. 

\begin{theorem}[Real-algebraicity of boundaries of quadrature domains]\label{real_algebraic}
If $\Omega$ is a  quadrature domain, then $\partial \Omega$ is a real-algebraic curve, possibly minus finitely many isolated points. In particular, the only singularities on $\partial \Omega$ are double points and cusps.
\end{theorem}

The next proposition characterizes simply connected quadrature domains (see \cite[Theorem~1]{AS}).

\begin{proposition}[Simply connected quadrature domains]\label{simp_conn_quad}
Let $\Omega$ be simply connected. Then, $\Omega$ is a quadrature domain if and only if the Riemann uniformization $\phi:\mathbb{D}\to\Omega$ is rational.
\end{proposition}

For a simply connected quadrature domain $\Omega$, the rational map $\phi$ that restricts to a biholomorphism $\phi:\D\to\Omega$ of $\Omega$ semi-conjugates the reflection map $1/\overline{z}$ of $\D$ to the Schwarz reflection map 
$\sigma$ of $\Omega$ (see the commutative diagram in Figure~\ref{comm_diag_schwarz}). This allows us to compute the Schwarz reflection map of a simply connected quadrature domain. Moreover, the order of a simply connected quadrature domain is equal to $\mathrm{deg}\ \phi$.

For a comprehensive account on quadrature domains and their applications, we refer the readers to \cite{QD}. 

\section{Dynamics of deltoid reflection}\label{deltoid_reflection}

The goal of this section is to carry out a detailed study of the dynamics of Schwarz reflection with respect to the deltoid. 

\subsection{Schwarz reflection with respect to the deltoid}\label{deltoid_schwarz_partition}

In this subsection, we introduce the Schwarz reflection map of the deltoid, study its basic mapping properties, and prove some fundamental topological properties of the dynamically meaningful sets associated with the Schwarz reflection map.

\subsubsection{Deltoid}
\begin{proposition}\label{f_univalent}
The map $\phi(w)=w+\frac1{2w^2}$ is univalent in $\widehat\C\setminus\D.$
\end{proposition}

We define
$$\Omega:=\phi(\widehat\C\setminus\overline{\D}),\qquad T:=\Omega^c.$$
By Proposition~\ref{simp_conn_quad}, $\Omega$ is an unbounded quadrature domain with associated Schwarz reflection map $\sigma:=\phi\circ(1/\overline{z})\circ (\phi\vert_{\widehat{\C}\setminus\D})^{-1}:\overline{\Omega}\to\widehat{\C}$.
Since $\phi$ commutes with multiplication by $\omega$ (where $\omega=e^{\frac{2\pi i}{3}}$), it follows that $T$ is symmetric under rotation by $\frac{2\pi}{3}$. Moreover, as $\phi$ has simple critical points at $\omega^k$ (for $k=0,1,2$), $\partial T$ has three $\frac32$-cusp points (see Figure~\ref{deltoid_reflection_pic} (right)). It is also worth mentioning that $\partial T$ is the classical Euler's deltoid curve, that is the locus of a point on the circumference of a circle of radius $1/2$ as it rolls inside a circle of radius $3/2$. 

\begin{figure}[ht!]
\begin{center}
\includegraphics[width=0.4\linewidth]{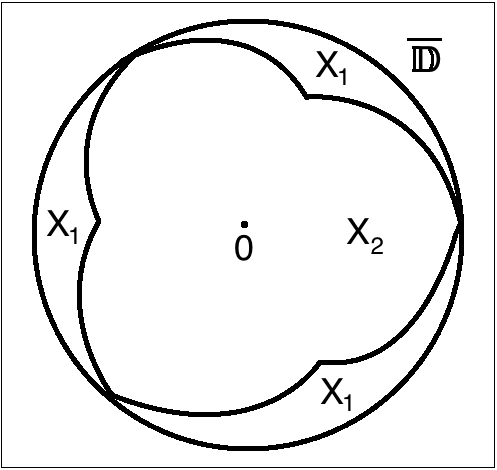}\quad \includegraphics[width=0.382\linewidth]{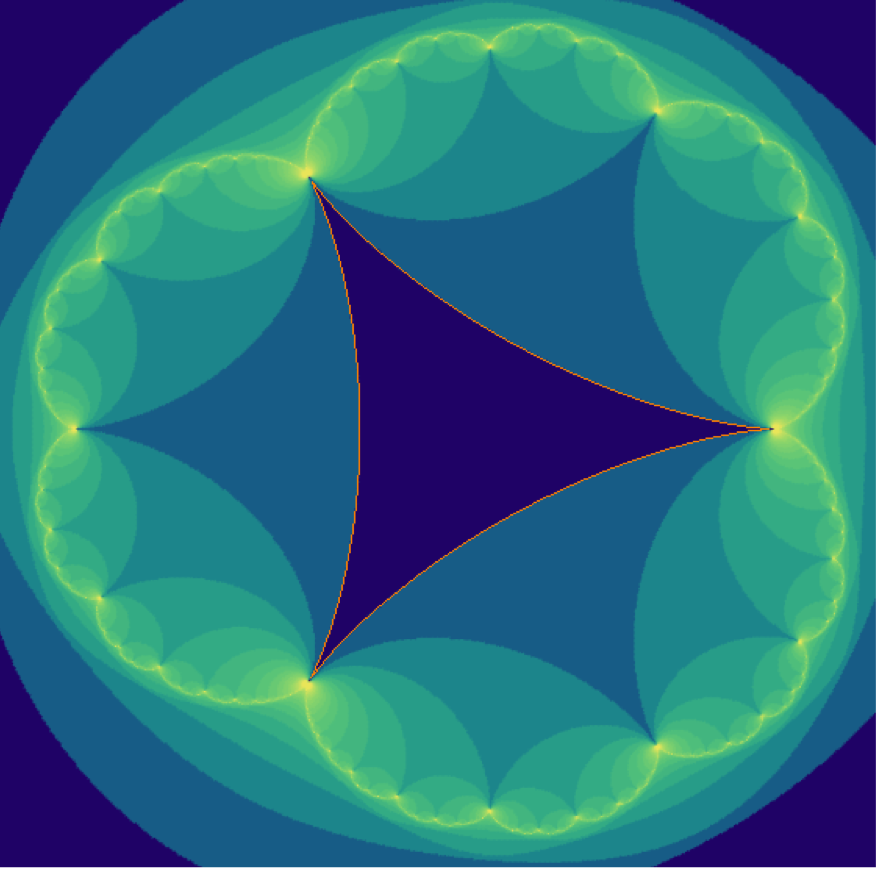}
\end{center}
\caption{Left: Each connected component of $X_1=\phi^{-1}(T^0)\cap\overline{\D}$ is mapped univalently onto $T^0$ by $\phi$. Moreover, $\phi$ is a two-to-one branched covering from the connected set $X_2$ onto $\Omega$, and the only ramification point is the origin. Right: Schwarz dynamics of the deltoid with the tiles of various ranks shaded.}
\label{deltoid_reflection_pic}
\end{figure}

\medskip \noindent Notation: $T^0$ is $T$ with the three cusp points removed.

\subsubsection{Schwarz reflection}

\begin{proposition}\label{deltoid_schwarz_critical}
The Schwarz reflection map $\sigma$ of $\Omega$ has a double pole at $\infty$, but no other critical points in $\Omega$.
\end{proposition}
\begin{proof}
This is a straightforward consequence of the definition of $\sigma$ and the fact that the only critical point of $\phi$ in $\D$ is at the origin.
\end{proof}

\subsubsection{Covering properties of $\sigma$}

\begin{proposition}\label{deltoid_schwarz_covering}
$\sigma:\sigma^{-1}(\Omega)\to\Omega$ is a proper branched $2$-cover branched only at $\infty$. On the other hand, $\sigma:\sigma^{-1}(T^0)\to T^0$ is a $3$-cover.
\end{proposition}

\begin{proof}
This follows from \cite[Lemma 4.1]{LM} that uses the extension of the Schwarz reflection map $\sigma$ to the Schottky double of $\Omega$ (see \cite{Gus}). Here is a more direct proof.

The degree $3$ rational map $\phi$ maps $\widehat\C\setminus\overline{\D}$ univalently onto $\Omega$. Note that the critical points of $\phi$ (all of which are simple) are $1, \omega, \omega^2,$ and $0$. They are mapped by $\phi$ to $\frac32, \frac32\omega, \frac32\omega^2,$ and $\infty$ respectively. In particular, $\Omega$ contains exactly one critical value of $\phi$, while $T^0$ does not contain any critical value of $\phi$. Let us set $X_1:=\phi^{-1}(T^0)\cap\overline{\D}$, and $X_2=\phi^{-1}(\Omega)\cap\overline{\D}$ (see Figure~\ref{deltoid_reflection_pic} (left)). It now readily follows that $\phi:X_1\rightarrow T^0$ is a proper covering of degree $3$, and $\phi:X_2\rightarrow \Omega$ is proper branched covering of degree $2$ branched only at $0$.

Finally, the definition of $\sigma$ implies that $\sigma^{-1}(T^0)=(\phi\circ\widetilde{\iota})(X_1)$, and $\sigma:(\phi\circ\widetilde{\iota})(X_1)\rightarrow T^0$ is a proper $3$-cover (where $\widetilde{\iota}$ is reflection in the unit circle). The same description also shows that $\sigma^{-1}(\Omega)=(\phi\circ\widetilde{\iota})(X_2)$, and $\sigma:(\phi\circ\widetilde{\iota})(X_2)\rightarrow\Omega$ is proper branched $2$-cover branched only at $\infty$.
\end{proof}

\begin{corollary}\label{deltoid_schwarz_iterate_covering}
The maps $\sigma^{\circ n}:\sigma^{-n}(\Omega)\to\Omega$ are proper branched covers of degree $2^n$ branched only at $\infty$.
\end{corollary}

\subsubsection{Tiling set} By definition,
$T^\infty :=\bigcup_{n\geq0} \sigma^{-n}(T^0).$ In Figure~\ref{deltoid_reflection_pic} (right), the bounded complementary component of the green curve is $T^\infty$. We will call it the \emph{tiling set}. We will call the components of $\sigma^{-n} T^0$ tiles of rank $n$. The next proposition directly follows from the construction of tiles.

\begin{proposition}\label{tiles_rank_n}
There are $3\cdot 2^{n-1}$ tiles of rank $n$. The union of tiles of rank $\le n$ is a ``polygon" with $3\cdot 2^n$ vertices (cusps).
\end{proposition}

\begin{proposition}\label{tiling_set_topology}
$T^\infty$ is a simply connected domain.
\end{proposition}
\begin{proof}
If $z\in T^\infty$ belongs to the interior of a tile, then it clearly belongs to $\Int{T^\infty}$. On the other hand, if $z\in T^\infty$ belongs to the boundary of a tile of rank $n$, then $z$ lies in $\Int{\left(\bigcup_{k=0}^{n+1} \sigma^{-k} T^0\right)}\subset\Int{T^\infty}$. Hence, $T^\infty$ is open.

The above argument also shows that $T^\infty$ is the increasing union of the connected, simply connected domains $\left\{\Int\left(\bigcup_{k=0}^{n} \sigma^{-k} T^0\right)\right\}_{n\geq 1}$. Hence, $T^\infty$ itself is connected and simply connected.
\end{proof}

Note that $\sigma:T^\infty\setminus T^0\to T^\infty$ is not a covering map as the degree of $\sigma~:~\sigma^{-1}(T^0)\to~T^0$ is three, while the degree of $\sigma:T^\infty\setminus\left(T^0\cup\sigma^{-1}(T^0)\right)\to T^\infty\setminus T^0$ is two. For this reason, some care should be taken to define inverse branches of the iterates of $\sigma$ on $T^\infty$.

For a subset $X$ of the plane, we denote by $N_\epsilon(X)$ the $\epsilon$-neighborhood of $X$. Let us fix some small $\epsilon>0$ and $K>1$ such that the set $$T^{\mathrm{hyp}}:=N_\epsilon\left(\overline{T^\infty}\right)\setminus\overline{N_{K\epsilon}\left(T\right)}$$ consists of three disjoint simply connected domains. The next result, which follows from the mapping properties of $\sigma$, tells us that $\sigma$ is hyperbolic near the boundary of the tiling set away from the cusp points.

\begin{proposition}\label{domain_hyperbolicity}
All inverse branches of $\sigma^{\circ n}$ (for $n\geq1$) are well-defined locally on $T^{\mathrm{hyp}}$. Moreover, $\sigma$ is hyperbolic on $T^{\mathrm{hyp}}$.
\end{proposition}
\begin{proof}
Note that $\overline{T^{\mathrm{hyp}}}\subset\Omega$. The first part of the proposition follows from the fact that $\sigma:\sigma^{-1}(\Omega)\to\Omega$ is a branched covering (see Proposition~\ref{deltoid_schwarz_covering}) and $T^{\mathrm{hyp}}$ does not intersect the post-critical set of $\sigma$.

For the second assertion, first recall that $\sigma(\phi(w))=\phi(\frac{1}{\overline{w}})$, $\vert w\vert>1$. Taking the $\overline\partial$-derivative of this relation yields 
\begin{equation}
\overline{\partial}\sigma\left(\phi(w)\right)\cdot\left(1-\frac{1}{\overline{w}^3}\right)=-\frac{1}{\overline{w}^2}+\overline{w},
\label{deltoid_schwarz_derivative}
\end{equation} 
for $\vert w\vert>1$. It follows from Relation~(\ref{deltoid_schwarz_derivative}) that 
\begin{equation}
\left\vert\overline{\partial}\sigma\left(\phi(w)\right)\right\vert=\vert w\vert,
\label{schwarz_multiplier}
\end{equation}
for $\vert w\vert>1$. Since $T^{\mathrm{hyp}}$ is compactly contained in $\Omega$, it follows that there exists $\lambda_0>1$ such that $\vert\overline{\partial}\sigma\vert>\lambda_0>1$ on $T^{\mathrm{hyp}}$.
\end{proof}

\subsubsection{Non-escaping set, basin of infinity}

The openness and complete invariance of the tiling set yields the following corollary.

\begin{corollary}\label{complement_tiling_set}
$\widehat\C\setminus T^\infty$ is a closed, completely invariant set.
\end{corollary}

Since every point in $T^\infty\setminus T^0$ escapes to $T^0$ under some iterate of $\sigma$, we can think of $T^\infty$ as the \emph{escaping set} of $\sigma$. On the other hand, $\widehat\C\setminus T^\infty$ consists of points that never land in $T^0$ and we call it the \emph{non-escaping set} of $\sigma$. 

By Proposition~\ref{deltoid_schwarz_critical}, $\sigma(\infty)=\infty$ and $(D\sigma)(\infty)=0$; i.e., $\infty$ is a super-attracting fixed point of $\sigma$ (see \cite[\S 9]{M1new} for a description of local dynamics of holomorphic germs near super-attracting fixed points, which also applies to the anti-holomorphic case). We denote the basin of attraction of $\infty$ by $A\equiv A(\infty)$. Clearly, $A\subset \widehat\C\setminus\overline{T^\infty}.$

\begin{remark}
Note that the tiling set (where points escape to $T^0$) of $\sigma$ plays the role of the basin of infinity (where points escape to infinity) of a polynomial. Moreover, the non-escaping set (where points never escape to $T^0$) of $\sigma$ is the analogue of the filled Julia set (where points never escape to infinity) of a polynomial.
\end{remark}

\begin{proposition}\label{basin_topology}
$A$ is a simply connected, completely invariant domain.
\end{proposition}
\begin{proof}
Since $A$ is the basin of attraction of a super-attracting fixed point, it is necessarily open and completely invariant. As $\sigma^{-1}(\infty)=\{\infty\}$ (by Proposition~\ref{deltoid_schwarz_covering}), $A$ is connected.

It remains to prove simple connectivity of $A$. Let $U\subset A$ be a small neighborhood of $\infty$. Clearly, $A$ is the increasing union of the domains $\left\{\sigma^{-k}(U)\right\}_{k\geq0}$. Since $\infty$ is the only critical point of $\sigma$, it follows from the Riemann-Hurwitz formula that each $\sigma^{-k}(U)$ is simply connected. Thus, $A$ is an increasing union of simply connected domains, and hence itself is such.
\end{proof}

\subsubsection{Singular points} We define $S:=\bigcup_{n\geq 0} \sigma^{-n}(T\setminus T^0),$
so $S$ is the collection of the cusp points of all tiles. It is clear, that $S\subset\partial T^\infty$. We also have

\begin{proposition}\label{singular_in_basin_boundary}
$S\subset \partial A$.
\end{proposition}
\begin{proof} 
Since $A$ is completely invariant under $\sigma$, so is $\partial A$. Thus, in light of the $2\pi/3$-rotational symmetry of the deltoid, it is enough to show that $\frac32$, a cusp point of $\Omega$, is in $\partial A$. For a real $x>1$ we have
(from $\sigma(\phi(w))=\phi(1/\overline w)$)
\begin{equation}\sigma\left(x+\frac1{2x^2}\right)=\frac1x+\frac{x^2}2>x+\frac1{2x^2}
\label{sigma_increasing}
\end{equation}
because at $x=1$ we have equality  in the last relation, and for the corresponding derivatives we have $-x^{-2}+x>1-x^{-3}$ for $x>1.$
Thus, the forward $\sigma$-orbit of any real $w>\frac32$ must converge to $\infty$ (otherwise, the orbit would converge to a fixed point of $\sigma$ in $(\frac{3}{2},+\infty)$, which contradicts Inequality~(\ref{sigma_increasing})). Hence, $(\frac32,+\infty)\subset A$, and we are done.
\end{proof}

\subsubsection{Main results}

\begin{theorem}\label{deltoid_dynamical_partition} 
We have
$$\partial T^\infty=\partial A=\overline{S},$$
and this set, which we denote by $\Gamma$, is a Jordan curve. Also,
$$\widehat {\mathbb C}=T^\infty\sqcup \Gamma\sqcup A.$$
\end{theorem}

\begin{lemma}\label{deltoid_lc}
$\partial T^\infty$ is locally connected.
\end{lemma}

We will give a detailed proof of Lemma~\ref{deltoid_lc} in Subsection~\ref{deltoid_local_conn}. We remark that this lemma is crucial and the techniques that will go into the proof will be used elsewhere in the paper too. The derivation of the main theorem from the main lemma is explained below.

\subsubsection{Proof of Theorem~\ref{deltoid_dynamical_partition}} Let $\psi^{\mathrm{in}}:\Pi\to T^0$ be the homeomorphic extension of a conformal isomorphism such that $\psi^{\mathrm{in}}(0)=0$, $\psi^{\mathrm{in}}(1)=\frac32$ (see Section~\ref{ideal_triangle} for the definition of the ideal triangle $\Pi$) .

Since $\sigma$ has no critical point in its tiling set $T^\infty$, the tiles of all rank of $\sigma$ map diffeomorphically onto $T^0$ under iterates of $\sigma$. Similarly, the tiles of the tessellation of $\D$ arising from the ideal triangle group $\mathcal{G}$ map diffeomorphically onto $\Pi$ under iterates of $\rho$. Furthermore, $\sigma$ and $\rho$ act as identity maps on $\partial T^0$ and $\partial \Pi$ respectively. This allows us to lift $\psi^{\mathrm{in}}$ to a conformal isomorphism from $\D$ (which is the union of all iterated preimages of $\Pi$ under $\rho$) onto $T^\infty$ (which is the union of all iterated preimages of $T^0$ under $\sigma$). Note that the trivial actions of $\sigma$ and $\rho$ on $\partial T^0$ and $\partial \Pi$ (respectively) ensure that the iterated lifts match on the boundaries of the tiles. By construction, the conformal map $\psi^{\mathrm{in}}$ conjugates $\rho:\D\setminus\Int{\Pi}\to\D$ to $\sigma:T^\infty\setminus \Int{T^0}\to T^\infty$. 
Since the cusp points of the ideal triangle group are dense in $\mathbb{T}$, we have $\overline{S}=\psi^{\mathrm{in}}(\mathbb{T})=\partial T^\infty,$
where we used local connectivity of $\partial T^\infty$.

By Proposition~\ref{singular_in_basin_boundary}, $S\subset \partial A$. Therefore
$\partial T^\infty=\overline{S}\subset \partial A.$
It remains to prove the opposite inclusion, $\partial A\subset \partial T^\infty.$ (Then we would have $\partial A=\partial T^\infty$ for two disjoint simply connected domains with a locally connected boundary, so the domains are Jordan.)
Let $a\in\partial A\setminus \partial T^\infty$. Then $a\not \in\overline{T^\infty}$, and there is an open set $U$ such that $a\in U\subset\C\setminus T^\infty$. All iterates of $\sigma$ are defined in $U$ and form a normal family (since they avoid $T$). It follows that $a\in A$, a contradiction. $\square$

\begin{remark}
One can also prove the statement that $\widehat{\mathbb C}~=~\overline{T^\infty}~\sqcup~A$ using more general arguments as will be done for Schwarz reflection maps in the circle-and-cardioid family.
\end{remark}

\begin{definition}[The limit set]\label{def_deltoid_limit_set}
The Jordan curve $\Gamma$, which is the common boundary of the tiling set $T^\infty$ and the basin of infinity $A$, is called the \emph{limit set} of $\sigma$.
\end{definition}

\subsection{Local connectivity of the boundary of the tiling set}\label{deltoid_local_conn}

\noindent Here we discuss the proof of Lemma~\ref{deltoid_lc}. We refer the reader to \cite[Expos{\'e} IX]{orsay}, \cite[\S 10]{M1new}, \cite[\S 23.7]{L6} for background on local dynamical properties of holomorphic maps near parabolic fixed points (i.e., fixed points with derivative equal to a root of unity), which will be extensively used in the following proof and elsewhere in the paper.

\subsubsection{Local dynamics near cusp points}\label{local_dyn_near_cusps}
The local dynamics of $\sigma$ near $\frac32$ and the other two cusps of $T$ are reminiscent of dynamics of parabolic germs. For $\epsilon>0$ small enough, let us denote
$$B:=B\left(\frac32,\epsilon\right),\qquad B^-:=B\cap \left\{\re(z)<\frac32\right\},\qquad B^+:=B\setminus B^-.$$
On the domain $B\cap\Omega$, we have the following Puiseux series expansion 
\begin{equation}
\sigma(\frac32+\delta)=\frac32+\overline{\delta}+k\overline{\delta}^{\frac32}+O(\overline{\delta}^2),
\label{puiseux_eqn}
\end{equation}
where $k$ is a positive constant, and the chosen branch of square root sends positive reals to positive reals. The Puiseux series expansion of $\sigma$ shows that $(\frac{3}{2},\frac32+\epsilon)$ is the unique repelling direction of $\sigma$ at $\frac32$ (compare with the proof of Proposition~\ref{singular_in_basin_boundary}). 

Moreover, one can conjugate the Puiseux series expansion given in Equation~\eqref{puiseux_eqn} by a change of coordinate $\kappa: w\mapsto \frac{k_1}{\sqrt{w-\frac32}}$ to obtain an asymptotic expansion of the form $\zeta\mapsto \overline{\zeta}+\frac12+O\left(\frac{1}{\overline{\zeta}}\right)$ on $\kappa(B\cap\Omega)$ (where $k_1<0$, and the branch of square root sends positive reals to positive reals). Note that $\kappa$ sends the unique repelling direction of $\sigma$ to the negative real axis near $\infty$, and the domain $\kappa(B\cap\Omega)$ subtends an angle $\pi$ at $\infty$. It follows from the above asymptotics that for any $\alpha\in(0,\pi/2)$, points with sufficiently large absolute value and lying between the boundary curves $\kappa\left(B\cap\partial\Omega\right)$ and the infinite rays $\kappa\left(\frac32+[0,\epsilon)e^{\pm i\alpha}\right)$ eventually escape to $\kappa(B\setminus\Omega)$. In the original coordinate, this means that sufficiently close to $\frac32$, the iterated preimages of the fundamental tile $T^0$ occupy a circular wedge with angle arbitrary close to $2\pi$. We record these observations below.

\begin{proposition}\label{dynamics_near_cusp}
If $\epsilon$ is sufficiently small, then $B^-\subset T^\infty,$ and $\sigma^{-n} B^+\to \left\{\frac32\right\}.$
\end{proposition}

\noindent (Here $\sigma^{-n}$ is the branch in $B^+$ which fixes $\frac32$; convergence is in the Hausdorff topology.)

\begin{proposition}\label{wedge_at_cusp}
For every $\theta\in[0,\pi/2)$, there exists $\epsilon>0$ such that $$W:=U\cup\omega U\cup\omega^2 U\subset T^\infty,\quad \mathrm{where}$$ 
 $$U:=\left\{\frac32+re^{i\theta'}: 0<r<\epsilon,\ \frac{\pi}{2}-\theta<\theta'<\frac{3\pi}{2}+\theta\right\},$$ and $\omega$ is a primitive third root of unity. 
\end{proposition}

The following proposition essentially follows from the observation that locally near the cusps, $\partial T^\infty$ is contained in the ``repelling petals" of the cusp points.

\begin{proposition}\label{temporary_near_cusp}
$\exists \epsilon>0$ such that if an orbit in $\partial T^\infty$ stays $\epsilon$-close to a cusp of $T$, then the orbit lands on this cusp.
\end{proposition}

\subsubsection{Puzzle pieces} Let us denote the Green function of $A$ with pole at infinity by $G$ ($G\equiv0$ on $A^c$). On $A$, the Green function $G$ can be explicitly written as
$$
G(z)= \lim_{n\to+\infty} 2^{-n} \log \vert\sigma^{\circ n}(z)\vert
$$ 
(cf. \cite[\S 9]{M1new}). It follows from the above formula of the Green function that for any $\rho>0$, $\sigma$ maps the equipotential (or level curve) $\{G=\rho\}$ to the equipotential $\{G=2\rho\}$. 

Let $\Delta_n$ be a tile of rank $n\ge1$. It is a ``triangle"; one of its sides is a side of a tile of rank $n-1$; we will call it (the side) the \emph{base} of $\Delta_n$. The vertices of the base are iterated preimages of the cusps of $T$. Since $\frac32$ is accessible from $\infty$ (see the proof of Proposition~\ref{singular_in_basin_boundary}), it follows that the vertices of the base of $\Delta_n$ are landing points of two external rays in $A(\infty)$. We denote by $P(\Delta_n)$ the closed region bounded by the base of $\Delta_n$, the two rays, and the equipotential $\{G=\frac{1}{2^n}\}$ (so $\Delta_n\subset P(\Delta_n)$), and call it a \emph{puzzle piece} of rank $n$.

\begin{remark}
A similar construction of puzzles in the context of polynomials with a parabolic fixed point can be found in \cite{Roe1,PR1}.
\end{remark}

The following propositions are immediate from the previous construction (see Figure~\ref{deltoid_reflection_pic} (right)), and are first indications of the usefulness of puzzle pieces.

\begin{proposition}\label{puzzle_connected}
For each $n\geq1$, the sets $\partial T^\infty\cap P(\Delta_n)$ are connected.
\end{proposition}

\begin{proposition}\label{puzzles_separate_impression}
The puzzle pieces separate the impressions of the internal rays of $T^\infty$ (images of $\mathcal{G}$-rays under $\psi^{\mathrm{in}}$) landing at points of $S$ from each other.
\end{proposition}

\begin{remark}
The internal rays of $T^\infty$ play the role of external dynamical rays of a polynomial with connected Julia set.
\end{remark}

\subsubsection{Local connectivity at ``radial" points.}

\begin{proposition}\label{radial_lc}
If $x\in \partial T^\infty\setminus S$, then $\partial T^\infty$ is locally connected at $x$.
\end{proposition}

\begin{proof}
Note first that if $\partial T^\infty$ is locally connected at $x$, then $\partial T^\infty$ is locally connected at $\sigma (x)$ and at every preimage of $x$.

Consider the orbit $x_n=\sigma^{\circ n}(x)$ of some $x\in \partial T^\infty\setminus S$. If dist$(x_n, T)\to 0$, then dist$(x_n, T\setminus T^0)\to 0$, and therefore the orbit converges to one of the cusps of $T$ (by continuity of $\sigma$). By Proposition~\ref{temporary_near_cusp}, this would imply that $x\in S$. But this contradicts our choice of $x$.

Thus there is a subsequence of $\{x_n\}$ at a positive distance from $T$. Let
$\zeta$ be a cluster point of this subsequence, so $\zeta\in \partial T^\infty$ is not a cusp of $T$. Thanks to Proposition~\ref{puzzles_separate_impression}, we can assume that $\zeta$ does not lie in the impression of the rays at angles $0, \frac13$, and $\frac23$ (possibly after replacing $\zeta$ by one of its iterated preimages, which is also a subsequential limit of $\{x_n\}$).

The above assumption on $\zeta$ allows us to choose a puzzle piece $P$ of sufficiently high rank such that
 $\zeta\in\Int{P}\subset P\subset\Int{P_1},$ where $P_1$ is a rank one puzzle piece. By construction, we have $x_n\in\Int{P}$ for infinitely many $n$'s. We now proceed to show that suitably chosen iterated preimages of $\Int{P}$ produce a basis of open, connected neighborhoods of $x$ in $\partial T^\infty$.

In order to achieve this, we use Proposition~\ref{domain_hyperbolicity} to define (for each $n$ with $x_n\in\Int{P}$) the inverse branches
 $\sigma^{-n}: \Int{P_1}\to \mathbb C, \ x_n\mapsto x.$ These inverse branches form a normal family on $\Int{P_1}$. We claim that $(\sigma^{-n})' \to 0 $ (along a subsequence) locally uniformly on $\Int{P_1}$, so uniformly on $P\cap \partial T^\infty$. Indeed, we have $\sigma^{-n_k}\to g$ on $\Int{P_1}$, and we need to show that $g'=0$ on some open set $V\subset P_1$. For $V$ we can take any small disc inside the rank one tile contained in $P_1$. Note that the preimages $V_k:=\sigma^{-n_k} V$ are disjoint (they sit inside tiles of different rank), and that by Koebe distortion,
 $${\rm area}(V_k)\asymp [{\rm diam}(V_k)]^2,$$
 so diam$(V_k)\to 0.$ This proves the claim.
 
 Thus, ${\rm diam}[ \sigma^{-n}(\Int{P}\cap\partial T^\infty)] \to 0$ along a subsequence. The sets $\sigma^{-n}(\Int{P}\cap\partial T^\infty)$ are open connected neighborhoods of $x$ in $\partial T^\infty$, and we have local connectivity at $x$.
\end{proof}

\subsubsection{Local connectivity at the cusp point $\frac32$}
To finish the proof of Lemma~\ref{deltoid_lc} it remains to check local connectivity at the point $\frac32$. Let $\Delta_n^\pm$ (for $n\geq 1$) be the two tiles of rank $n$ which have $\frac32$ as a common vertex. Let
$$\widetilde{P}_n:=\Int{(P(\Delta_n^+)\cup P(\Delta_n^-))}. $$
Note that for $n\geq 2$, the map $\sigma:\widetilde{P}_n\to \widetilde{P}_{n-1}$
is a bijection. In what follows, we will work with the corresponding inverse branch of $\sigma$ defined on $\widetilde{P}_1$.

We have $\sigma^{-1}(\widetilde{P}_1)\subset \widetilde{P}_1$, and $\frac32$ is the only common boundary point. The sets $(\partial T^\infty\cap\widetilde{P}_n)\cup\lbrace\frac32\rbrace$ are open (in $\partial T^\infty$) and connected. Moreover, their diameters go to zero by Proposition~\ref{dynamics_near_cusp} (and Denjoy-Wolff). Hence, they form a basis of open, connected neighborhoods of $\frac32$ (in the relative topology of $\partial T^\infty$).

This completes the proof of Lemma~\ref{deltoid_lc}.

\subsection{Conformal removability of the limit set}\label{conf_remo_sec}
By Lemma~\ref{deltoid_lc}, the Riemann uniformization $\psi^{\mathrm{in}}:\mathbb D \to T^\infty$ (constructed in Theorem~\ref{deltoid_dynamical_partition}) is continuous up to the boundary. We now prove a stronger version of this result.

\begin{definition}[John domain]
A domain $D\subset\C$ is called a John domain if there exists $c>0$ such that for any $z_0\in D$, there exists an arc $\gamma$ joining $z_0$ to some fixed reference point $w_0\in D$ satisfying 
\begin{equation}
\delta(z)\geq c\vert z-z_0\vert,\quad z\in \gamma,
\label{John_condition}
\end{equation}
where $\delta(z)$ stands for the Euclidean distance between $z$ and $\partial D$.
\end{definition}

Intuitively, Condition~(\ref{John_condition}) means that the arc $\gamma$ is ``protected" from the boundary of the domain $D$.

For us, the importance of John domains stems from the fact that boundaries of John domains are conformally removable  \cite[Corollary 1]{Jones} implying uniqueness in the mating theory (see Subsection~\ref{what_is_mating}). Moreover, the Riemann uniformization of a John domain extends as a H\"{o}lder continuous map to $\partial\D$. 

Let us now state a condition that is equivalent to the John condition for a simply connected domain $D$ (see \cite[p. 11]{CJY}, also compare \cite[Theorem, p. 263]{Gama} for equivalent formulations). For $z\in D$, we denote by $\gamma^z$ the part of the hyperbolic geodesic of $D$ passing through $z$ and a fixed base-point $w_0$ that runs from $z$ to $\partial D$. A simply connected domain $D$ is a John domain if and only if there exists $M>0$ such that for all $z\in D$, 
\begin{equation}
w\in\gamma^z,\ d_D(z,w)\geq M\quad \implies\quad \delta(w)\leq \frac12\delta(z),
\label{John_condition_equiv}
\end{equation}
where $d_D$ is the hyperbolic distance in $D$.

\begin{theorem}\label{tiling_set_John}
$T^\infty$ is a John domain. 
\end{theorem}

\begin{corollary}\label{tiling_set_removable}
$\partial T^\infty$ is conformally removable, and the Riemann uniformization $\psi^{\mathrm{in}}$ is H\"{o}lder continuous up to the boundary.
\end{corollary}

To prove Theorem~\ref{tiling_set_John}, we will sometimes use the conformal model $\rho\vert_{\D}$ of $\sigma\vert_{T^\infty}$. To study the dynamics near cusp points, it will be useful to consider the following wedges. In the dynamical plane of $\sigma$, we define for $\theta\in[0,\pi/2)$,
$$ 
W_\theta:=U_\theta\cup\omega U_\theta\cup\omega^2 U_\theta\subset T^\infty,\quad \mathrm{where}
$$ 
$$
U_\theta:=\left\{\frac32+re^{i\theta'}: 0<r<\frac{1}{100}, \frac{\pi}{2}-\theta<\theta'< \frac{3\pi}{2}+\theta\right\},
$$ 
and $\omega$ is a primitive third root of unity (compare Proposition~\ref{wedge_at_cusp}). Then, $W_0$ is a union of three disjoint sectors each of angle $\pi$. In the dynamical plane of $\rho$, the wedges are similarly defined with angles that are twice smaller. More precisely, for $\theta\in[0,\pi/4)$,
$$ 
\mathbf{W}_\theta:=\mathbf{U}_\theta\cup\omega \mathbf{U}_\theta\cup\omega^2 \mathbf{U}_\theta\subset \D,\quad \mathrm{where}
$$ 
$$
\mathbf{U}_\theta:=\left\{1+re^{i\theta'}: 0<r<\frac{1}{100}, \frac{3\pi}{4}-\theta<\theta'< \frac{5\pi}{4}+\theta\right\}.
$$ 
In this case, $\mathbf{W}_0$ is a union of three disjoint sectors each of angle $\pi/2$. If necessary, we modify $W_\theta$ and $\mathbf{W}_\theta$ in an obvious way so that the slices are invariant under $\sigma$ or $\rho$ (on the part where the maps are defined).

\subsubsection{Quasi-rays} Let $\pmb{\gamma}_1(\tau)$ and $\pmb{\gamma}_2(\tau)$ be paths in $\mathbb{D}$, $0\leq\tau<\infty$, and let $C>0$. We say that the paths \emph{shadow} each other (with a constant $C$) if $$\forall\tau>0,\quad d_{\mathbb{D}}(\pmb{\gamma}_1(\tau),\pmb{\gamma}_2(\tau))\leq C;$$ and we denote it by: $\pmb{\gamma}_1\approx\pmb{\gamma}_2$.

Let us fix a small number $\eta>0$ for the rest of the proof of Theorem~\ref{tiling_set_John}. We will now state a ``shadowing" lemma for the model map $\rho$. For a fixed $\theta_0\in[0,\pi/4)$, let $\mathbf{z}_0\in\D\setminus \mathbf{W}_{\theta_0}$ with $d(\mathbf{z}_0,\mathbb{T})<\eta$, and $\pmb{\gamma}^{\mathbf{z}_0}$ be the segment $[\mathbf{z}_0,\zeta_0)$, where $\zeta_0=\frac{\mathbf{z}_0}{\vert \mathbf{z}_0\vert}$. We parametrize $\pmb{\gamma}^{\mathbf{z}_0}$ so that $d_{\mathbb{D}}(\mathbf{z}_0,\pmb{\gamma}^{\mathbf{z}_0}(\tau))=\tau.$ We denote $\mathbf{z}_n:=\rho^{\circ n}(\mathbf{z}_0)$ (assuming that it is defined), $\zeta_n:= \rho^{\circ n}(\zeta_0)$, and $\mathbf{z}_n'=\vert \mathbf{z}_n\vert\zeta_n$ (see Figure~\ref{john_proof_pic}). Let $N$ be a positive integer such that $d(\mathbf{z}_k,\mathbb{T})<\eta$ and $\mathbf{z}_k\notin \mathbf{W}:=\mathbf{W}_{\theta_0}$, for $k=1,\cdots,N$.

\begin{lemma}\label{geodesic_distortion_rho}
For a fixed $\theta_0\in[0,\pi/4)$, we have that $\rho^{\circ N}\circ\pmb{\gamma}^{\mathbf{z}_0}\approx\pmb{\gamma}^{\mathbf{z}_N'},$ where the shadowing constant $C=C(\theta_0)$ is independent of $\mathbf{z}_0$ and $N$.
\end{lemma}

\begin{figure}[ht!]
\centering
\includegraphics[scale=0.28]{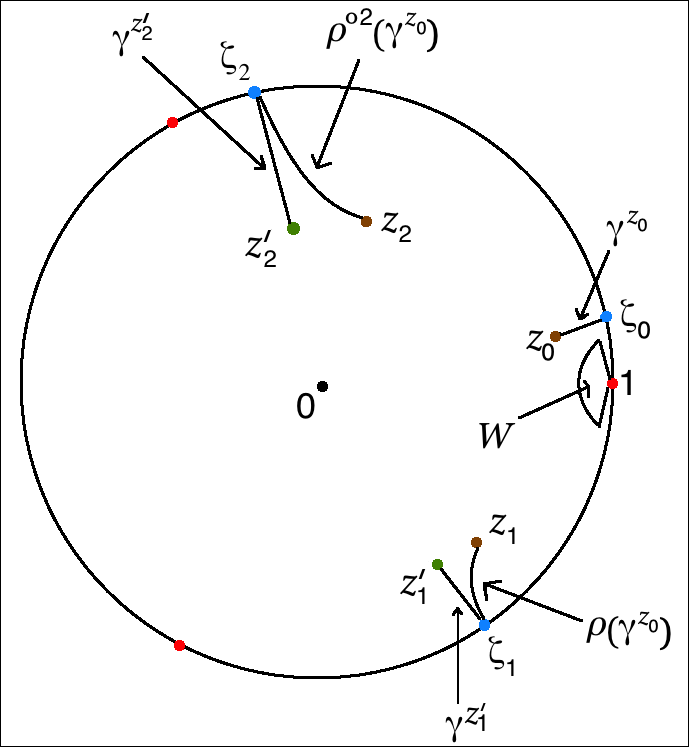}
\caption{The image of the geodesic ray $\pmb{\gamma}^{\mathbf{z}_0}$ under $\rho^{\circ N}$ shadows the actual geodesic ray $\pmb{\gamma}^{\mathbf{z}_N'}$.}
\label{john_proof_pic}
\end{figure}

\begin{proof}
Let $\widetilde{\iota}$ be reflection in the unit circle. Since $\rho$ fixes $\mathbb{T}$ as a set, we can extend it to the ``symmetrized" open set $\mathcal{V}:=\widehat{\C}\setminus\overline{\Pi\cup\widetilde{\iota}(\Pi)}$ by the Schwarz reflection principle. We denote the connected components of $\mathcal{V}$ by $\mathcal{V}_i$, $i\in\{1,2,3\}$.

As $\mathbf{z}_N\notin \mathbf{W}$ and $d(\mathbf{z}_N,\mathbb{T})<\eta$ for $\eta$ small enough, we have that $\mathbf{z}_N\notin\partial\Pi$, and hence, $\mathbf{z}_N\in\mathcal{V}_i$ for some $i$. Now, $\rho^{-N}(\mathcal{V}_i)$ is a disjoint union of finitely many round disks that are invariant under $\widetilde{\iota}$. Since $\mathbf{z}_0$ lies in one of these disks, its projection $\zeta_0$ on $\mathbb{T}$ also lies in the same disk. Hence in particular, $\mathbf{z}_N$ and $\zeta_N$ are in the same component $\mathcal{V}_i$. It follows that an inverse branch of $\rho^{\circ N}$, defined on the $\widetilde{\iota}$-invariant open set $\mathcal{V}_i$, carries $\mathbf{z}_N, \zeta_N$ to $\mathbf{z}_0, \zeta_0$ respectively. The facts that $\mathbf{z}_N\notin \mathbf{W}$ and $d(\mathbf{z}_N,\mathbb{T})<\eta$ now imply that
$$\mathcal{A}:=\mathcal{V}_i\setminus\overline{\left(\rho^{\circ N}\left(\pmb{\gamma}^{\mathbf{z}_0}\right)\cup\widetilde{\iota}\left(\rho^{\circ N}\left(\pmb{\gamma}^{\mathbf{z}_0}\right)\right)\right)}$$ is an annulus of definite modulus that depends only on $\theta_0$ and is independent of $\mathbf{z}_0$ and $N$. 

Note that $\rho^{-N}$ is an isometry from $\mathcal{V}_i$ onto $\rho^{-N}\left(\mathcal{V}_i\right)$ where both domains are equipped with the corresponding hyperbolic metrics. Hence, $\rho^{-N}(\mathcal{A})$ is an annulus of definite modulus (independent of $\mathbf{z}_0$ and $N$) surrounding the symmetrized radial line segment $\overline{\pmb{\gamma}^{\mathbf{z}_0}\cup\widetilde{\iota}(\pmb{\gamma}^{\mathbf{z}_0})}$ in $\rho^{-N}\left(\mathcal{V}_i\right)$.

The above lower bound on the modulus of the annulus $\rho^{-N}(\mathcal{A})$ implies that the line segment $\overline{\pmb{\gamma}^{\mathbf{z}_0}\cup\widetilde{\iota}(\pmb{\gamma}^{\mathbf{z}_0})}$ is uniformly bounded away from the boundary of $\rho^{-N}\left(\mathcal{V}_i\right)$, and hence, restricted to $\pmb{\gamma}^{\mathbf{z}_0}$, the hyperbolic metric of $\D$ and that of $\rho^{-N}\left(\mathcal{V}_i\cap\D\right)$ are both uniformly comparable to the reciprocal of the distance to $\mathbb{T}$. Thus, the hyperbolic metric of $\D$ and that of $\rho^{-N}\left(\mathcal{V}_i\cap\D\right)$ are uniformly comparable on $\pmb{\gamma}^{\mathbf{z}_0}$, and the hyperbolic metric of $\D$ and that of $\mathcal{V}_i\cap\D$ are uniformly comparable on $\rho^{\circ N}(\pmb{\gamma}^{\mathbf{z}_0})$. 

Therefore, $\rho^{\circ N}:\pmb{\gamma}^{\mathbf{z}_0}~\to~\rho^{\circ N}(\pmb{\gamma}^{\mathbf{z}_0})$ is a quasi-isometry with respect to the hyperbolic distance $d_{\D}$ of the unit disk with quasi-isometry constants independent of $\mathbf{z}_0$ and $N$. Therefore, $\rho^{\circ N}\circ\pmb{\gamma}^{\mathbf{z}_0}$ is shadowed by a hyperbolic geodesic of $\D$ with one end-point at $\zeta_N$. Since $\rho$ is expanding away from the third roots of unity, we conclude that the other end-point of this shadowing geodesic is bounded away from $\zeta_N$. It follows that $\rho^{\circ N}\circ\pmb{\gamma}^{\mathbf{z}_0}$ is shadowed by the geodesic arc $\pmb{\gamma}^{\mathbf{z}_N'}$, with a shadowing constant independent of $\mathbf{z}_0$ and $N$.
\end{proof}

\begin{remark}
1) For $\theta\in(\theta_0,\pi/4)$, the sector $\mathbf{W}_\theta$ contains $\mathbf{W}_{\theta_0}$. Hence the lower bound on the modulus of the annulus $\rho^{-N}(\mathcal{A})$ in the above proof remains valid for all $\theta\in(\theta_0,\pi/4)$. Thus, the same shadowing constant in $\approx$ remains valid if $\mathbf{W}_{\theta_0}$ is replaced by $\mathbf{W}_\theta$ for $\theta\in(\theta_0,\pi/4)$.

2) We will use the notation boldface $\mathbf{z}$ (respectively, $\pmb{\gamma}^{\mathbf{z}}$) for $\rho$-dynamics, and ordinary $z$ (respectively, $\gamma^z$) for $\sigma$-dynamics. 
\end{remark}

\subsubsection{Two uniform estimates}\label{unif_estimates} We will now state a couple of uniform estimates for the $\sigma$-dynamics which are needed for the proof of Theorem~\ref{tiling_set_John}. Recall that $0\in T^\infty$ corresponds to $0\in\D$ (for the model map $\rho$) under $\psi^{\mathrm{in}}$. The curves $\gamma^z$ in $T^\infty$ are now geodesic rays through the origin. We set $W:=W_{\theta_0}$ in the $\sigma$-dynamical plane, for some fixed $\theta_0\in[0,\pi/2)$.

The following assertion holds since $\partial T^\infty=\Gamma$ is a Jordan curve. 

\begin{lemma}\label{estimate_lc}
$\forall\epsilon>0$, $\exists M>0$ such that if $\delta(z')\geq\eta,\, w'\in\gamma^{z'},$ and $d_{T^\infty}(w',z')\geq M$, then
$\delta(w')\leq\epsilon\delta(z').$
\end{lemma}

The next estimate follows from the geometry of the limit set near a cusp point (see Proposition~\ref{wedge_at_cusp}, also compare \cite[p.~21, Inequality~6.8]{CJY}). In the statement below, $C$ stands for the shadowing constant from Lemma~\ref{geodesic_distortion_rho}, and $N_{\mathrm{hyp}}(W,C):=\{z\in T^\infty: d_{T^\infty}(z,W)\leq C\}$ is a hyperbolic neighborhood of $W$ of radius $C$.

\begin{lemma}\label{estimate_wedge_parabolic}
$\forall\epsilon>0$, $\exists M>0$ such that if $\sigma(z')\in N_{\mathrm{hyp}}(W,C),\ w'\in\gamma^{z'},$ and $d_{T^\infty}(w',z')\geq M$, then
$\delta(w')\leq\epsilon\delta(z')$.
\end{lemma}
\begin{proof}
Given $\theta_0$, pick $0<\alpha\ll \pi/2-\theta_0$, and shrink the domain $T^\infty$ slightly near $\frac32$ so that the radial lines at angles $\pm\alpha$ centered at $\frac32$ form the boundary of the new domain $\widehat{T}^\infty$ near $\frac32$.
It is readily checked that the defining condition~\ref{John_condition} of John property holds for any $z'$ near $\frac32$ in $\widehat{T}^\infty$ (for example, with base point $0$), and hence the equivalent condition~\ref{John_condition_equiv} is also satisfied by geodesic rays in $\widehat{T}^\infty$ starting at points close to $\frac32$. Since the radial lines at angles $\pm\alpha$ centered at $\frac32$ well approximate $\partial T^\infty$ near $\frac32$, Euclidean distances from the boundaries of $T^\infty$ and $\widehat{T}^\infty$ are comparable for points in $N_{\mathrm{hyp}}(W,C)$. Finally, since hyperbolic distances are larger in $\widehat{T}^\infty$, condition~\ref{John_condition_equiv} in $\widehat{T}^\infty$ gives the desired uniform estimate in $T^\infty$.
\end{proof}

\begin{remark}
We use the notation $z'$ and $w'$ (instead of just $z$ and $w$) to make it consistent with the notation in the next subsection.
\end{remark}

\subsubsection{Proof of Theorem~\ref{tiling_set_John}}\label{proof_John_sec} Recall that our goal is to prove the existence of $M>0$ such that for all $z\in T^\infty$, $$w\in\gamma^z,\ d_{T^\infty}(z,w)\geq M\quad \implies\quad \delta(w)\leq \frac12\delta(z).$$ For all $z\in T^\infty$ with $\delta(z)<\eta$, we define the \emph{stopping time} $N=N(z)$ as the largest integer $n$ such that $$\delta(z_i)<\eta,\ i=1,\cdots,n,\quad \mathrm{and}\quad z_{n-1}\notin W,$$ where $z_i=\sigma^{\circ i}(z)$. The stopping time $N$ is well-defined because of expansiveness of $\sigma$. So we either have $\delta(z_{N+1})\geq\eta$ (``Case 1") or/and $z_N\in W$ (``Case 2"). 
\medskip

\noindent\textbf{Case 1.} Let $z_N\notin W$. We first observe that $\delta(z_N)\asymp\eta$ with a constant in $\asymp$ independent of $z$. By Lemma~\ref{geodesic_distortion_rho}, we have $\sigma^{\circ N}\circ \gamma^z\approx\gamma^{z_N'}$ ($z_N'$ is defined via the model map $\rho$); in particular, $d_{T^\infty}(z_N,z_N')\leq C$. It then follows that $\delta(z_N')\asymp\eta$ and we can apply Lemma~\ref{estimate_lc} with $z':=z_N'$. Let $\epsilon$ be a small number to be specified at the end of the argument, and let $M$ be as in Lemma~\ref{estimate_lc}. Let $w\in\gamma^z$ satisfy $d_{T^\infty}(z,w)>M$. Define $w_N:=\sigma^{\circ N}(w)$, and $w_N'\in\gamma^{z_N'}$ by the condition $d_{T^\infty}(z,w)=d_{T^\infty}(z_N',w_N').$ Then we have $d_{T^\infty}(z_N',w_N')>M$ and (by Lemma~\ref{estimate_lc}) $\delta(w_N')\leq\epsilon\delta(z_N').$ Since the hyperbolic metric (of $T^\infty$) at a point is comparable to the reciprocal of its Euclidean distance to the boundary $\partial T^\infty=\Gamma$, and the curves $\sigma^{\circ N}\circ \gamma^z$ and $\gamma^{z_N'}$ shadow each other (with a constant $C$ independent of $z$), it follows that $ \delta(w_N)$ and $\delta(w_N')$ (respectively, $\delta(z_N)$ and $\delta(z_N')$) are uniformly comparable. So we get $\delta(w_N)\leq C\epsilon\delta(z_N).$ 

If $\theta$ is sufficiently close to $\pi/2$ (which does not affect the shadowing constant in Lemma~\ref{geodesic_distortion_rho}), then the Euclidean distance between $\sigma^{\circ N}(\gamma^z)$ and $W_0$ is uniformly comparable to $\eta$. A straightforward application of the Koebe distortion theorem (applied to a suitable inverse branch of $\sigma^{\circ N}$) now gives us $$\delta(w)\leq\widetilde{C}\epsilon\delta(z)= \frac12\delta(z)$$ ($\widetilde{C}$ is independent of $z$ and $\epsilon$ is defined by the last equation).
\medskip

\noindent\textbf{Case 2.} We now consider the case $z_N\in W$. By construction, we have $z_{N-1}\notin W$ and $\delta(z_{N-1})<\eta$. By Lemma~\ref{geodesic_distortion_rho}, we have $\sigma^{\circ (N-1)}(\gamma^z)\approx\gamma^{z_{N-1}'}$; and in particular, $d_{\D}(z_{N-1}, z_{N-1}')\leq C$. Hence, $z_{N-1}$ and $z_{N-1}'$ lie in the same component of $T^\infty\setminus T^0$ implying that $d_{\D}(\sigma(z_{N-1}), \sigma(z_{N-1}'))=d_{\D}(z_{N}, \sigma(z_{N-1}'))\leq C$. Since $z_N\in W$, it follows that $\sigma(z_{N-1}')$ lies in a $C$-hyperbolic neighborhood of $W$. We can now apply Lemma~\ref{estimate_wedge_parabolic} with $z'=z_{N-1}'$ to obtain $\delta(w_{N-1}')\leq\epsilon\delta(z_{N-1}').$ The rest of the argument is exactly the same as in Case 1.

This completes the proof of Theorem~\ref{tiling_set_John}.

\subsection{Reflection map as a mating}\label{sec_conformal_mating}

We will now show that the Schwarz reflection of the deltoid arises as the unique conformal mating of the anti-polynomial $\overline{z}^2$ and the reflection map $\rho$ coming from the ideal triangle group.
\vspace{1mm}

\subsubsection{Dynamical systems associated with $\sigma$}~

\noindent (a) {\it Action on the non-escaping set}
$\widehat\C\setminus T^\infty=\overline{A}=A\sqcup \Gamma.$

By Proposition~\ref{basin_topology}, the basin of infinity $A$ (of $\sigma$) is simply connected. Let $\psi^{\mathrm{out}}:\widehat\C\setminus\mathbb D\to \overline A$ be the Riemann uniformization such that $\infty\mapsto\infty$ and $1\mapsto\frac32$. 

\begin{proposition}\label{conjugacy_filled_Julia}
$\sigma: \overline{A}\to \overline{A}$ is (conformally) equivalent to the polynomial
$f_0:~\widehat\C\setminus\D\to\widehat\C\setminus\D,\ f_0(z)=\overline z^2.$ The conjugacy is given by $\psi^{\mathrm{out}}$.

\end{proposition}
\begin{proof}
The continuous extension of the Riemann uniformization $\psi^{\mathrm{out}}$ from $\widehat\C\setminus\D$ to $\overline{A}$ normalized as in the statement of the proposition conjugates $\sigma$ to a degree two anti-Blaschke product with a fixed critical point at $\infty$. Clearly, the conjugated map is $f_0$.
\end{proof}

\noindent (b) {\it Action on the tiling set $T^\infty$.} We will now show that the tessellation of $T^\infty$ by the preimages of the fundamental tile $T^0$ corresponds to the action of the {\it deltoid group} $\mathcal{G}_\Delta\subset{\rm Aut}(T^\infty)$ that is conformally equivalent to the ideal triangle group $\mathcal{G}$.

\begin{proposition} Let $T^1_j$ (j=1,2,3) be the three tiles of rank $1$, so
$$\sigma^{-1}T^0=T^1_1\sqcup T^1_2\sqcup T^1_3.$$ Then each map $\sigma: T_j^1\to T^0$ extends to a conformal automorphism
$\sigma_j: T^\infty\to T^\infty$. The deltoid group $\mathcal{G}_\Delta:=\left\langle \sigma_1,\sigma_2,\sigma_3\right\rangle\subset{\rm Aut}(T^\infty)$ is conformally conjugate to the ideal triangle group $\mathcal{G}$. In particular, $T^0$ is a fundamental domain of $\mathcal{G}_\Delta$.
\end{proposition}
\begin{proof}
The conformal map $\psi^{\mathrm{in}}:\mathbb D\to T^\infty$ (constructed in the proof of Theorem~\ref{deltoid_dynamical_partition}) conjugates $\rho$ to $\sigma$, and hence in particular, conjugates $\rho_j\vert_{\rho_j(\Pi)}$ to $\sigma\vert_{T_j^1}$. Thus, $\sigma_j:= \psi^{\mathrm{in}}\circ\rho_j\circ (\psi^{\mathrm{in}})^{-1}$ is a conformal automorphism of $T^\infty$ extending $\sigma: T_j^1\to T^0$. The desired conjugacy between $\mathcal{G}$ and $\mathcal{G}_\Delta$ is given by $\psi^{\mathrm{in}}$ .
\end{proof}

\noindent (c) {\it Action on the limit set}, $\sigma:\Gamma\to \Gamma$. This is the common boundary of the systems described in (a) and (b). On the one hand, $\sigma:\Gamma\to \Gamma$ is topologically equivalent to $f_0:~\mathcal{J}\to\mathcal{J}$, where $\mathcal{J}=\mathbb{T}$ is the Julia set of $f_0$. On the other hand, $\sigma:\Gamma\to \Gamma$ is topologically equivalent to the {\it Markov} map
$\rho: \Lambda\to \Lambda$ where $\Lambda=\mathbb{T}$ is the limit set of the ideal triangle group $\mathcal{G}$.

\begin{proposition}\label{external_maps_conjugacy}
There is a unique orientation-preserving homeomorphism $\mathcal{E}:~\mathbb{T}\to \mathbb{T}$, $1\mapsto 1$, that conjugates $\rho$ and $f_0$ on $\mathbb{T}$, i.e. $\mathcal{E}\circ \rho=f_0\circ\mathcal{E}.$
\end{proposition}
\begin{proof}
The existence of the conjugacy $\mathcal{E}$ was demonstrated at the end of Section~\ref{ideal_triangle}. The uniqueness of such a conjugacy follows from the fact that the only orientation-preserving automorphism of $\mathbb{T}$ commuting with $f_0:\mathbb{T}\to\mathbb{T}$ and fixing $1$ is the identity map.
\end{proof}

We can summarize this discussion as follows. The Schwarz reflection $\sigma$ is a (conformal) mating between the polynomial dynamics $f_0$ on $\widehat\C\setminus\D$ and the ideal triangle group $\mathcal{G}$ on $\D$ (or more accurately, the reflection map $\rho$ associated with the group). The precise meaning of the term ``mating"  is explained in Subsection~\ref{what_is_mating} below.

\subsubsection{Aside: the question mark function}\label{question_mark_subsec} The homeomorphism $\mathcal{E}:\mathbb{T}\to \mathbb{T}$ is a close relative of the classical Minkowski question mark function $\ciq:~ [0,1]\to [0,1].$ This is one of the earliest examples of singular strictly increasing homeomorphisms of the interval that possesses various interesting fractal and number-theoretic properties \cite{Min,Den38,Sal43,Con01}. The function $\ciq$ also admits a dynamical interpretation of being a conjugacy between the Farey map and the tent map, which makes it a useful tool in studying ergodic properties of continued fractions and the Farey map \cite{KS}.
Moreover, the function $\ciq$ played an important role in the Bullett-Penrose construction of algebraic correspondences that arise as matings of quadratic maps and the modular group \cite{BP}.  

In this subsection, we derive an explicit relation between our conjugacy $\mathcal{E}$ and the Minkowski question mark function. As we shall see below, the connection between $\ciq$ and $\mathcal{E}$ involves Farey fractions and the Farey map, which naturally appear in the study of geodesic flow on the modular surface \cite{Ser} as well in the study of self-similarity properties of the Mandelbrot set \cite{DLS}.  The relation between $\ciq$ and $\mathcal{E}$ allows one to transfer the number-theoretic properties of $\ciq$ to similar properties for $\mathcal{E}$, which are heavily exploited in \cite{LLMM19} to obtain distortion estimates and construct a David extension for the map $\mathcal{E}$.

One way to define the question mark function is to set $\ciq(\frac01)=0$ and $\ciq(\frac11)=1$ and then use the recursive formula
\begin{equation}
\ciq\left(\frac{p+r}{q+s}\right)=\frac12\left (\ciq\left(\frac{p}{q}\right)+\ciq\left(\frac{r}{s}\right)\right)
\label{question_mark_recursion}
\end{equation}
which  gives us the values of $\ciq$ on all rational numbers (Farey fractions) in $[0,1]$. In particular, $\ciq$ is an increasing homeomorphism of $[0,1]$ that sends the vertices of level $n$ of the Farey tree to the vertices of level $n$ of the dyadic tree (see Figure~\ref{question_mark_tree}).

\begin{figure}[ht!]
\centering
\includegraphics[scale=0.42]{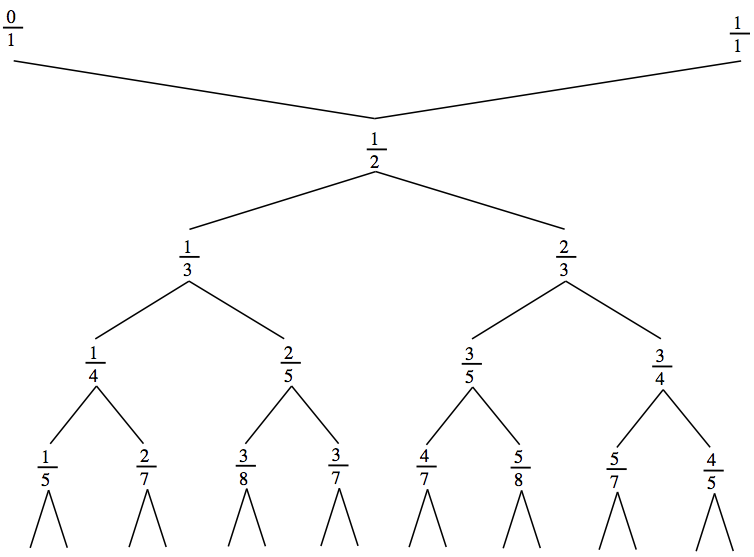} \includegraphics[scale=0.42]{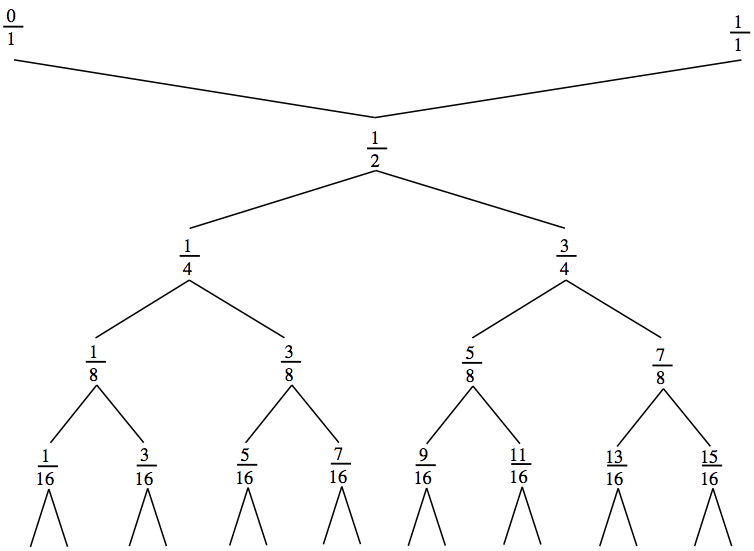} 
\caption{Left: The Farey tree. Right: The dyadic tree.}
\label{question_mark_tree}
\end{figure}

Formula~(\ref{question_mark_recursion}) can be described in terms of the action of an ideal triangle group  in the upper half-plane $\mathbb H$. Let $\Pi'$ be the ideal triangle with vertices $0, 1, \infty$. Then $\Pi'$ is a fundamental domain of the corresponding reflection group $\mathcal{G}'$. The corresponding tessellation of $\mathbb H$ consists of ideal triangles of rank $0,1, 2,$ etc. The rank zero triangle is $\Pi'$. There is exactly one rank $1$ triangle with vertices in $[0,1]$; the vertices are $\frac01, \frac12, \frac11$. The only new vertex is $\frac12$, which is the unique level $1$ vertex of the Farey tree. The map $\ciq$ sends this vertex to $\frac12$, which is the unique level $1$ vertex of the dyadic tree. There are exactly two rank $2$ triangles with vertices in $[0,1]$; the new vertices are $\frac{1}{3}$ and $\frac{2}{3}$, these are precisely the two vertices of level $2$ of the Farey tree. The map $\ciq$ sends these vertices to $\frac14$ and $\frac34$ (respectively), which are the two vertices of level $2$ of the dyadic tree, and so on.

Returning to our homeomorphism $\mathcal{E}:\mathbb{T}\to \mathbb{T}$ we note that the restriction of $\mathcal{E}$ to each of the three components of $\mathbb{T}\setminus \sqrt[3]1$ is given by an analogous construction which involves the tessellation of $\mathbb D$ with translates of $\Pi$ and the associated reflection group $\mathcal{G}$.

To describe a precise connection between the homeomorphism $\mathcal{E}$ and Minkowski question mark function $\ciq$, we need to define an analogue of the map $\rho$ in the upper half-plane model. Let $\phi$ be the M\"obius transformation $\phi$ that takes the unit disk onto the upper half-plane and such that $\phi(1)=0, \phi(e^{2\pi i/3})=1$, and $\phi(e^{4\pi i/3})=\infty$. The map $\phi$ conjugates $\rho$ to the following reflection map associated with the ideal triangle $\Pi'$ in $\mathbb{H}$
having its vertices at $0,1,\infty$:
$$
\theta:\R\cup\{\infty\}\to\R\cup\{\infty\},\qquad 
\theta(t)= \left\{\begin{array}{ll}
                    -t \qquad t\in\left[-\infty,0\right],  \\
                     \frac{t}{2t-1} \qquad t\in\left[0,1\right],\\
                     2-t \qquad t\in\left[1,+\infty\right].
                                          \end{array}\right. 
$$
By construction, $\theta$ maps $[0,\frac12]$ to $[-\infty,0]$ (respectively, $[\frac12,1]$ to $[1,+\infty]$). Composing $\theta$ with a conformal rotation (of $\mathbb{H}$) that brings $\theta([0,\frac12])=[-\infty,0]$ (respectively, $\theta([\frac12,1])=[1,+\infty]$) back to $[0,1]$ defines the orientation-reversing degree two map (which can be seen as the first return of $\theta$ to $[0,1)$)
 $$\tau:[0,1)\to[0,1),\qquad \displaystyle \tau(t)= \left\{\begin{array}{ll}
                     \frac{2t-1}{t-1}\ ({\rm mod}\ 1) \qquad t\in\left[0,\frac12\right),  \\
                     \frac{1-t}{t}\hspace{2.5mm} ({\rm mod}\ 1) \qquad t\in\left[\frac12,1\right).
                                          \end{array}\right. 
$$
The relation between Farey fractions and the action of the reflection group $\mathcal{G}'$, and the symmetry of the $\mathcal{G}'$-tessellation (of $\mathbb{H}$) under the above conformal rotations now imply that $\tau$ sends the vertices of level $n$ of the Farey tree to the vertices of level $(n-1)$. On the other hand, the anti-doubling map $m_{-2}:[0,1)\to[0,1)$ 
$$
m_{-2}(x)=\left\{\begin{array}{ll}
                     -2x+1\ ({\rm mod}\ 1) \qquad x\in\left[0,\frac12\right),  \\
                     -2x+2\ ({\rm mod}\ 1) \qquad x\in\left[\frac12,1\right),
                                          \end{array}\right.
$$
sends the vertices of level $n$ of the dyadic tree to the vertices of level $(n-1)$.

One can now verify that the map 
$$\mathfrak{a}:[0,1]\to[0,1]\quad x\mapsto \phi\left(\mathcal{E}^{-1}\left(e^{2\pi i\frac{x}{3}}\right)\right)
$$ 
is an increasing homeomorphism that conjugates $m_{-2}$ to $\tau$. 
Due to the conjugation property, $\mathfrak{a}$ carries the $n$-th preimages of $0$ under $m_{-2}$ (i.e., level $n$ of the dyadic tree) to the $n$-th preimages of $0$ under $\tau$ (i.e., level $n$ of the Farey tree). In light of the description of the map $\ciq$ given above, we conclude that $\ciq^{-1}$ and $\mathfrak{a}$ agree on a dense set of points, and hence 
$$
\ciq^{-1}(x)=\phi\left(\mathcal{E}^{-1}\left(e^{2\pi i\frac{x}{3}}\right)\right),\ \forall\ x\in [0,1].
$$
Thus, the Minkowski question mark function $\ciq$ is the restriction of the homeomorphism $\mathcal{E}$ to the arc $I:= [1,e^{2\pi i/3}]\subset\mathbb{T}$ written in appropriate coordinates. It conjugates the maps $\tau$ and $m_{-2}$ naturally induced by the triangle modular group and $\bar{z}^2$ on $I$.

\subsubsection{Uniqueness of mating}\label{what_is_mating} Here we will formalize the notion of mating and summarize the above discussion. 

\noindent   (a) {\it Setup.} We have two conformal dynamical systems
$$\rho: \overline{\mathbb D}\setminus\Int{\Pi}\to  \overline{\mathbb D},\quad \mathrm{and}\quad f_0: \widehat\C\setminus\mathbb D\to \widehat\C\setminus\mathbb D.$$

We also have a mating tool, the homeomorphism $\mathcal{E}: \mathbb{T}\to \mathbb{T}$ which conjugates $\rho$ on the limit set and $f_0$ on the Julia set, $\mathcal{E}\circ\rho=f_0\circ\mathcal{E}$.

\noindent   (b) {\it Topological mating.} Define $X~=~\overline{\mathbb D}~\vee_\mathcal{E}~(\widehat\C\setminus\mathbb{D})$, $Y=X\setminus\Int{\Pi},$
so $X$ is a topological sphere, and $Y$ is a closed Jordan disc in $X$. The well-defined topological map 
$\eta~\equiv~\rho~\vee_\mathcal{E}~f_0:~ Y\to X$
is  the topological mating between $\rho$ and $f_0$.

\noindent   (c) {\it Conformal mating.} The two Riemann uniformizations, $\psi^{\mathrm{in}}$ and $\psi^{\mathrm{out}}$,  glue together into a homeomorphism 
$$
H:  (X,Y) \rightarrow  (\widehat{\C},\overline{\Omega})
$$
which is conformal outside $H^{-1}(\Gamma)$ and which conjugates $\eta$ to $\sigma$.
It endows $X$ with a conformal structure compatible with the one on $X\setminus H^{-1}(\Gamma)$
that turns $\eta$ into an anti-holomorphic map conformally conjugate to $\sigma$. 
In this way, the mating $\eta$ provides us with a model for $\sigma$.   

\noindent   (d) {\it Uniqueness of conformal mating.} There is only one conformal structure on $X$ compatible with the standard structure on $X\setminus H^{-1}(\Gamma)$. Indeed, another structure would result in a non-conformal homeomorphism $\widehat \C \rightarrow \widehat\C$ which is conformal outside $\Gamma$, contradicting the conformal removability of $\Gamma$ (see Corollary~\ref{tiling_set_removable}). In this sense, $\sigma$ is the unique conformal mating of the map $\rho$ arising from the ideal triangle group and the anti-polynomial $\overline{z}^2$.   

Theorem~\ref{deltoid_mating_dynamical_partition}, that we announced in the introduction, is now obvious.

\begin{proof}[Proof of Theorem~\ref{deltoid_mating_dynamical_partition}]
The first part of the theorem follows from Theorem~\ref{deltoid_dynamical_partition} and Corollary~\ref{tiling_set_removable}.

The second statement is the content of Subsection~\ref{what_is_mating}.
\end{proof}

\begin{remark}
The conformal welding map $\mathcal{E}:\mathbb{T}\to\mathbb{T}$ appearing in the above mating construction is not a quasi-symmetry as it conjugates $\rho$ to $f_0$ mapping the parabolic fixed points of $\rho$ to the repelling fixed points of $f_0$. 
\end{remark}

\section{Iterated reflections with respect to the cardioid and a circle}\label{sec_main}

In this section, we carry out a detailed analysis of the dynamics and parameter space of Schwarz reflections with respect to a fixed cardioid and a family of circumscribing circles.
\vspace{1mm}

\noindent Notation.  By $B(a,r)$ (respectively, $\overline{B}(a,r)$), we will denote the open (respectively, closed) disk centered at $a$ with radius $r$. 

\subsection{Schwarz reflection with respect to the cardioid}\label{sec_cardioid}

The simplest examples of quadrature domains are interior and exterior disks $B(a,r)$ and $\overline{B}(a,r)^c$ respectively. As quadrature domains, their orders are $1$ and $0$ (respectively). In the rest of this section, we will focus on a particular quadrature domain of order $2$ that is of interest to us. 

\subsubsection{Cardioid as a quadrature domain}\label{cardioid_qd_sec} The principal hyperbolic component $\heartsuit$ of the Mandelbrot set (i.e. the unique hyperbolic component of period one) has a Riemann uniformization 
$$\phi:\D\to\heartsuit,\quad \phi(\lambda)=\lambda/2-\lambda^2/4$$ 
(see \cite[\S 2.1]{LLMM2} for details). Its boundary $\partial\heartsuit$ is the cardioid $$\phi(\partial\D)=\lbrace e^{i\theta}/2-e^{2i\theta}/4:\theta\in[0,2\pi)\rbrace.$$ Dynamically, this means that the quadratic map $z^2+w$ has a fixed point of multiplier $\lambda\in\D$ precisely when $w=\lambda/2-\lambda^2/4$ (here, \emph{multiplier} stands for the the $z$-derivative of $z^2+w$ at a fixed point). Since this Riemann uniformization is rational, it follows from Proposition~\ref{simp_conn_quad} that $\heartsuit$ is a quadrature domain. The quadrature function of $\heartsuit$ is $(\frac{3}{8z}-\frac{1}{16z^2})$. Hence, $\heartsuit$ is a quadrature domain of order $2$ with a unique node at $0$. 

\subsubsection{Schwarz reflection of $\heartsuit$}\label{schwarz_ref_cardioid_sec} Thanks to the commutative diagram in Figure~\ref{comm_diag_schwarz}, we have an explicit description of the Schwarz reflection map $\sigma:\overline{\heartsuit}\to\widehat{\C}$. Indeed, the commutative diagram implies that:

\begin{equation}
\sigma(\phi(\lambda))=\phi(1/\overline{\lambda}),\quad \mathrm{i.e.}\quad \sigma(\lambda/2-\lambda^2/4)=(2\overline{\lambda}-1)/4\overline{\lambda}^2
\label{schwarz_cardioid}
\end{equation}
for each $\lambda\in\overline{\D}$.

As $\phi$ is a two-to-one branched cover of $\widehat{\C}$, the commutative diagram~\eqref{schwarz_cardioid} shows that $\sigma(\overline{\heartsuit})=\widehat{\C}$. Since the only critical point of $\phi$ outside $\overline{\D}$ is at $\infty$, it follows that $0$ is the only critical point of $\sigma$ in $\heartsuit$. Moreover, some interesting functional values can be directly computed from this formula; e.g. $\sigma(0)=\infty,\ \sigma(3/16)=0,\ \sigma(5/36)=-3/4$.

Since $\heartsuit$ parametrizes all quadratic polynomials $z^2+w$ with an attracting fixed point of multiplier $\lambda$, it will be useful to understand the Schwarz reflection of the cardioid in terms of its action on multipliers of quadratic polynomials. A simple computation using Relation~(\ref{schwarz_cardioid}) shows that if the multiplier of the non-repelling fixed point of $z^2+w$ (where $w\in\overline{\heartsuit}\setminus\{0\}$) is $\lambda$, then the multipliers of the fixed points of $z^2+\sigma(w)$ are $2-1/\overline{\lambda}$ and $1/\overline{\lambda}$.

Since a quadratic polynomial $z^2+w$ is uniquely determined by any of its fixed point multipliers, the above discussion gives the following description of $\sigma$ on $\heartsuit\setminus\{0\}$.

\begin{proposition}[Multiplier description of Schwarz reflection]\label{multiplier_reflection}
Let $w=\phi(\lambda)\in\heartsuit\setminus\{0\}$ for some $\lambda\in\D^*$ (i.e., $z^2+w$ has an attracting fixed point of multiplier $\lambda$). Then, $\sigma(w)=c$ if and only if the quadratic polynomial $z^2+c$ has a fixed point of multiplier $1/\overline{\lambda}$.
\end{proposition}

\subsubsection{Covering properties of $\sigma$}\label{mapping_sigma_sec} This description allows us to conveniently study some mapping properties of the map $\sigma$.

\begin{proposition}\label{small_cardioid}
$\sigma^{-1}(\heartsuit)=\phi(B(\frac{2}{3}, \frac{1}{3}))$, and $\sigma:\sigma^{-1}(\heartsuit)\to\heartsuit$ is an anti-holomorphic isomorphism. Moreover, $\sigma:\partial \sigma^{-1}(\heartsuit)\to\partial\heartsuit$ is an orientation-reversing homeomorphism.
\end{proposition}
\begin{proof}
First note that $\sigma$ fixed $\partial\heartsuit$ pointwise. Now let $w\in\heartsuit$ and $\lambda$ be the attracting multiplier of $z^2+w$; i.e. $w=\phi(\lambda)$. Since $\lambda<1$, we have that $\vert 1/\overline{\lambda}\vert>1$. Therefore, $\sigma(w)\in\heartsuit$ if and only if $\vert 2-1/\overline{\lambda}\vert<1$. A simple computation shows that this happens if and only if $\vert \lambda-\frac{2}{3}\vert<\frac{1}{3}$; i.e. $\lambda\in B(2/3,1/3)$. This proves that $\sigma(w)\in\heartsuit\Leftrightarrow w\in \phi(B(\frac{2}{3}, \frac{1}{3}))$. Hence, $\sigma^{-1}(\heartsuit)=\phi(B(2/3, 1/3))$.

The above discussion also shows that the (inverse of the) Riemann uniformization $\phi$ of $\heartsuit$ conjugates $\sigma:\sigma^{-1}(\heartsuit)\to\heartsuit$ to the map $$g: B(2/3,1/3) \to \D,\ g(\lambda)= 2-1/\overline{\lambda}.$$ Since $g$ is an anti-holomorphic isomorphism, we conclude that $\sigma:\sigma^{-1}(\heartsuit)\to\heartsuit$ is such as well.

Similarly, the inverse of the homeomorphic boundary extension of the Riemann uniformization $\phi$ of $\heartsuit$ conjugates $\sigma:\partial\sigma^{-1}(\heartsuit)\to\partial\heartsuit$ to the map $g: \partial B(2/3,1/3) \to \partial\D$. Since $g$ is an orientation-reversing homeomorphism, the same is true for $\sigma:\partial\sigma^{-1}(\heartsuit)\to\partial\heartsuit$.
\end{proof}

\begin{remark}
$\sigma^{-1}(\heartsuit)$ is a smaller cardioid with its cusp at $\frac{1}{4}$ and vertex at $\frac{5}{36}$. Also, $\sigma^{-1}(\heartsuit)\cap\R=(\frac{5}{36},\frac{1}{4})$.
\end{remark}

The following behavior of $\sigma$ outside of $\sigma^{-1}(\heartsuit)$ can be directly deduced from Proposition~\ref{multiplier_reflection}.

\begin{proposition}
$\sigma:\heartsuit\setminus\overline{\sigma^{-1}(\heartsuit)}\to\widehat{\C}\setminus\overline{\heartsuit}$ is a two-to-one branched covering branched at $0$.
\end{proposition}

\begin{corollary}\label{fiber_sum}
For two distinct points $w_1=\phi(\lambda_1), w_2=\phi(\lambda_2)\in\heartsuit\setminus\left(\overline{\sigma^{-1}(\heartsuit)}\cup\{0\}\right)$, we have $\sigma(w_1)=\sigma(w_2)$ if and only if $1/\overline{\lambda_1}+1/\overline{\lambda_2}=2$.
\end{corollary}
\begin{proof}
This follows directly from the multiplier description of $\sigma$ given in Proposition~\ref{multiplier_reflection} and the fact that the sum of the multipliers of the two fixed points of a quadratic polynomial is $2$.
\end{proof}

\subsubsection{Iteration of $\sigma$}\label{cardioid_ref_iteration_sec}

\begin{corollary}[Iterated preimages of $\heartsuit$]\label{iterated_pre_image_cardioid}
$\sigma^{-n}(\heartsuit)=\phi(B(\frac{2n}{2n+1},\frac{1}{2n+1}))$, and $\sigma^{\circ n}:\sigma^{-n}(\heartsuit)\to\heartsuit$ is a univalent map. In particular, $\mathrm{diam}(\sigma^{-n}(\heartsuit))\rightarrow 0$, as $n\to+\infty$.
\end{corollary}
\begin{proof}
This is simply an extension of the arguments used in the proof of Proposition~\ref{small_cardioid}. Iterating the map $\sigma$ is equivalent to iterating the univalent map $g: \lambda\mapsto 2-1/\overline{\lambda}$ (as long as $w\in\sigma^{-1}(\heartsuit)$ or equivalently $\lambda\in g^{-1}(\D)$). As $(\iota\circ g)^{\circ n}(\lambda)=\frac{(n+1)\lambda-n}{n\lambda-(n-1)}$, it follows that $g^{-n}(\D)=B(\frac{2n}{2n+1},\frac{1}{2n+1})$. Hence, $\sigma^{-n}(\heartsuit)=\phi(B(\frac{2n}{2n+1},\frac{1}{2n+1}))$. The other statements readily follow.
\end{proof}

\subsubsection{An explicit formula for $\sigma$}\label{real_formula_sec} Recall that Relation~(\ref{schwarz_cardioid}) gives an implicit formula for $\sigma$. Choosing the branch of square root which sends positive reals to positive reals, we have the following explicit formula for $\sigma$: 
\begin{equation}
\sigma(w)=\frac{1-2\sqrt{1-4\overline{w}}}{4(1-\sqrt{1-4\overline{w}})^2},\quad \mathrm{for\ all}\ w\in\overline{\heartsuit}.  
\label{real_formula}
\end{equation}

\subsubsection{A characterization of $\heartsuit$ as a quadrature domain}\label{cardioid_char_sec}

\begin{proposition}[Characterization of cardioids and disks as quadrature domains]\label{cardioid_char_prop}
1) Let $\Omega\ni \infty$ be a simply connected quadrature domain with associated Schwarz reflection map $\sigma_\Omega$. If $\sigma_\Omega:\sigma_\Omega^{-1}(\overline{\Omega}^c)\to\overline{\Omega}^c$ is univalent, then $\Omega$ is an exterior disk.

2) Let $\Omega\ni 0$ be a bounded simply connected quadrature domain such that $\partial\Omega$ has a cusp at $\frac14$. Moreover, suppose that the Schwarz reflection map $\sigma_\Omega$ of $\Omega$ has a double pole at $0$, and $\sigma_\Omega:\sigma_\Omega^{-1}(\Omega)\to\Omega$ is univalent. Then, $\Omega=\heartsuit$.
\end{proposition}
\begin{proof}
1) By Proposition~\ref{simp_conn_quad}, there exists a rational map $\phi_\Omega$ on $\widehat{\C}$ such that $\phi_\Omega:\D\to\Omega$ is a Riemann uniformization of $\Omega$. Note that by Figure~\ref{comm_diag_schwarz}, the degree of $\sigma_\Omega:\sigma_\Omega^{-1}(\overline{\Omega}^c)\to\overline{\Omega}^c$ is equal to the number of preimages in $\widehat{\C}$ of a generic point in $\overline{\Omega}^c$ under $\phi_\Omega$. Our assumption now implies that $\mathrm{deg}\ \phi_\Omega=1$; i.e. $\phi_\Omega$ is a M{\"o}bius map. The conclusion follows.

2) As in the previous part, let $\phi_\Omega$ be a rational map on $\widehat{\C}$ such that $\phi_\Omega:\D\to\Omega$ is a Riemann uniformization of $\Omega$. We can assume that $\phi_\Omega(0)=0$, and $\phi_\Omega(1)=\frac14$. Since $\phi_\Omega(\partial\D)=\partial\Omega$ has a cusp at $\frac14$, it follows that $\phi_\Omega'(1)=0$.

Since the Schwarz reflection map $\sigma_\Omega$ of $\Omega$ has a double pole at $0$, the commutative diagram in Figure~\ref{comm_diag_schwarz} readily yields that $\phi_\Omega(\infty)=\infty$, and $D\phi_\Omega(\infty)=0$. 

The same commutative diagram also implies that the degree of $\sigma_\Omega:\sigma_\Omega^{-1}(\Omega)\to\Omega$ is equal to the number of preimages in $\widehat{\C}\setminus\overline{\D}$ of a generic point in $\Omega$ under $\phi_\Omega$. Hence, our assumption that $\sigma_\Omega:\sigma_\Omega^{-1}(\Omega)\to\Omega$ is univalent translates to the fact that $\mathrm{deg}\ \phi_\Omega=2$. It follows that $\phi_\Omega$ is a quadratic polynomial. Since $\phi_\Omega(0)=0$, $\phi_\Omega(1)=\frac14$, and $\phi_\Omega'(1)=0$, we conclude that $\phi_\Omega(\lambda)=\frac{\lambda}{2}-\frac{\lambda^2}{4}$. Therefore, $\Omega=\heartsuit$.
\end{proof}

\begin{remark}
The simple connectivity assumption in Proposition~\ref{cardioid_char_prop} can be dropped (although we will not need this improvement). However, one needs to employ more involved arguments to deduce the results without this assumption (see \cite{Gus} for the cardioid case, and \cite{Ep2} for the case of disks).
\end{remark}

\subsection{The dynamical and parameter planes}\label{dynamical_plane}

Let $\heartsuit$ be the principal hyperbolic component of the Mandelbrot set. For any $a\in\C$, let $B(a,r_a)$ be the smallest disk containing $\heartsuit$ and centered at $a$; i.e. $\partial B(a,r_a)$ is the circumcircle to $\heartsuit$.

\begin{proposition}\label{circle_cardioid_touching}
1) For $a\in(-\infty,-\frac{1}{12})$, the circumcircle $\partial B(a,r_a)$ touches $\partial\heartsuit$ at exactly two points. On the other hand, for any $a\in\C\setminus (-\infty,-\frac{1}{12})$, the circumcircle $B(a,r_a)$ touches $\partial\heartsuit$ at exactly one point.

2) For $a\neq-\frac{1}{12}$, the circle $\partial B(a,r_a)$ has a contact of order one with $\partial\heartsuit$. On the other hand, $\partial B(a,r_a)$ has a contact of order three with $\partial\heartsuit$ if $a=-\frac{1}{12}$.
\end{proposition}
\begin{proof}
1) Note that the cardioid $\partial\heartsuit$ and its circumcircle $\partial B(a,r_a)$ can have at most two points of tangency. 
Since the real line is the only axis of symmetry of $\heartsuit$, it follows that $\partial B(a,r_a)$ can touch $\partial\heartsuit$ at two points only if $a$ is real. Moreover, when $a$ is real, $\partial\heartsuit\cap\partial B(a,r_a)$ is either the singleton $\{-\frac34\}$ or consists of a pair of complex conjugate points.

It is clear that for $a\leq-\frac{3}{4}$, the circumcircle $\partial B(a,r_a)$ touches $\partial\heartsuit$ at a pair of complex conjugate points. 

Now let $-\frac34<a<-\frac{1}{12}$. For such values of $a$, the curvature of the circle $\partial B(a, a+\frac34)$ is larger than $\frac32$. But the curvature of $\partial\heartsuit$ at $-\frac34$ is $\frac32$ (in other words, the boundary of the cardioid curves less than the circle does at $-\frac34$). Since $\partial B(a, a+\frac34)$ is real-symmetric, we have that $B(a, a+\frac34)$ is locally contained in $\heartsuit$ near $-\frac34$. Hence, $\partial B(a, a+\frac34)$ is not a circumcircle to $\partial\heartsuit$. It now follows that $\partial\heartsuit\cap\partial B(a,r_a)$ consists of a pair of complex conjugate points (see Figure~\ref{real_slit_double}).

Finally, let $a\geq-\frac{1}{12}$. Note that the curvature of the circle $\partial B(a, a+\frac34)$ is at most $\frac32$, while the curvature of $\partial\heartsuit$ is at least $\frac32$ at each point (in other words, the boundary of the cardioid curves more than the circle everywhere). Since $\partial B(a, a+\frac34)$ is real-symmetric, it follows that $B(a, a+\frac34)$ contains $\heartsuit$. Hence, $r_a=a+\frac34$, and $\partial B(a, a+\frac34)$ is the circumcircle to $\heartsuit$ centered at $a$. In particular, $\partial\heartsuit\cap\partial B(a, a+\frac34)=\{-\frac34\}$.

2) The evolute (locus of centers of curvature) of $\partial\heartsuit$ is a cardioid $\frac{1}{3}$-rd the size of $\partial\heartsuit$ with a cusp at $-\frac{1}{12}$. For any $a$ not on the evolute, the circle $\partial B(a,r_a)$ is, by definition, not an osculating circle to $\partial\heartsuit$ (i.e. $\partial B(a,r_a)$ has a simple tangency with $\partial\heartsuit$). 

Now suppose that $a\neq-\frac{1}{12}$ is a point on the evolute of $\heartsuit$. More precisely, let $a\neq-\frac{1}{12}$ be the center of curvature of $\partial\heartsuit$ at some point $\alpha_a'$. A simple computation now shows that the radius of curvature of $\partial\heartsuit$ at $\alpha_a'$ is strictly smaller than the distance between $a$ and $\overline{\alpha_a'}\in\partial\heartsuit$; i.e. $\vert a-\alpha_a'\vert<\vert a-\overline{\alpha_a'}\vert$. It follows that the osculating circle to $\partial\heartsuit$ centered at $a$ does not circumscribe $\heartsuit$; in other words, $\partial B(a,r_a)$ is not an osculating circle to $\partial\heartsuit$ for these parameters $a$ as well.

Finally, for $a=-\frac{1}{12}$, the circle $\partial B(a,r_a)$ is indeed the osculating circle of $\partial\heartsuit$ at $-\frac{3}{4}$. Moreover, as $-\frac{3}{4}$ is a vertex of $\partial\heartsuit$ (i.e. the curvature of $\partial\heartsuit$ has a vanishing derivative at $-\frac{3}{4}$), the osculating circle of $\partial\heartsuit$ at $-\frac{3}{4}$ (which is centered at $-\frac{1}{12}$) has a third order tangency with $\partial\heartsuit$. Therefore, for all $a\in\C\setminus (-\infty,-\frac{1}{12}]$, the circle $\partial B(a,r_a)$ touches $\partial\heartsuit$ at exactly one point and has a contact of order one with $\partial\heartsuit$ at that point, while for $a=-\frac{1}{12}$, the circumcircle $\partial B(a,r_a)$ touches $\partial\heartsuit$ at exactly one point and has a contact of order three with $\partial\heartsuit$.
\end{proof}

\begin{figure}[ht!]
\centering
\includegraphics[width=0.32\linewidth]{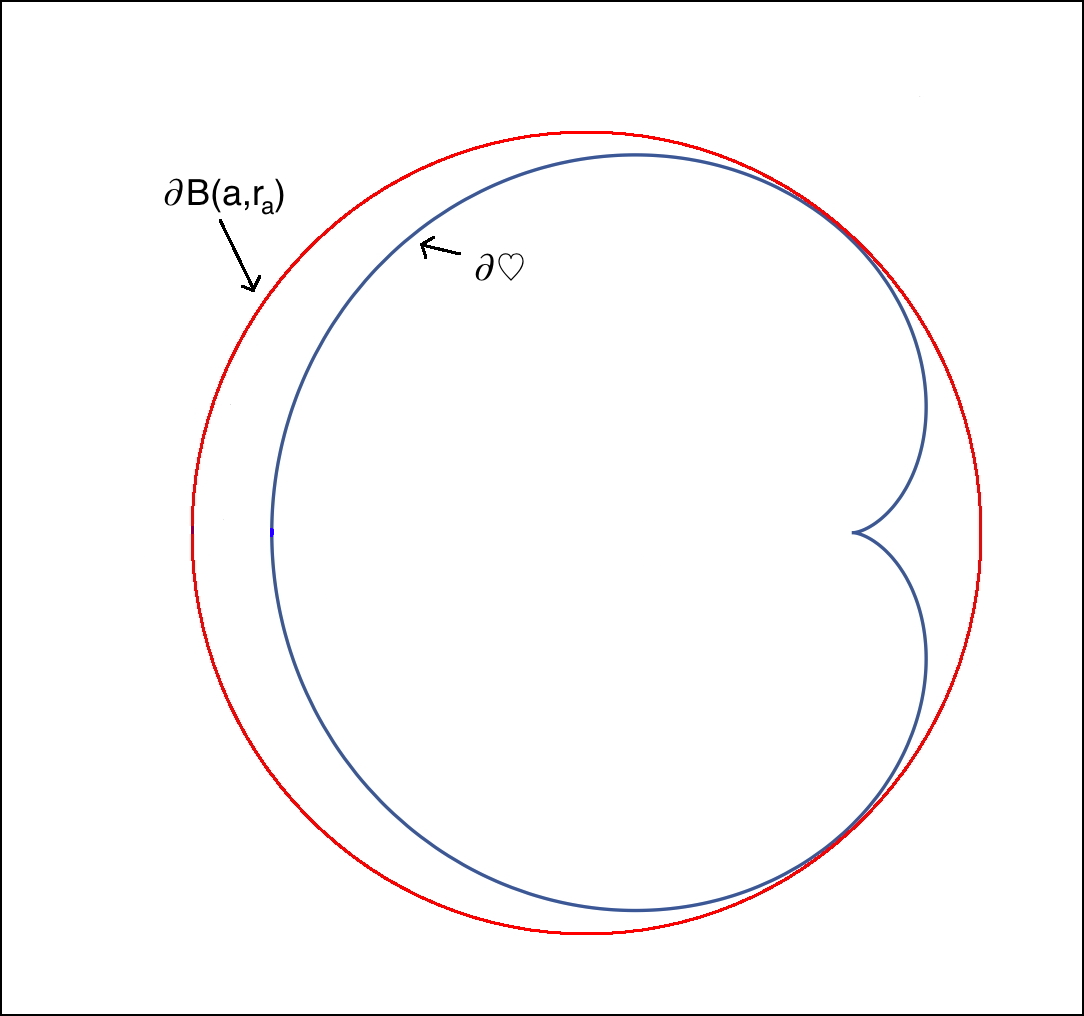}\ \includegraphics[width=0.48\linewidth]{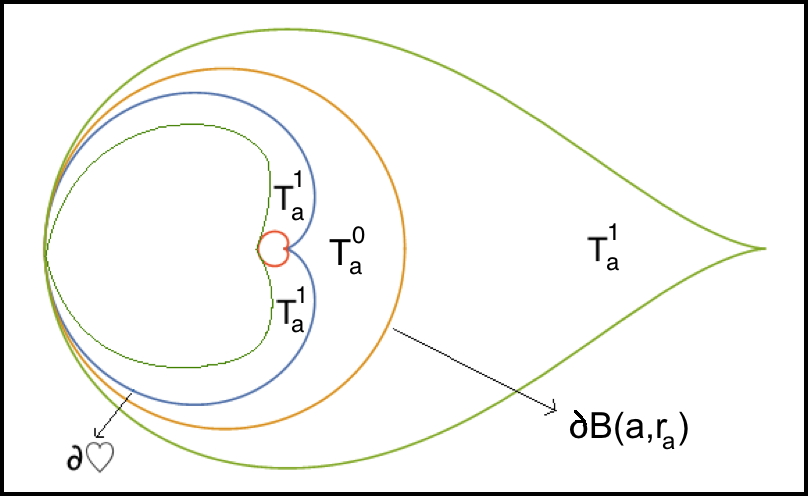}
\caption{Left: For any $a$ in the slit $(-\infty,-\frac{1}{12})$, the disk $B(a,r_a)$ touches $\heartsuit$ at two points. Right: For any $a\in\mathbb{C}\setminus(-\infty,-\frac{1}{12})$, the tiles of rank $0$ and $1$ are labeled as $T_a^0$ and $T_a^1$ respectively. The union of the tiles of rank $0$ and $1$ is denoted by $E_a^1$.}
\label{real_slit_double}
\end{figure}

\subsubsection{The family $\mathcal{S}$}\label{family_defn} In this paper, we will only be interested in the situation where the circumcircle to the cardioid touches it at only one point. For $a\in\C\setminus (-\infty,-\frac{1}{12})$, let $$\Omega_a := \heartsuit\cup\overline{B}(a,r_a)^c,\quad \mathrm{and}\quad T_a:=\Omega_a^c= \overline{B}(a,r_a)\setminus\heartsuit.$$ We now define our dynamical system $F_a:\overline{\Omega}_a\to\widehat{\C}$ as, 
$$
w \mapsto \left\{\begin{array}{ll}
                    \sigma(w) & \mbox{if}\ w\in\overline{\heartsuit}, \\
                    \sigma_a(w) & \mbox{if}\ w\in B(a,r_a)^c, 
                                          \end{array}\right. 
$$
where $\sigma$ is the Schwarz reflection of $\heartsuit$, and $\sigma_a$ is reflection with respect to the circle $\vert w-a\vert=r_a$. It follows from our previous discussion that $0$ is the only critical point of $F_a$. We will call this family of maps $\mathcal{S}$; i.e. 
$$
\mathcal{S}:=\left\{F_a:\overline{\Omega}_a\to\widehat{\C}:a\in\C\setminus \left(-\infty,-\frac{1}{12}\right)\right\}.
$$

\begin{proposition}\label{cardioid_circle_branched_cover}
For each $a\in\C\setminus\left(-\infty,-\frac{1}{12}\right)$, the map $F_a: F_a^{-1}(\Omega_a)\to\Omega_a$ is a two-to-one branched covering, branched only at $0$.
\end{proposition}
\begin{proof}
Note that $F_a^{-1}(\Omega_a)\subset\Omega_a$ consists of three open connected components. Two of which, namely $\sigma^{-1}(\heartsuit)$ and $\sigma_a^{-1}(\heartsuit)$, map univalently onto $\heartsuit$. On the other hand, $0\in\sigma^{-1}(\overline{B}(a,r_a)^c)$, so $F_a$ is a $2:1$ branched cover from $\sigma^{-1}(\overline{B}(a,r_a)^c)$ onto $\overline{B}(a,r_a)^c.$
\end{proof}

\begin{corollary}\label{circle_cardioid_inverse_branches}
Let $V\subset\Omega_a$ be a simply connected domain such that the forward orbit of the critical point $0$ does not intersect $V$. Then, all inverse branches of $F_a^{\circ n}$ ($n\geq1$) are defined on $V$. 
\end{corollary}

\subsubsection{Dynamics near singular points}\label{dyn_sing_subsec} Note that $\partial T_a$ has two singular points; namely the double point $\alpha_a$ and the cusp point $\frac{1}{4}$. Both of these are fixed points of $F_a$. As the map $F_a$ admits no anti-holomorphic extension in neighborhoods of these fixed points, we need to obtain local expansions of $F_a$ in suitable relative neighborhoods of $\alpha_a$ and $\frac{1}{4}$. 

Formula~(\ref{real_formula}) shows that the map $F_a$ has no single-valued anti-holomorphic extension in a neighborhood of $\frac{1}{4}$ since extending $F_a$ near $\frac{1}{4}$ involves choosing a branch of square root. A straightforward computation yields the following asymptotics of $F_a$ near $\frac{1}{4}$.

\begin{proposition}[Dynamics near cardioid cusp]\label{cusp_asymp}
Choosing the branch of square root that sends positive reals to positive reals, we have $$F_a(w)=\overline{w}-k(\frac{1}{4}-\overline{w})^{\frac32}+O((\frac{1}{4}-\overline{w})^2)$$ where $w\in B(\frac{1}{4},\epsilon)\setminus(\frac{1}{4},\frac{1}{4}+\epsilon)$, for some $k>0$ and $\epsilon>0$ small enough. Hence, $\frac{1}{4}$ repels nearby real points on its left.
\end{proposition}

Moreover, Corollary~\ref{iterated_pre_image_cardioid} gives a precise description of the dynamics of $F_a$ near the fixed point $\frac{1}{4}$; in particular, all points in $\heartsuit\cap B(\frac{1}{4},\epsilon)$ eventually leave $\heartsuit\cap B(\frac{1}{4},\epsilon)$ or escape to $T_a$ under iterations of $F_a$.  

On the other hand, both $\sigma$ and $\sigma_a$ admit anti-holomorphic extensions in a neighborhood of $\alpha_a$. As germs, these extensions are anti-holomorphic involutions. However, these extensions do not match up to yield a single anti-holomorphic map in a neighborhood of $\alpha_a$. In this case, it will be more convenient to work with the second iterate $F_a^{\circ 2}$. Note that 
\begin{equation}
F_a^{\circ 2}=\left\{\begin{array}{ll}
                      \sigma_a\circ\sigma & \mbox{on}\ \sigma^{-1}(\overline{B}(a,r_a)^c), \\
                      \sigma\circ\sigma_a & \mbox{on}\ \sigma_a^{-1}(\heartsuit).
                                          \end{array}\right. 
\label{iterate_alpha}
\end{equation}

Let us first look at the parameters $a\neq -\frac{1}{12}$. For such parameters $a$, the curves $\partial\heartsuit$ and $\partial B(a,r_a)$ have a simple tangency at $\alpha_a$ (in particular, they have common tangent and normal lines). The following result describes the local behavior of the maps $\sigma_a\circ\sigma$ and $\sigma\circ\sigma_a$ near $\alpha_a$. 

\begin{proposition}[Dynamics near double point; $a\neq-\frac{1}{12}$]\label{double_asymp}
Let $a\neq -\frac{1}{12}$. Then, $\sigma$ and $\sigma_a$ extend as local anti-holomorphic involutions near $\alpha_a$ such that 
$$(\sigma_a\circ\sigma)(w)=w+k_a(w-\alpha_a)^2+O((w-\alpha_a)^3),\quad \mathrm{and}$$ $$(\sigma_a\circ\sigma)^{-1}(w)=(\sigma\circ\sigma_a)(w)=w-k_a(w-\alpha_a)^2+O((w-\alpha_a)^3)$$ for all $w\in B(\alpha_a,\epsilon)$ and some $k_a\neq 0$. Moreover, the inward (respectively, outward) normal vector to $\partial\heartsuit$ at $\alpha_a$ is the repelling direction of the parabolic germ $\sigma_a\circ\sigma$ (respectively, of $\sigma\circ\sigma_a$). These are the only repelling directions for $F_a^{\circ 2}$ at $\alpha_a$.
\end{proposition}
\begin{proof}
The fact that $\sigma$ (respectively, $\sigma_a$) admits a local anti-holomorphic extension near $\alpha_a$ follows from the fact that $\alpha_a$ is a non-singular point of $\partial\heartsuit$ (respectively, of $\partial B(a,r_a)$). Since these local extensions are anti-holomorphic involutions, it follows that $\sigma_a\circ\sigma$ is the inverse of $\sigma\circ\sigma_a$ (as germs).

The first two terms in the local expansions of the Schwarz reflection maps $\sigma$ and $\sigma_a$ can be computed in terms of the slope and curvature of the corresponding curves at $\alpha_a$ (see \cite[Chapter~7]{Dav74}). Moreover, since $\partial\heartsuit$ and $\partial B(a,r_a)$ do not have the same curvature at $\alpha_a$ (recall that $\partial B(a,r_a)$ is not the osculating circle to $\partial\heartsuit$ at $\alpha_a$ for $a\neq -\frac{1}{12}$), it follows that $k_a\neq 0$. The final statement on repelling directions of the parabolic germs is a consequence of the fact that $\partial\heartsuit$ has greater curvature than $\partial B(a,r_a)$ at $\alpha_a$ (as $\partial B(a,r_a)$ is a circumcircle of $\partial\heartsuit$).
\end{proof}

\begin{remark}\label{contact_schwarz}
More generally, if two real-analytic smooth curve germs $\gamma_1$ and $\gamma_2$ with associated Schwarz reflection maps $\sigma_1$ and $\sigma_2$ (respectively) have contact of order $k$ at the origin, then $\sigma_1\circ\sigma_2$ is of the form $z+cz^{k+1}+\cdots$.
\end{remark}
 
Let us choose repelling petals $P_1$ and $P_2$ of the parabolic germs $\sigma_a\circ\sigma$ and $\sigma\circ\sigma_a$ contained in $\sigma^{-1}(\overline{B}(a,r_a)^c)$ and $\sigma_a^{-1}(\heartsuit)$ respectively (in other words, $P_1$ and $P_2$ are repelling and attracting petals for the parabolic germ $\sigma_a\circ\sigma$). Then, 
\begin{equation}
\bigcap_{n\in\N}\ F_a^{-2n}(\overline{P_1})=\bigcap_{n\in\N}\ F_a^{-2n}(\overline{P_2})=\{\alpha_a\},
\label{petals_shrink}
\end{equation}
where we have chosen the inverse branch of $F_a^{\circ 2}$ near $\alpha_a$ that fixes $\alpha_a$.

We now turn our attention to the parameter $a=-\frac{1}{12}$.

\begin{proposition}[Dynamics near double point; $a=-\frac{1}{12}$]\label{fat_basilica_local_dyn}
Let $a=-\frac{1}{12}$. Then $\sigma$ and $\sigma_a$ extend as local anti-holomorphic involutions near $\alpha_a=-\frac{3}{4}$ such that 
$$(\sigma_a\circ\sigma)(w)=w-k'_a(w+\frac{3}{4})^4+O((w+\frac{3}{4})^5),\quad \mathrm{and}$$
$$(\sigma_a\circ\sigma)^{-1}(w)=(\sigma\circ\sigma_a)(w)=w+k'_a(w+\frac{3}{4})^4+O((w+\frac{3}{4})^5)$$ for all $w\in B(-\frac{3}{4},\epsilon)$ and some $k'_a>0$. Moreover, the positive (respectively, negative) real direction at $-\frac{3}{4}$ is an attracting vector of the former (respectively, latter) parabolic germ, and these are the only attracting directions for $F_a^{\circ 2}$ at $-\frac34$. Thus, $-\frac{3}{4}$ attracts nearby real points under iterates of $F_{a}$.
\end{proposition}
\begin{proof}
For $a=-\frac{1}{12}$, the curves $\partial\heartsuit$ and $\partial B(a,r_a)$ have a third order tangency at $\alpha_a=-\frac{3}{4}$ (i.e. a contact of order $3$), so the desired asymptotics follow from Remark~\ref{contact_schwarz}. The fact that $\sigma$ (respectively, $\sigma_a$) admits a local anti-holomorphic extension near $-\frac{3}{4}$ follows from the fact that $-\frac{3}{4}$ is a non-singular point of $\partial\heartsuit$ (respectively, of $\partial B(a,r_a)$). Since these extensions are involutions, $\sigma_a\circ\sigma$ is the inverse of $\sigma\circ\sigma_a$. The asymptotics of $\sigma_a\circ\sigma$ and $\sigma\circ\sigma_a$ near $-\frac34$ can be explicitly computed using Formula~(\ref{real_formula}). The attracting directions of the parabolic germs are readily seen from these asymptotics. The fact that these are the only attracting directions for $F_a^{\circ 2}$ at $-\frac34$ follows from Formula~(\ref{iterate_alpha}).
\end{proof}

\begin{remark}
We see from the above local expansions that each of the parabolic germs $\sigma_a\circ\sigma$ and $\sigma\circ\sigma_a$ has two more attracting directions at $-\frac34$, however they lie in regions where the germs do not coincide with $F_a^{\circ 2}$, see Formula~(\ref{iterate_alpha}).
\end{remark}

\subsubsection{Tiling set, non-escaping set, and limit set}\label{dyn_inv_sets} We now proceed to define the dynamically relevant sets for a general map $F_a\in\mathcal{S}$. It is easy to see that the points $\alpha_a$ and $\frac{1}{4}$ have only two preimages under $F_a$, and every point of $T_a\setminus\{\alpha_a,\frac{1}{4}\}$ has three preimages under $F_a$. In order to get an honest covering map, we will therefore work with $T_a^0:= T_a\setminus\{\alpha_a,\frac{1}{4}\}$. Then, the restriction $F_a:F_a^{-1}(T_a^0)\to T_a^0$ is a degree $3$ covering.

\begin{definition}[Tiling set, non-escaping set, and limit set]

\begin{itemize}
\item For any $k~\geq~0$, the connected components of $F_a^{-k}(T_a^0)$ are called \emph{tiles} (of $F_a$) of rank $k$. The unique tile of rank $0$ is $T_a^0$.

\item The \emph{tiling set} $T_a^\infty$ of $F_a$ is defined as the set of points that eventually escape to $T_a^0$; i.e. $T_a^\infty=\bigcup_{k=0}^\infty F_a^{-k}(T_a^0)$. Equivalently, the tiling set is the union of all tiles.

\item The \emph{non-escaping set} $K_a$ of $F_a$ is the complement $\widehat{\C}\setminus T_a^\infty$. Connected components of $\Int{K_a}$ are called \emph{Fatou components} of $F_a$. All iterates of $F_a$ are defined on $K_a$. 

\item The boundary of $T_a^\infty$ is called the \emph{limit set} of $F_a$, and is denoted by $\Gamma_a$.
\end{itemize}
\end{definition}

\begin{remark}
Note that the tiling set, non-escaping set, and limit set of $F_a$ are the analogues of basin of infinity, filled Julia set, and Julia set (respectively) of a polynomial (see \cite[\S 18]{M1new} for definitions and basic properties of these sets).
\end{remark}

\begin{figure}
\includegraphics[width=0.36\linewidth]{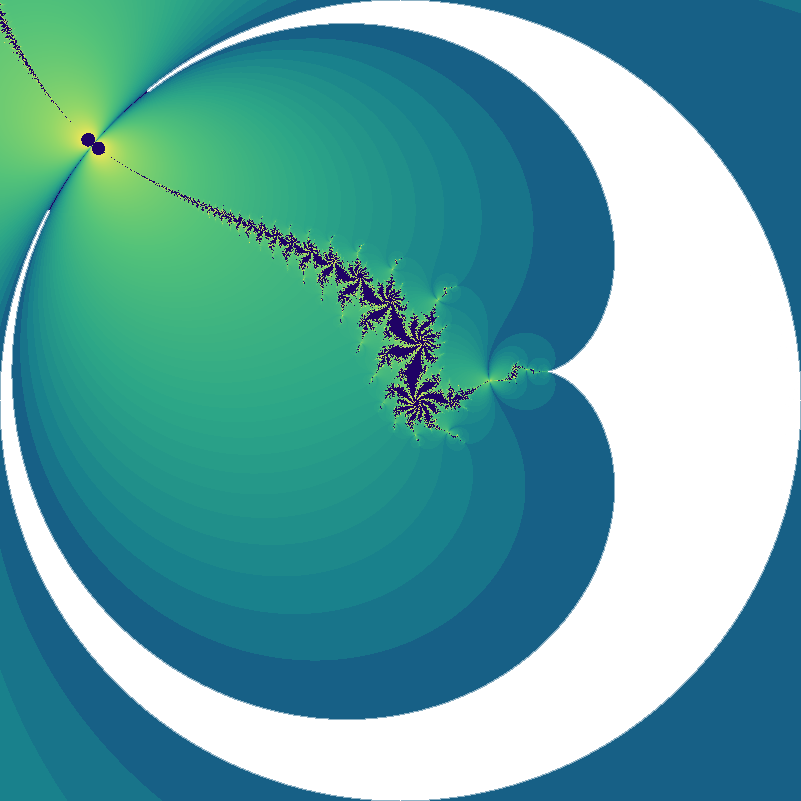}\quad \includegraphics[width=0.35\linewidth]{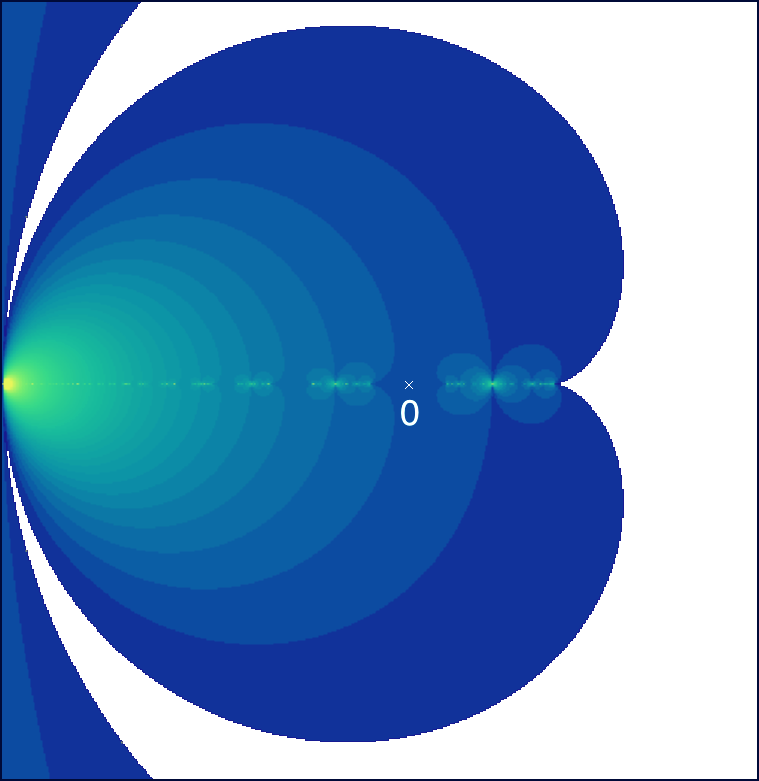}

\caption{Left: A part of the dynamical plane of a parameter in $\cC(\mathcal{S})$. The white region denotes the tile of rank $0$. Right: A part of the dynamical plane of a parameter outside $\cC(\mathcal{S})$. The critical point $0$ is contained in a ramified tile, which disconnects the non-escaping set.}

\label{various_limit_sets}
\end{figure}

The tiling set and the non-escaping set yield an invariant partition of the dynamical plane of $F_a$.

\begin{proposition}[Tiling set is open and connected]\label{escape_connected}
For each $a\in\C\setminus(-\infty,-\frac{1}{12})$, the tiling set $T_a^\infty$ is an open connected set, and hence the non-escaping set $K_a$ is closed.
\end{proposition} 
\begin{proof}
Let us denote the union of the tiles of rank $0$ through $k$ by $E_a^k$ (see Figure~\ref{real_slit_double}). Since every tile of rank $k\geq1$ is attached to a tile of rank $(k-1)$ along a boundary curve and $\Int{E_a^0}$ is connected, it follows that $\Int{E_a^k}$ is connected. Moreover, $\Int{E_a^k}\subsetneq\Int{E_a^{k+1}}$ for each $k\geq 0$. 

Note that if $z\in T_a^\infty$ belongs to the interior of a tile of rank $k$, then it lies in the interior of $E_a^k$. On the other hand, if $z\in T_a^\infty$ belongs to the boundary of a tile of rank $k$, then it lies in the interior of $E_a^{k+1}$. Hence, 
$$
T_a^\infty=\bigcup_{k\geq 0}\Int{E_a^k}.
$$ 
Thus, $T_a^\infty$ is an increasing union of open, connected sets, and hence itself is such. Consequently, its complement $K_a$ is a closed set. 
\end{proof}

\begin{corollary}\label{fatou_simp_conn}
Each Fatou component of $F_a$ is simply connected.
\end{corollary}

\begin{proposition}[Complete invariance]\label{dynamically_inv}
For each $a\in\C\setminus(-\infty,-\frac{1}{12})$, both $T_a^\infty$ and $K_a$ are completely invariant under $F_a$.
\end{proposition}

The next proposition sheds light on the geometry of the tiling set near the singular points $\frac14$ and $\alpha_a$. It roughly says that the tiling set contains sufficiently wide angular wedges near the singular points.

\begin{proposition}\label{wedge_at_cardioid_sing} 
1) Let $a\in\C\setminus(-\infty,-\frac{1}{12})$. Then for every $\delta\in\left(0,\pi\right)$, there exists $R>0$ such that $$W_a^1:=\left\{\frac14+re^{i\theta}: 0<r<R,\ -(\pi-\delta)<\theta<(\pi-\delta)\right\}\subset T_a^\infty.$$ 

2) Let $a\in\C\setminus(-\infty,-\frac{1}{12}]$, and $\theta_0$ be the slope of the inward normal vector to $\partial\heartsuit$ at $\alpha_a$. Then for every $\delta\in\left(0,\frac{\pi}{2}\right)$, there exists $R>0$ such that $$W_a^2:=\alpha_a+\left\{re^{i(\theta_0+\theta)}: 0<r<R,\ \delta<\vert\theta\vert<\pi-\delta\right\}\subset T_a^\infty.$$
\end{proposition}
\begin{proof}
1) This is a consequence of the ``parabolic" behavior of $F_a$ near $\frac14$. The local Puiseux series expansion of $F_a$ near $\frac14$ (see Proposition~\ref{cusp_asymp}) imply that $(\frac14-\epsilon,\frac14)$ is the unique repelling direction of $F_a$ at $\frac14$, and the iterated preimages of the fundamental tile $T_a^0$ occupy a circular sector of Stolz angle $2(\pi-\delta)$ (disjoint from $(\frac14-\epsilon,\frac14)$) at $\frac14$ (compare Figure~\ref{various_limit_sets}).

2) This is a consequence of the parabolic dynamics of $F_a^{\circ 2}$ near $\alpha_a$. The local power series expansions of the piecewise definitions of $F_a^{\circ 2}$ near $\alpha_a$ (see Formula~(\ref{iterate_alpha}) and Proposition~\ref{double_asymp}) imply that the inward and outward normal vectors to $\partial\heartsuit$ at $\alpha_a$ are the only repelling directions of $F_a^{\circ 2}$ at that point, and the iterated preimages of each of the two connected components of $T_a^0\cap B(\alpha_a,\epsilon)$ (for $\epsilon>0$ sufficiently small) occupy a circular sector of Stolz angle $(\pi-2\delta)$ (disjoint from the repelling directions) at $\alpha_a$ (compare Figure~\ref{various_limit_sets}).
\end{proof}

We will now show that connectivity of the non-escaping set $K_a$ is completely determined by the dynamics of the unique critical point $0$ (as in the case of quadratic polynomials or anti-polynomials).

\begin{figure}[h!]
\begin{tikzpicture}
\node[anchor=south west,inner sep=0] at (0,0) {\includegraphics[width=0.3\linewidth]{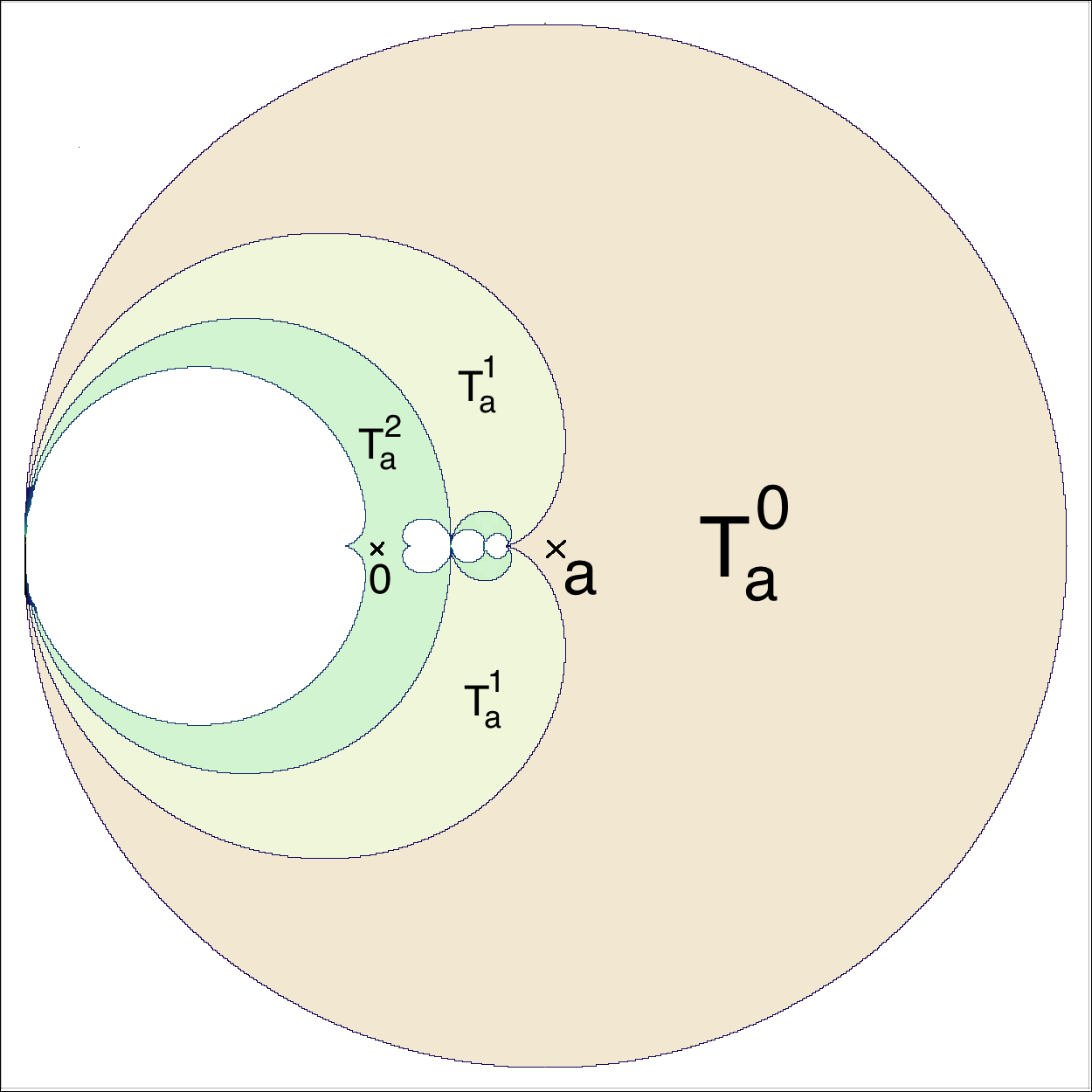}};
\node[anchor=south west,inner sep=0] at (5,0) {\includegraphics[width=0.484\linewidth]{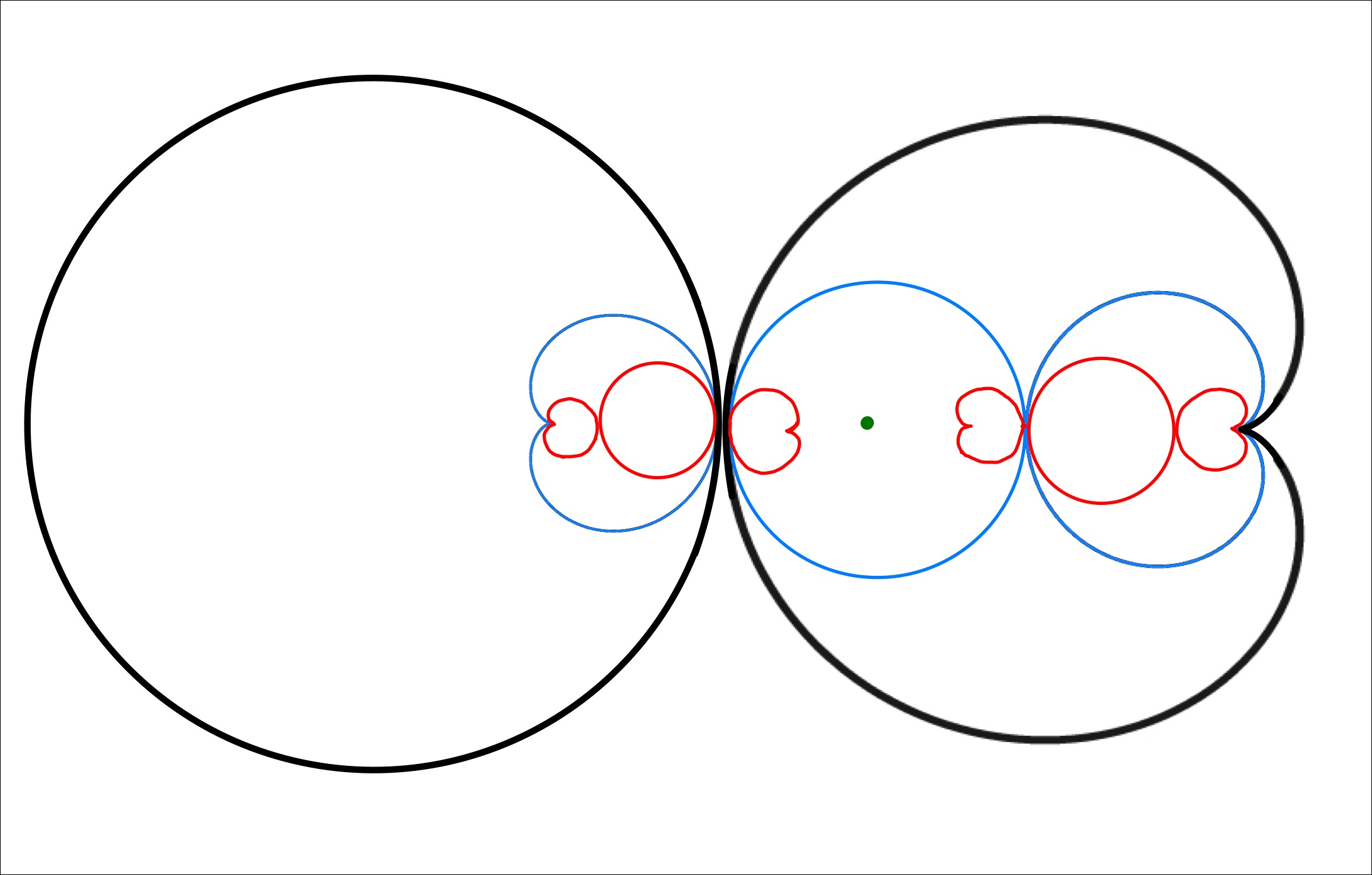}};
\node at (8.2,3.2) {$T_a^0$};
\node at (9.6,2.8) {$T_a^1$};
\node at (9.6,1) {$T_a^1$};
\node at (6,2) {$T_a^1$};
\node at (9,1.6) {\begin{tiny}$T_a^2$\end{tiny}};
\end{tikzpicture}
\caption{Left: Displayed are some of the initial tiles for a parameter $a$ in the escape locus for which the critical point $0$ escapes in two iterates. For such a parameter $a$, the external conjugacy $\psi_a$ of Proposition~\ref{schwarz_group} extends conformally to all the tiles of the first rank. Right: The same dynamical plane in an inverted coordinate system such that the droplet contains $\infty$ in its interior. The tiles of rank $0, 1$, and the rank $2$ tile containing the critical point are marked. This ramified rank $2$ tile disconnects the non-escaping set.}
\label{quickest_escape}
\end{figure}

\begin{proposition}\label{prop_critical_conn_locus}
The non-escaping set $K_a$ is connected if and only if it contains the critical point $0$.
\end{proposition}
\begin{proof}
If $0$ does not escape under $F_a$, then every tile is unramified, and $\widehat{\C}\setminus\Int{E_a^k}$ is a full continuum for each $k\geq 0$. Therefore, $K_a=\bigcap_{k\geq 0}\left(\widehat{\C}\setminus\Int{E_a^k}\right)$ is a nested intersection of full continua and hence is a full continuum itself. 

Now suppose that $0\in T_a^\infty$. Since $\infty\notin T_a^0$, there exists a smallest integer $k\geq 2$ such that $0$ lies in a tile of rank $k$. Then the tile containing $0$ is a subset of the closure of $\sigma^{-1}(\overline{B}(a,r_a)^c)$. As the tiles of rank $0$ through $k-1$ do not contain the critical point, it follows that $E_a^{k-1}$ is simply connected. Also, since $\sigma:\sigma^{-1}(\overline{B}(a,r_a)^c)\to\overline{B}(a,r_a)^c$ is a two-to-one branched cover branched at $0$, it follows that $\sigma:\sigma^{-1}(\overline{B}(a,r_a)^c)\setminus E_a^k\to \overline{B}(a,r_a)^c\setminus E_a^{k-1}$ is a two-to-one covering map onto a simply connected set. Thus, $\sigma^{-1}(\overline{B}(a,r_a)^c)\setminus E_a^k$ must consist of two disjoint copies of $\overline{B}(a,r_a)^c\setminus E_a^{k-1}$. Since $K_a\cap\sigma^{-1}(\overline{B}(a,r_a)^c)$ is contained in the union of these two copies, and intersects both copies non-trivially, we obtain a disconnection of $K_a$ (see Figure~\ref{quickest_escape} for the situation when $k=2$).
\end{proof}

\subsubsection{The connectedness locus}\label{conn_locus_def} The above proposition leads to the following definition.

\begin{definition}[Connectedness locus and escape locus]
The connectedness locus of the family $\mathcal{S}$ is defined as $$\cC(\mathcal{S})=\{a\in\C\setminus(-\infty,-\frac{1}{12}): 0\notin T_a^\infty\}=\{a\in\C\setminus(-\infty,-\frac{1}{12}): K_a\ \textrm{is\ connected}\}.$$ The complement of the connectedness locus in the parameter space is called the \emph{escape locus}.
\end{definition}

Let us collect a few easy facts about the connectedness locus $\cC(\mathcal{S})$.

\begin{proposition}\label{conn_locus_in_cardioid}
$\cC(\mathcal{S})\subset \heartsuit\cup\{\frac{1}{4}\}$.
\end{proposition}
\begin{proof}
Note that for all $a\in\C\setminus(-\infty,-\frac{1}{12})$, we have $F_a^{\circ 2}(0)=a\in B(a,r_a)$. Clearly, $a\neq \alpha_a$. Therefore, if $a\notin\heartsuit\cup\{\frac{1}{4}\}$, then $F_a^{\circ 2}(0)\in T_a^0$; i.e. $0\in T_a^\infty$.
\end{proof}

\begin{proposition}\label{conn_locus_almost_closed}
$\cC(\mathcal{S})$ is closed in $\C\setminus(-\infty,-\frac{1}{12})$.
\end{proposition}
\begin{proof}
Note that the fundamental tile $T_a^0$ varies continuously with the parameter as $a$ runs over $\C\setminus\left(-\infty,-\frac{1}{12}\right)$. Now let $a_0\in\C\setminus\left(-\infty,-\frac{1}{12}\right)$ be a parameter outside the connectedness locus $\cC(\mathcal{S})$. Then there exists some positive integer $n_0$ such that $F_{a_0}^{\circ n_0}(0)\in T_{a_0}^0$. It follows that for all $a$ sufficiently close to $a_0$, we have $F_{a}^{\circ n_0}(0)\in T_a^0$ or $F_{a}^{\circ (n_0+1)}(0)\in T_a^0$. Hence, $\cC(\mathcal{S})$ is closed in $\C\setminus\left(-\infty,-\frac{1}{12}\right)$.
\end{proof}

\begin{proposition}[Real slice of the connectedness locus]\label{real_connected}
$\cC(\mathcal{S})\cap\R=\left[-\frac{1}{12}, \frac{1}{4}\right]$.
\end{proposition}
\begin{proof}
Let $a\in\left[-\frac{1}{12},\frac{1}{4}\right]$. For such parameters $a$, we have $\alpha_a=-\frac{3}{4}$ and $r_a=(a+\frac{3}{4})$. Moreover, using Formula~(\ref{real_formula}), we see that:
\begin{itemize}
\item $F_a(\left[-\infty,-\frac{3}{4}\right])=\left[-\frac{3}{4},a\right]$ ($F_a$ is decreasing),

\item $F_a(\left[-\frac{3}{4},0\right])=\left[-\infty,-\frac{3}{4}\right]$ ($F_a$ is decreasing),

\item $F_a(\left[0,\frac{1}{4}\right])=\left[-\infty,\frac{1}{4}\right]$ ($F_a$ is increasing).
\end{itemize}

In particular, the interval $\left[-\infty, \frac{1}{4}\right]$ is invariant under $F_a$. Since $\left[-\infty, \frac{1}{4}\right]$ is disjoint from $T_a^0$, it follows that $0$ does not escape to $T_a^0$. Therefore, $\left[-\frac{1}{12}, \frac{1}{4}\right]\subset\cC(\mathcal{S})$. 

But by Proposition~\ref{conn_locus_in_cardioid}, we already know that $\cC(\mathcal{S})\subset\heartsuit\cup\lbrace\frac{1}{4}\rbrace$. This completes the proof.
\end{proof}

The above proof also shows that for each $a\in\left[-\frac{1}{12}, \frac{1}{4}\right]$, $F_a:\left[-\infty, \frac{1}{4}\right]\to\left[-\infty, \frac{1}{4}\right]$ is a unimodal map with a simple critical point at $0$.
\begin{corollary}\label{real_C1}
For each $a\in\left[-\frac{1}{12}, \frac{1}{4}\right]$, the map $F_a:\left[-\infty, \frac{1}{4}\right]\to\left[-\infty, \frac{1}{4}\right]$ is a $C^1$ unimodal map with a simple critical point at $0$.
\end{corollary}
\begin{proof}
Since $\sigma_a$ and $\sigma$ are anti-meromorphic on $\overline{B}(a,a+\frac{3}{4})^c$ and $\heartsuit$ respectively, it follows that $\sigma_a$ is $C^1$ on $\left[-\infty,-\frac{3}{4}\right)$ and $\sigma$ is $C^1$ on $\left(-\frac{3}{4},\frac{1}{4}\right]$. Moreover, the real-analytic extensions of these maps (in an $\R$-neighborhood of $-\frac{3}{4}$) have a common derivative $-1$ at $-\frac{3}{4}$. This proves that $F_a$ is $C^1$-smooth on $\left[-\infty, \frac{1}{4}\right]$.
\end{proof}

\subsubsection{Quasiconformal deformations in $\mathcal{S}$}\label{qcdef_sec} Quasiconformal deformations play a crucial role in studying the parameter spaces of holomorphic/anti-holomorphic dynamical systems. In this spirit, we will prove a lemma that will allow us to talk about quasiconformal deformations in the family $\mathcal{S}$.

\begin{lemma}[Quasiconformal deformation of Schwarz reflections/changing the mirrors]\label{schwarz_qcdef}
Let $\mu$ be an $F_a$-invariant Beltrami coefficient on $\widehat{\C}$, and $\Phi$ be a quasiconformal map integrating the Beltrami coefficient $\mu$ such that $\Phi$ fixes $0, \frac{1}{4},$ and $\infty$. Let $b=\Phi(a)$. Then, $b\in\C\setminus (-\infty,-\frac{1}{12})$, $\Phi(\Omega_a)=\Omega_b$, and $F_b=\Phi\circ F_a\circ\Phi^{-1}$ on $\Omega_b$.
\end{lemma}
\begin{proof}
The assumption that $\mu$ is $F_a$-invariant implies that $\Phi\circ F_a\circ\Phi^{-1}$ is anti-meromorphic on $\Phi(\Omega_a)$. Hence, $\Phi\circ F_a\circ\Phi^{-1}$ is an anti-meromorphic map on $\Phi(\overline{B}(a,r_a)^c)$ that continuously extends to the identity map on $\Phi(\partial \overline{B}(a,r_a)^c)=\partial\Phi(\overline{B}(a,r_a)^c)$. Since $\Phi$ fixes $\infty$, it follows that $\Phi(\overline{B}(a,r_a)^c)$ is an unbounded quadrature domain with Schwarz reflection map $\Phi\circ \sigma_a\circ\Phi^{-1}$. Since $F_a$ maps $\overline{B}(a,r_a)^c$ univalently onto $B(a,r_a)$, we conclude that $\Phi\circ F_a\circ\Phi^{-1}$ maps $\Phi(\overline{B}(a,r_a)^c)$ univalently onto $\Phi(B(a,r_a))$. It now follows by Proposition~\ref{cardioid_char_prop} that $\Phi(\overline{B}(a,r_a)^c)$ is an exterior disk. 

Let $b=\Phi(a)$. Recall that the Schwarz reflection map of an exterior disk maps $\infty$ to the center of the interior part of the disk. Since $\sigma_a(\infty)=a$ and $\Phi(\infty)=\infty$, it follows that the Schwarz reflection map $\Phi\circ \sigma_a\circ\Phi^{-1}$ of $\Phi(\overline{B}(a,r_a)^c)$ maps $\infty$ to $b$. Hence, $\Phi(\overline{B}(a,r_a)^c)$ is of the form $\overline{B}(b,r)^c$, for some $r>0$.

Again, $\Phi\circ F_a\circ\Phi^{-1}$ is an anti-meromorphic map on $\Phi(\heartsuit)$ that continuously extends to the identity map on $\Phi(\partial \heartsuit)=\partial\Phi(\heartsuit)$. As $\Phi$ fixes $0$ and $\infty$, it follows that $\Phi(\heartsuit)\ni 0$ is a bounded simply connected quadrature domain with associated Schwarz reflection map $\Phi\circ \sigma\circ\Phi^{-1}$. Also note that $\Phi\circ \sigma\circ\Phi^{-1}$ maps $\Phi(\sigma^{-1}(\heartsuit))$ univalently onto $\Phi(\heartsuit)$. Moreover, $\Phi\circ \sigma\circ\Phi^{-1}$ has a double pole at $0$, and $\partial\Phi(\heartsuit)$ has a cusp at $\frac14$. Therefore by Proposition~\ref{cardioid_char_prop}, we have that $\Phi(\heartsuit)=\heartsuit$, and $\Phi\circ \sigma\circ\Phi^{-1}\equiv\sigma$ on $\heartsuit$. 

As $\partial B(a,r_a)$ and $\partial\heartsuit$ touch at a unique point, the same is true for $\partial\overline{B}(b,r)$ and $\partial\heartsuit$. Hence, $\partial B(b,r)$ is a circumcircle of $\heartsuit$ touching $\partial\heartsuit$ at a single point. So, $r=r_b$, and $\Phi\circ \sigma_a\circ\Phi^{-1}\equiv\sigma_b$ on $\overline{B}(b,r_b)^c$. In particular, $b\in\C\setminus (-\infty,-\frac{1}{12})$.

To summarize, we have proved that $\Phi(\heartsuit)=\heartsuit$ and $\Phi(\overline{B}(a,r_a)^c)=\overline{B}(b,r_b)^c$ (for some $b\in\C\setminus (-\infty,-\frac{1}{12})$), where $\partial B(b,r_b)$ circumscribes $\partial\heartsuit$ and touches it at a single point. Hence, $\Phi(\Omega_a)=\Omega_b$. Furthermore, $\Phi\circ \sigma\circ\Phi^{-1}\equiv\sigma$ on $\heartsuit$, and $\Phi\circ \sigma_a\circ\Phi^{-1}\equiv\sigma_b$ on $\overline{B}(b,r_b)^c$. Therefore, $\Phi\circ F_a\circ\Phi^{-1}\equiv F_b$ on $\Omega_b$.

This completes the proof.
\end{proof}

\subsubsection{Classification of Fatou components}\label{fatou_class_sec} We will now discuss the classification of Fatou components, and their relation with the post-critical set (of the maps in $\mathcal{S}$). 

We define the \emph{multiplier} of a periodic point of an anti-holomorphic map as the $z$-derivative of its holomorphic first return map (compare \cite[\S 1.1]{Sa1}). A cycle of an anti-holomorphic map is called \emph{super-attracting}, \emph{attracting}, \emph{neutral}, or \emph{repelling} if the  absolute value of the associated multiplier is equal to $0$, lies in $(0,1)$, is equal to $1$, or is greater than $1$ (respectively).
A neutral cycle is called \emph{parabolic} if the associated multiplier is a root of unity, and \emph{irrationally neutral} otherwise. It is easily seen that a neutral cycle of odd period (of an anti-holomorphic map) is necessarily parabolic with multiplier $1$ (see \cite[Lemma~2.8]{MNS}). Finally, an irrationally neutral periodic point of an anti-holomorphic map (necessarily of even period) is called a \emph{Siegel} point (respectively, \emph{Cremer} point) if its first return map is linearizable (respectively, non-linearizable); i.e., conformally conjugate (respectively, not conformally conjugate) to a rigid irrational rotation. In the former case, the maximal domain on which a linearization exists is a conformal disk, called a \emph{Siegel disk}.

We refer the readers to \cite{M1new} for a comprehensive discussion of the local fixed point theory of holomorphic maps, and to \cite[\S 1.1]{Sa1} for analogous results in the anti-holomorphic setting.

\begin{proposition}[Classification of Fatou components]\label{fatou_classification}
Let $a\in\C\setminus(-\infty,-\frac{1}{12})$. Then the following hold true.

1) Every Fatou component of $F_a$ is eventually preperiodic.

2) For $a\neq-\frac{1}{12}$, every periodic Fatou component of $F_a$ is either the (immediate) basin of attraction of an attracting cycle, or the (immediate) basin of attraction of a parabolic cycle, or a Siegel disk.

3) For $a=-\frac{1}{12}$, each periodic Fatou component of $F_a$ is either the (immediate) basin of attraction of an attracting/parabolic cycle, or a Siegel disk, or the (immediate) basin of attraction of the singular point $-\frac34$. In fact, the map $F_a$ has a $2$-cycle of Fatou components such that every point in these components converges (under the dynamics) to the singular point $-\frac34$ through a period two cycle of attracting petals (intersecting the real line). Moreover, these are the only two periodic Fatou components of $F_a$ where orbits converge to $-\frac34$.
\end{proposition}

\begin{remark}
We will soon see that the only periodic Fatou components of the map $a=-\frac{1}{12}$ are the two periodic Fatou components touching at the singular point $-\frac34$. 
\end{remark}

\begin{proof}
1) The classical proof of non-existence of wandering Fatou components for rational maps \cite[Theorem~1]{Sul} can be adapted to the current setting.\footnote{Proofs of nonexistence of wandering Fatou components for various classes of non-rational maps can be found in \cite{EL1,EL2,GK,Adam}.} This is possible because the family $\mathcal{S}$ is quasiconformally closed (by Lemma~\ref{schwarz_qcdef}) and the parameter space of $\mathcal{S}$ is finite-dimensional (in fact, real two-dimensional). If $F_a$ were to have a wandering Fatou component (such a component would necessarily be simply connected by Corollary~\ref{fatou_simp_conn}), then one could construct an infinite-dimensional space of nonequivalent deformations of $F_a$ supported on the grand orbit of the wandering component, all of which would be members of $\mathcal{S}$. Evidently, this would contradict the fact that the parameter space of $\mathcal{S}$ is finite-dimensional.

2) Let $a\neq-\frac{1}{12}$, and $U$ be a periodic Fatou component of period $k$ of $F_a^{\circ 2}$. Recall that by Corollary~\ref{fatou_simp_conn}, $U$ is simply connected. We can now invoke \cite[Theorem~5.2, Lemma~5.5]{M1new} to conclude that the Fatou component $U$ is either the (immediate) basin of attraction of an attracting cycle, or a Siegel disk, or all orbits in $U$ converge to a unique boundary point $w_0\in\partial U$ with $F_a^{\circ n}(w_0)=w_0$. 

To complete the proof, we only need to focus on the third case. To this end, suppose that all orbits in $U$ converge to a unique boundary point $w_0\in\partial U$ with $F_a^{\circ n}(w_0)=w_0$. 

\noindent \textbf{Case 1: $w_0$ is not a singular point of $T_a$.} In this case, $F_a^{\circ 2}$ has a holomorphic extension in a full neighborhood of $w_0$. Therefore, the classical Snail Lemma (see \cite[Lemma~16.2]{M1new}) asserts that $w_0$ is a parabolic fixed point of multiplier $1$ of $F_a^{\circ 2k}$, and hence $U$ is an immediate basin of attraction of a parabolic cycle.

\noindent \textbf{Case 2: $w_0$ is a singular point of $T_a$.} In this case, we have that $w_0\in\{\frac14,\alpha_a\}$. Since $F_a^{\circ 2}$ does not have a holomorphic extension in a full neighborhood of the singular points, we cannot apply the Snail Lemma. 

However, the local asymptotics of $F_a$ and the geometry of $T_a^\infty$ near $\frac14$ (Propositions~\ref{cusp_asymp} and~\ref{wedge_at_cardioid_sing}) imply that $U\subset K_a$ is contained in the ``repelling petal" of $F_a$ at $\frac14$; i.e. orbits in $U$ cannot non-trivially converge to $\frac14$. Hence, $w_0\neq\frac14$.

Again, since $a\neq-\frac{1}{12}$, the local asymptotics of $F_a^{\circ 2}$ and the geometry of $T_a^\infty$ near $\alpha_a$ (Propositions~\ref{double_asymp} and~\ref{wedge_at_cardioid_sing}) imply that $U\subset K_a$ is contained in one of the two ``repelling petals" of $F_a^{\circ 2}$ at $\alpha_a$; i.e. orbits in $U$ cannot non-trivially converge to $\alpha_a$. Hence, $w_0\neq\alpha_a$. 

This completes the proof of the second statement of the proposition.

3) Finally, let us look at $a=-\frac{1}{12}$. For this parameter, we have $\alpha_a=-\frac34$. We only need to prove the assertions concerning Fatou components where orbits converge to $-\frac34$. 

We have seen in Proposition~\ref{fat_basilica_local_dyn} that the positive and negative real directions at $-\frac34$ are attracting vectors for $F_a^{\circ 2}$, and in particular, $-\frac{3}{4}$ attracts nearby real points under iterations of $F_{a}^{\circ 2}$. These two axes are permuted by $F_a$. 

Clearly, a sufficiently small attracting petal $P_1$ of $\sigma_a\circ\sigma$ (respectively, $P_2:=F_a(P_1)$ of $\sigma\circ\sigma_a$) containing the positive (respectively, negative) real direction at $-\frac34$ is contained in $\Int{K_a}$. Therefore, these two petals are contained in a two-periodic cycle of Fatou components of $F_a$ touching at the fixed point $-\frac{3}{4}$. 

\begin{figure}
\centering
\includegraphics[scale=0.3]{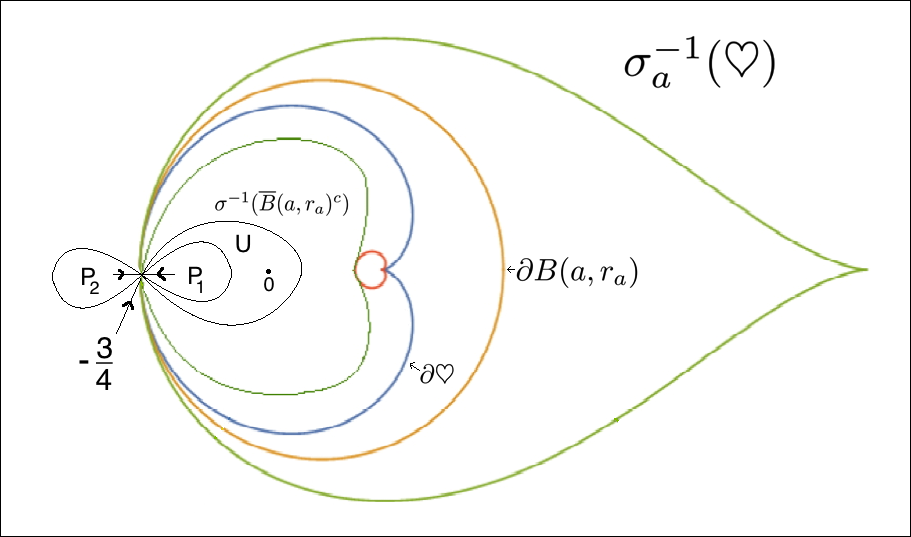}
\caption{The dynamical plane of the map $a=-\frac{1}{12}$ is shown with the attracting petals $P_1$ and $P_2$ at $-\frac34$ marked. The Fatou component $U\subset\sigma^{-1}(\overline{B}(a,r_a)^c)$ contains the petal $P_1$. The $F_a^{\circ 2}-$~orbit of every point in $U$ converges to $-\frac34$ through the attracting petal $P_1$. The critical point $0$ lies in $U$.}
\label{petals_alpha}
\end{figure}

One of these two Fatou components is contained in $\sigma^{-1}(\overline{B}(a,r_a)^c)\subset\heartsuit$, let us call it $U$. Then, $P_1\subset U$ (see Figure~\ref{petals_alpha}). Applying \cite[Lemma~5.5]{M1new} to $F_a^{\circ 2}\vert_U$, we conclude that the $F_a^{\circ 2}$-orbit of every point in $U$ converges to the boundary fixed point $-\frac{3}{4}$. Moreover, $F_a^{\circ 2}\equiv\sigma_a\circ\sigma$ on $U$. The power series expansion of $\sigma_a\circ\sigma$ near $-\frac34$ (see Proposition~\ref{fat_basilica_local_dyn}) implies that  the only attracting direction of $\sigma_a\circ\sigma$ at $-\frac34$ that is contained in $\sigma^{-1}(\overline{B}(a,r_a)^c)$ is the positive real direction. Therefore, the $F_a^{\circ 2}$-orbit of every point in $U$ must converge to $-\frac34$ through the attracting petal $P_1$. It follows that the $F_a^{\circ 2}$-orbit of every point in $F_a(U)$ must converge to $-\frac34$ through the attracting petal $P_2$.

The last statement now readily follows.
\end{proof}

Let us also mention the relation between (closures of) various types of Fatou components and the post-critical set.

\begin{proposition}[Fatou components and critical orbit]\label{fatou_critical}
1) If $F_a$ has an attracting or parabolic cycle, then the forward orbit of the critical point $0$ converges to this cycle.  

2) If $U$ is a Siegel disk of $F_a$, then $\partial U\subset\overline{\left\{F_a^{\circ n}(0)\right\}_{n\geq0}}$. Every Cremer point (i.e. an irrationally neutral, non-linearizable periodic point) of $F_a$ is also contained in $\overline{\left\{F_a^{\circ n}(0)\right\}_{n\geq0}}$.

3) For $a=-\frac{1}{12}$, the critical orbit of $F_a$ converges to the singular point $-\frac34$.
\end{proposition}
\begin{proof}
1) The proof follows classical arguments of Fatou and is similar to the corresponding proof for rational maps \cite[Theorem~8.6, Theorem~10.15]{M1new}. 

2) Let $U$ be a Siegel disk of $F_a$ of period $k$. Since $U$ is a connected component of $\Int{K_a}$ and $K_a=\widehat\C\setminus T_a^\infty$, we have that $\partial U\subset \partial T_a^\infty=\Gamma_a$. 

Let us choose $w_0\in\partial U\setminus\{\frac14,\alpha_a\}$, and assume that $w_0$ is not in the closure of the post-critical set $\overline{\{F_a^{\circ n}(0)\}_{n\geq0}}$. Then there exists $\epsilon_0>0$ such that $B(w_0,\epsilon_0)\subset \Omega_a$ and $B(w_0,\epsilon_0)\cap\overline{\{F_a^{\circ n}(0)\}_{n\geq0}}=\emptyset$. Proposition~\ref{circle_cardioid_inverse_branches} and our assumption on $B(w_0,\epsilon_0)$ ensure that we can define inverse branches of $F_a^{\circ kn}$ (for each $n\geq1$) on $B(w_0,\epsilon_0)$ such that they map $B(w_0,\epsilon_0)\cap U$ biholomorphically into $U$ and map $B(w_0,\epsilon_0)\cap\partial U$ into $\partial U$. Moreover, the image of $B(w_0,\epsilon_0)$ under all these inverse branches avoid $T_a$. Hence, they form a normal family, so some subsequence converges to a holomorphic map $g$. Since the maps $F_a^{\circ kn}\vert_U$ are conjugate to irrational rotations, it follows that $g$ is univalent on $B(w_0,\epsilon_0)$. Moreover, $g(w_0)\in \partial U\subset\Gamma_a$.

Let us set $V:=g(B(w_0,\frac{\epsilon_0}{2}))$. Since $g$ is a subsequential limit of the chosen inverse branches of $F_a^{\circ kn}$ (on $B(w_0,\epsilon_0)$), it follows that infinitely many $F_a^{\circ kn}$ map $V$ into $B(w_0,\epsilon_0)\subset\Omega_a$. 

On the other hand, note that $V$ is a neighborhood of $g(w_0)\in\Gamma_a$, and hence $V$ intersects the tiling set $T_a^\infty$ non-trivially. But every point in $V\cap T_a^\infty$ is mapped to the fundamental tile $T_a^0$ by a sufficiently large iterate of $F_a$. Clearly, this contradicts the conclusion of the previous paragraph. This contradiction implies that the assumption $w_0\notin\overline{\left\{F_a^{\circ n}(0)\right\}_{n\geq0}}$ was false. Hence, $\partial U\setminus\{\frac14,\alpha_a\}\subset\overline{\left\{F_a^{\circ n}(0)\right\}_{n\geq0}}$. The result (for boundaries of Siegel disks) now follows by taking topological closure on both sides.

Since a Cremer point of $F_a$ is necessarily contained in $\Omega_a$, the proof of the corresponding statement for Cremer points is analogous to the Siegel case (compare \cite[Corollary~14.4]{M1new}).

3) Let $U\subset\sigma^{-1}(\overline{B}(a,r_a)^c)$ be the period two Fatou component of $F_a$ such that the $F_a^{\circ 2}$-orbit of every point in $U$ converges to $-\frac34$ asymptotic to the positive real direction at $-\frac34$ (the existence of this component was proved in Proposition~\ref{fatou_classification}(Case~3), compare Figure~\ref{petals_alpha}). It follows from the proof of the same proposition and \cite[Theorem~10.9]{M1new} that $F_a^{\circ 2}$, restricted to a small attracting petal $P_1$ in $U$, is conformally conjugate to translation by $+1$ on a right half-plane. One can now argue as in \cite[Theorem~10.15]{M1new} to conclude that the boundary of the maximal domain of definition of this conjugacy contains a critical point of $F_a^{\circ 2}$. It follows that $0\in U$.
\end{proof}

\begin{corollary}[Counting attracting/parabolic cycles]\label{att_para_count}
1) Let $a\neq-\frac{1}{12}$, and $F_a$ have an attracting (respectively, parabolic) cycle. Then the basin of attraction of this attracting (respectively, parabolic) cycle is equal to $\Int{K_a}$. In particular, $F_a$ has no other attracting, parabolic, or Siegel cycle.

2) Let $a=-\frac{1}{12}$. Then the basin of attraction of the singular point $-\frac34$ is equal to $\Int{K_a}$. In particular, $F_a$ has no attracting, parabolic, or Siegel cycle.
\end{corollary}
\begin{proof}
Both parts follow from Propositions~\ref{fatou_classification}, \ref{fatou_critical}, and the fact that $F_a$ has a unique critical point.
\end{proof}

\subsubsection{Counting irrationally neutral cycles}\label{siegel_cremer_count}
In Proposition~\ref{fatou_critical}, we showed that the boundary of a Siegel disk is contained in the post-critical set of $F_a$. This, however, is not sufficient to give an upper bound on the number of Siegel cycles of the map $F_a$ (a priori, the unique critical point of $F_a$ could be ``shared" by various Siegel cycles of $F_a$). We will end this subsection with an adaptation of Shishikura's results on rational maps \cite[Theorem~1]{Shi2} that will allow us to count the number of irrationally neutral cycles of $F_a$.

\begin{proposition}\label{non_repelling_count}
For any $a\in\C\setminus(-\infty,-\frac{1}{12})$, the map $F_a$ has at most one non-repelling cycle.
\end{proposition}
\begin{proof}[Sketch of the proof.]
Thanks to Proposition~\ref{att_para_count}, it suffices to prove that $F_a$ has at most one irrationally neutral cycle.

The main idea in the proof of \cite[Theorem~1]{Shi2} is to construct a quasi-regular perturbation of the original map in Siegel disks or in a small neighborhood of Cremer points such that all the irrationally neutral cycles of the original map become attracting for the modified map. This argument, being purely local, applies equally well to the map $F_a$, and produces a quasi-regular map $G_a$ with as many attracting cycles as the number of irrationally neutral cycles of $F_a$. Moreover, $G_a$ is a small perturbation of $F_a$, and agrees with $F_a$ outside the Siegel disks and away from the Cremer points of $F_a$. It follows that $G_a$ maps $\overline{B}(a,r_a)^c$ univalently onto $B(a,r_a)$, and $G_a^{-1}(\heartsuit)$ univalently onto $\heartsuit$. We can now adapt the proof of Proposition~\ref{schwarz_qcdef} to conclude that $G_a$ is quasiconformally conjugate to some member $F_{a'}$ of the family $\mathcal{S}$. Since attracting cycles are preserved under quasiconformal conjugacies, it follows from Proposition~\ref{att_para_count} that $G_a$ can have at most one attracting cycle. The conclusion follows.
\end{proof}

\begin{corollary}\label{siegel_cremer_complete}
1) If $F_a$ has a cycle of Siegel disks, then every Fatou component of $F_a$ eventually maps to this cycle of Siegel disks.

2) If $F_a$ has a Cremer point, then $\Int{K_a}=\emptyset$; i.e. $K_a=\Gamma_a$.
\end{corollary}

\begin{remark}
1) For maps with locally connected limit sets, the Snail Lemma rules out the presence of Cremer points.

2) A weaker upper bound on the number of neutral cycles of $F_a$ can be obtained by employing Fatou's classical perturbation arguments (see \cite[Lemma~13.2]{M1new})
\end{remark}

\subsection{Dynamical uniformization of the tiling set}\label{escape_set_group}

Near the super-attracting fixed point at infinity, the dynamics of a quadratic anti-polynomial is conjugate to the model map $\overline{z}^2$ via the B{\"o}ttcher coordinate (see \cite[Lemma~1]{Na1}). For the maps $F_a$, we have the following analogue of B{\"o}ttcher coordinates in the tiling set $T_a^\infty$ that conjugates $F_a$ to  the reflection map $\rho$ arising from the ideal triangle group $\mathcal{G}$ (see Section~\ref{ideal_triangle} for definitions of the ideal triangle $\Pi$ and the reflection map $\rho$). 

\subsubsection{Uniformization of the tiling set}\label{tiling_set_unif_sec} Note that for $a\in\C\setminus\left(\left(-\infty,-\frac{1}{12}\right)\cup\cC(\mathcal{S})\right)$, the critical value $\infty$ eventually escapes to the tile $T_a^0$. This leads to the following definition of depth for parameters in the escape locus.

\begin{definition}[Depth]\label{def_depth}
For $a\in\C\setminus\left(\left(-\infty,-\frac{1}{12}\right)\cup\cC(\mathcal{S})\right)$, the \emph{smallest} positive integer $n(a)$ such that $F_a^{\circ n(a)}(\infty)\in T_a^0$ is called the \emph{depth} of $a$.
\end{definition} 

\begin{proposition}\label{schwarz_group}
1) For $a\in\cC(\mathcal{S})$, $F_a:T_a^\infty\setminus\Int{T_a^0}\to T_a^\infty$ is conformally conjugate to $\rho:\D\setminus\Int{\Pi}\to\D$.

2) For $a\in\C\setminus\left(\left(-\infty,-\frac{1}{12}\right)\cup\cC(\mathcal{S})\right)$, $F_a:\displaystyle\bigcup_{n=1}^{n(a)} F_a^{-n}(T_a^0)\to\displaystyle\bigcup_{n=0}^{n(a)-1} F_a^{-n}(T_a^0)$ is conformally conjugate to $\rho:\displaystyle\bigcup_{n=1}^{n(a)} \rho^{-n}(\Pi)\to\displaystyle\bigcup_{n=0}^{n(a)-1}\rho^{-n}(\Pi)$.
\end{proposition}

\begin{figure}[ht!]
\centering
\includegraphics[width=0.4\linewidth]{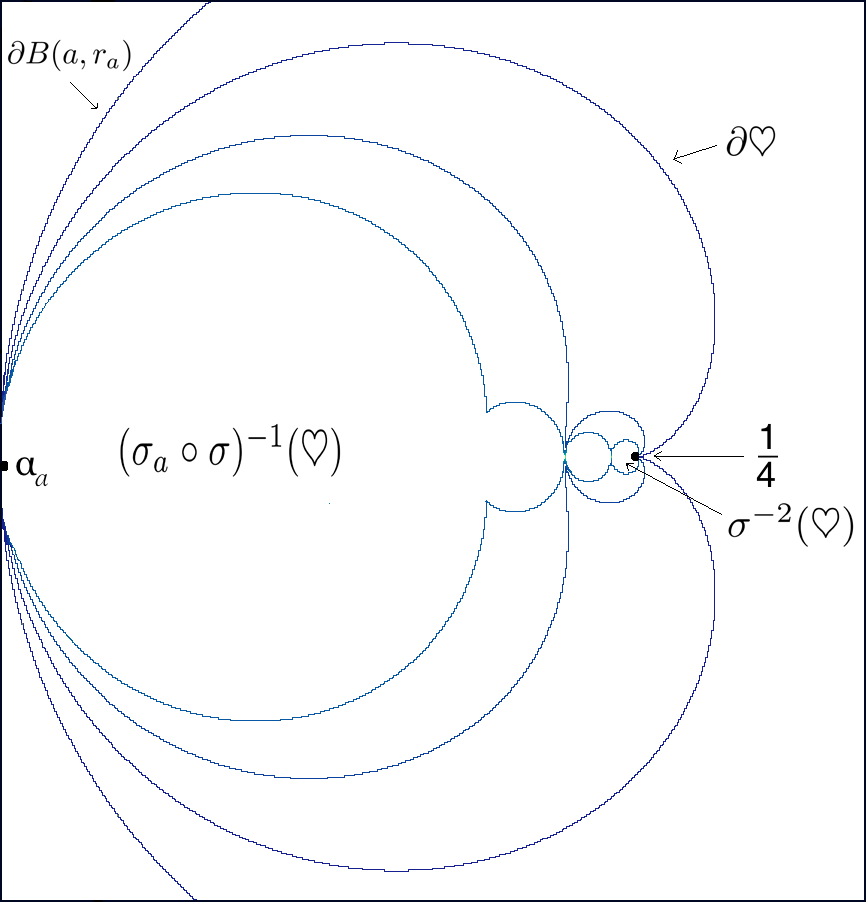}
\caption{The tail of the $0$-ray of $F_a$ is trapped inside $\sigma^{- 2}(\heartsuit)$, which is mapped univalently onto $\heartsuit$ by $\sigma^{\circ 2}$. Similarly, there is a representative of the $\frac13$-ray whose tail is contained in $(\sigma_a\circ\sigma)^{-1}(\heartsuit)$.}
\label{ray_landing_fig}
\end{figure}

\begin{proof}
Let $\psi_a$ be (the homeomorphic extension of) a conformal isomorphism between $T_a^0$ and $\Pi$ such that $\psi_a$ maps the upper (respectively lower) half of $\partial\heartsuit\setminus\{\alpha_a,\frac{1}{4}\}$ onto $\widetilde{C}_1$ (respectively $\widetilde{C}_3$) and $\partial B(a,r_a)\setminus\{\alpha_a\}$ onto $\widetilde{C}_2$ (compare Section~\ref{ideal_triangle}). 

1) Since $a\in\cC(\mathcal{S})$, the tiles of all rank of $F_a$ map diffeomorphically onto $T_a^0$ under iterates of $F_a$. Similarly, the tiles of the tessellation of $\D$ arising from the ideal triangle group $\mathcal{G}$ map diffeomorphically onto $\Pi$ under iterates of $\rho$. Furthermore, $F_a$ and $\rho$ act as identity maps on $\partial T_a^0$ and $\partial \Pi$ respectively. This allows us to lift $\psi_a$ to a conformal isomorphism from $T_a^\infty$ (which is the union of all iterated preimages of $T_a^0$ under $F_a$) onto $\D$ (which is the union of all iterated preimages of $\Pi$ under $\rho$). Note that the trivial actions of $F_a$ and $\rho$ on $\partial T_a^0$ and $\partial \Pi$ (respectively) ensure that the iterated lifts match on the boundaries of the tiles. By construction, the extended map conjugates $F_a:T_a^\infty\setminus \Int{T_a^0}\to T_a^\infty$ to $\rho:\D\setminus\Int{\Pi}\to\D$. 

2) Let $a\in\C\setminus\left(\left(-\infty,-\frac{1}{12}\right)\cup\cC(\mathcal{S})\right)$. By definition of depth of $a$, the unique critical point $0$ does not lie in $\bigcup_{n=1}^{n(a)} F_a^{-n}(T_a^0)$  (compare Figure~\ref{quickest_escape}). So each connected component of $ F_a^{-n}(T_a^0)$, $n=1,\cdots,n(a)$ (i.e. tiles of rank $1$ through $n(a)$) maps onto $T_a^0$ diffeomorphically under iterates of $F_a$. Hence, $\psi_a$ lifts to a conformal isomorphism from $\bigcup_{n=1}^{n(a)} F_a^{-n}(T_a^0)$ onto $\bigcup_{n=1}^{n(a)} \rho^{-n}(\Pi)$ such that it conjugates $F_a:\bigcup_{n=1}^{n(a)} F_a^{-n}(T_a^0)\to\bigcup_{n=0}^{n(a)-1} F_a^{-n}(T_a^0)$ to $\rho:\bigcup_{n=1}^{n(a)} \rho^{-n}(\Pi)\to\bigcup_{n=0}^{n(a)-1}\rho^{-n}(\Pi)$.
\end{proof}

\begin{definition}\label{psi_def}
The conformal conjugacy between $F_a$ and $\rho$, constructed in Proposition~\ref{schwarz_group}, is denoted by $\psi_a$.
\end{definition}

\begin{remark}\label{extension_external}
We say that a tile is degenerate if it contains the critical point $0$ or some of its iterated preimages. For $a\notin\cC(\mathcal{S})$, the conjugacy $\psi_a$ extends biholomorphically to all the non-degenerate tiles of $F_a$. 
\end{remark}

\subsubsection{Tiles and rays}\label{tiles_rays_sec} The conjugacy $\psi_a$ allows us to associate a symbol sequence to every non-degenerate tile of $F_a$. Indeed, every non-degenerate tile of $F_a$ is mapped biholomorphically by $\psi_a$ onto some tile $T^{i_1,\cdots,i_k}$ of $\D$ (see Definition~\ref{def_tiles}).

\begin{definition}[Coding of dynamical tiles]\label{tile_code}
i) Let $a\in\cC(\mathcal{S})$. For any $M$-admissible word $(i_1,\cdots,i_k)$, the dynamical tile $T_a^{i_1,\cdots,i_k}$ is defined as $$T_a^{i_1,\cdots,i_k}:=\psi_a^{-1}(T^{i_1,\cdots,i_k}).$$ 

ii) Let $a\in\C\setminus\left(\left(-\infty,-\frac{1}{12}\right)\cup\cC(\mathcal{S})\right)$. For any $M$-admissible word $(i_1,\cdots,i_k)$, the dynamical tile $T_a^{i_1,\cdots,i_k}$ is defined as above, provided $T^{i_1,\cdots,i_k}$ is in the image of $\psi_a$.
\end{definition}

We can use the map $\psi_a$ (see Definition~\ref{psi_def}) to define dynamical rays for the maps $F_a$.

\begin{definition}[Dynamical rays of $F_a$]\label{dyn_ray_schwarz}
The preimage of a $\mathcal{G}$-ray at angle $\theta$ under the map $\psi_a$ is called a $\theta$-dynamical ray of $F_a$.
\end{definition}

\begin{remark}\label{def_ray_well_defined}
1) Although a $\mathcal{G}$-ray at angle $\theta$ is not always uniquely defined, it is easy to see that the corresponding $\theta$-dynamical rays define the same prime end to $K_a$.

2) For $a\notin\cC(\mathcal{S})$, we say that a dynamical ray of $F_a$ at angle $\theta$ bifurcates if the corresponding sequence of tiles (of $F_a$) contains a degenerate tile. Otherwise, a dynamical ray is well-defined all the way to $\Gamma_a$.
\end{remark}

\subsubsection{Landing of preperiodic rays}\label{ray_landing_sec} As in the case for polynomials with connected Julia sets, we will now show that for a map $F_a$ with connected limit set, all dynamical rays at preperiodic angles (under $\rho$) land on $\Gamma_a$ (compare \cite[Expos{\'e}~VIII]{orsay}, \cite[Theorem~18.10]{M1new}).

We denote the set of all points of $\mathbb{T}$ that are (pre)periodic under $\rho$ by $\mathrm{Per}(\rho)$. 

\begin{proposition}[Landing of preperiodic rays]\label{per_rays_land}
Let $a\in\cC(\mathcal{S})$, and $\theta\in\mathrm{Per}(\rho)$. Then the following statements hold true.
\begin{enumerate}
\item The dynamical $\theta$-ray of $F_a$ lands on $\Gamma_a$. 

\item The $0$-ray of $F_a$ lands at $\frac{1}{4}$, while the $\frac{1}{3}$ and $\frac{2}{3}$-rays land at $\alpha_a$. No other ray lands at $\frac14$ or $\alpha_a$. 

\item The iterated preimages of the $0, \frac13,$ and $\frac23$-rays land at the iterated preimages of $\frac14$ and $\alpha_a$.

\item Let $\theta\in\mathrm{Per}(\rho)\setminus\displaystyle\bigcup_{n\geq0}\rho^{-n}(\lbrace 0,\frac{1}{3},\frac{2}{3}\rbrace)$. Then, the dynamical ray of $F_a$ at angle $\theta$ lands at a repelling or parabolic (pre)periodic point on $\Gamma_a$. 
\end{enumerate}
\end{proposition}
\begin{proof}
1) The idea of the proof is classical (see \cite[Theorem~18.10]{M1new}). But there are subtleties related to the presence of singular points that we need to address.

Without loss of generality, we can assume that $\theta$ is periodic of period $k$ under $\rho$. Then by Section~\ref{ideal_triangle}, there exists a shift-periodic element $\left(\overline{i_1, i_2, \cdots, i_k}\right)\in M^\infty$ such that  $Q(\left(\overline{i_1, i_2, \cdots, i_k}\right))=\theta$. Consider the sequence of points $$w_n:=\psi_a^{-1}\left(\left(\rho_{i_1}\circ\cdots\circ\rho_{i_k}\right)^{\circ n}(0)\right),\ n\geq1$$ on the $\theta$-dynamical ray of $F_a$. Let $\gamma_n$ be the sequence of arcs on the $\theta$-dynamical ray (of $F_a$) bounded by $w_n$ and $w_{n+1}$. Clearly, $F_a(\gamma_n)=\gamma_{n-1}$, and there is a fixed positive constant $c$ such that $\ell_{\mathrm{hyp}}(\gamma_n)=c$ for all $n\geq1$ (here, $\ell_{\mathrm{hyp}}$ denotes the hyperbolic length in $T_a^\infty$). We now consider two cases.

\noindent\textbf{Case a.} (The accumulation set of the $\theta$-ray is contained in $\cup_{k\geq 0} F_a^{-n}(\{\frac14, \alpha_a\})$.) Since the accumulation set of a ray is connected, and the set of iterated preimages of $\{\frac14, \alpha_a\}$ is totally disconnected, we conclude that the $\theta$-ray must land. 

\noindent\textbf{Case b.} (The accumulation set of the $\theta$-ray intersects $\Gamma_a\setminus\cup_{k\geq 0} F_a^{-n}(\{\frac14, \alpha_a\})$.) Let $w\in \Gamma_a\setminus\cup_{k\geq 0} F_a^{-k}(\{\frac14, \alpha_a\})$ be an accumulation point of the $\theta$-ray. In this case, the map $F_a^{\circ n}$ admits an anti-holomorphic extension in a neighborhood of $w$. Moreover, as the arcs $\gamma_n$ accumulate on the limit set $\Gamma_a$, it follows that the Euclidean lengths of these arcs tend to $0$ as $n\to\infty$. Hence, the arguments of the proof of \cite[Theorem~18.10]{M1new} apply to the present setting, and show that $F_a^{\circ n}(w)=w$. But the set of the fixed points of $F_a^{\circ k}$ in $\Gamma_a$ is totally disconnected (as the set of fixed points of the holomorphic second iterate $F_a^{\circ 2k}$ in $\Gamma_a$ is discrete). Finally, connectedness of the accumulation set of a ray implies that the $\theta$-ray lands.

2) Note that the landing point of a ray at a fixed angle must be a fixed point, and the only fixed points of $F_a$ on $\Gamma_a$ are $\frac14$ and $\alpha_a$. By the construction of $\psi_a$ and the dynamical rays of $F_a$, we see that the tail of the $0$-ray is contained in $\sigma^{-1}(\heartsuit)$. Hence, it must land at $\frac{1}{4}$. Similarly, the tails of the $\frac{1}{3}$ and $\frac{2}{3}$-rays lie outside $\sigma^{-1}(\heartsuit)$, and hence they must land at $\alpha_a$. 

Now let $\theta\in\R/\Z\setminus\{0,1/3,2/3\}$. We need to show that the dynamical $\theta$-ray does not land at $\frac14$ or $\alpha_a$. By way of contradiction, let us assume that the dynamical $\theta$-ray of $F_a$ lands at $\frac14$ or $\alpha_a$. We consider the two cases separately and arrive at a contradiction in each case. 

\noindent\textbf{Case a.} (The dynamical $\theta$-ray lands at $\frac14$.) Since $\frac14$ is a fixed point, the dynamical rays at angles $\rho^{\circ n}(\theta),\ n\geq 1,$ must also land at $\frac14$. Clearly, all these rays (including the $0$-ray) are locally contained in $\sigma^{-2}(\heartsuit)$ near $\frac14$ (see Figure~\ref{ray_landing_fig}). Note that $F_a^{\circ 2}\equiv\sigma^{\circ 2}$ on $\sigma^{-2}(\heartsuit)$. By Corollary~\ref{iterated_pre_image_cardioid}, $\sigma^{\circ 2}\vert_{\sigma^{-2}(\heartsuit)}$ extends to an orientation-preserving local homeomorphism in a neighborhood of $\frac14$. Hence, $F_a^{\circ 2}$ acts injectively on these rays, and preserves their circular order around $\frac14$. Applying the arguments of \cite[Lemma~2.3]{M2a} on $F_a^{\circ 2}=\sigma^{\circ 2}$ near $\frac14$, one concludes that $\theta$ is fixed under $\rho^{\circ 2}$. This is a contradiction since there is no angle of period $2$ under $\rho$, and $\theta$ is not a fixed angle.

\noindent\textbf{Case b.} (The dynamical $\theta$-ray lands at $\alpha_a$.) As $\alpha_a$ is a fixed point, the dynamical rays at angles $\rho^{\circ 2n}(\theta)$ ($n\geq 1$) also land at $\alpha_a$. Locally near $\alpha_a$, either all these rays are contained in $(\sigma_a\circ\sigma)^{-1}(\heartsuit)$ or all of them are contained in $(\sigma\circ\sigma_a)^{-1}(\overline{B}(a,r_a)^c)$. For definiteness, let us assume that they are all contained in $(\sigma_a\circ\sigma)^{-1}(\heartsuit)$ (see Figure~\ref{ray_landing_fig}). We can also choose a representative of the $\frac13$-ray in $(\sigma_a\circ\sigma)^{-1}(\heartsuit)$. Note that $F_a^{\circ 2}\equiv\sigma_a\circ\sigma$ on $(\sigma_a\circ\sigma)^{-1}(\heartsuit)$ (see Formula~(\ref{iterate_alpha})). By Propositions~\ref{double_asymp} and~\ref{fat_basilica_local_dyn}, $(\sigma_a\circ\sigma)\vert_{(\sigma_a\circ\sigma)^{-1}(\heartsuit)}$ has a holomorphic extension around $\alpha_a$ which is locally injective near $\alpha_a$. Therefore, $F_a^{\circ 2}$ acts injectively on the rays mentioned above (including the $\frac13$-ray), and preserves their circular order around $\alpha_a$. Applying again \cite[Lemma~2.3]{M2a} on $F_a^{\circ 2}=\sigma_a\circ\sigma$ near $\alpha_a$, we conclude that $\theta$ must be fixed under $\rho^{\circ 2}$ which contradicts the assumption that $\theta\notin\{0,1/3,2/3\}$. 

This shows that no ray lands at the singular points other than the $0,\frac13$, and $\frac23$-rays.

3) Since the iterated preimages of the singular points (excluding the singular points themselves) are mapped by locally holomorphic/anti-holomorphic iterates of $F_a$ to the singular points, it follows that the dynamical rays at angles in $\bigcup_{n\geq0}\rho^{-n}(\{0,1/3,2/3\})\setminus\{0,1/3,2/3\}$ land at iterated preimages of $\frac14$ and $\alpha_a$ (excluding $\frac14$ and $\alpha_a$).

4) Finally, let $\theta\in\mathrm{Per}(\rho)\setminus\bigcup_{n\geq0}\rho^{-n}(\{0,1/3,2/3\})$. Once again, we can assume that $\theta$ is periodic of period $k>1$ under $\rho$. By the previous parts of the proposition, the landing point of the $\theta$-ray of $F_a$ is not an iterated preimage of the singular points. Therefore, $F_a^{\circ 2k}$ is defined in a small neighborhood of the landing point. An application of the classical Snail Lemma now shows that the landing point of the $\theta$-dynamical ray is repelling or parabolic (compare \cite[Lemma~18.9]{M1new}). 
\end{proof}

Let us also state a converse to the previous proposition.

\begin{proposition}[Repelling and parabolic points are landing points of rays]\label{rep_para_landing_point}
Let $a\in\cC(\mathcal{S})$. Then, every repelling and parabolic periodic point of $F_a$ is the landing point of finitely many (at least one) dynamical rays. Moreover, all these rays have the same period under $F_a^{\circ 2}$.
\end{proposition}
\begin{proof}
The proof of \cite[Theorem~24.5, Theorem~24.6]{L6} applies mutatis mutandis to our situation.
\end{proof}

\begin{remark}
Unlike in the holomorphic situation, it is not true that all rays landing at a periodic point have the same period under $F_a$ (this is already false for quadratic anti-polynomials, see \cite[Theorems~2.6,~3.1]{Sa1}). The combinatorics of dynamical rays landing at periodic points of $F_a$ will be discussed in detail in \cite{LLMM2}.
\end{remark}

\subsubsection{Locally connected limit sets, and density of repelling cycles}
Similar to the situation of rational maps or Kleinian groups, we have the following statements.

\begin{proposition}[Density of repelling cycles]\label{density_loc_conn}
Let $a\in\cC(\mathcal{S})$, and $\Gamma_a$ be locally connected. Then the following hold true.

1) Iterated preimages of $\frac14$ under $F_a$ are dense in $\Gamma_a$.

2) Repelling periodic points of $F_a$ are dense in $\Gamma_a$.
\end{proposition}
\begin{proof}
1) Since $\Gamma_a=\partial T_a^\infty$ is connected and locally connected, the inverse of the Riemann uniformization $\psi_a^{-1}:\D\to\ T_a^\infty$ that conjugates $\rho$ to $F_a$ (constructed in Proposition~\ref{schwarz_group}) extends continuously to $\mathbb{T}$, and semi-conjugates $\rho\vert_{\mathbb{T}}$ to $F_a\vert_{\Gamma_a}$. Abusing notations, let us denote this continuous extension by $\psi_a^{-1}$. Note that $\psi_a^{-1}(0)=\frac14$. Moreover, the iterated preimages of $0$ under $\rho$ are dense in $\mathbb{T}$, and $\psi_a^{-1}$ maps these points to the iterated preimages of $\frac14$ under $F_a$. It now follows that $\bigcup_{k=0}^\infty F_a^{-k}(\{1/4\})$ is dense in $\Gamma_a$.

2) Since periodic points of $\rho$ are dense in $\mathbb{T}$, one can argue as in the first part to conclude that periodic points of $F_a$ are dense in $\Gamma_a$. The result now immediately follows from Proposition~\ref{non_repelling_count}. 
\end{proof}

\begin{remark}
1) Note that for $w\in\Gamma_a$, there exists no $\epsilon>0$ such that infinitely many forward iterates of $F_a$ are defined on $B(w_0,\epsilon)$. This prohibits us from using the standard normal family argument to prove density of iterated preimages or density of repelling cycles in $\Gamma_a$. In the locally connected case, we circumvent this problem by working with the topological model $\rho\vert_{\mathbb{T}}$ for $F_a\vert_{\Gamma_a}$.

2) Under the local connectivity assumption, the density statement is also true for the preimages of the double point $\alpha_a$, and can be proved similarly.
\end{remark}

\subsubsection{Landing map and lamination}\label{landing_map_lami_sec} Let $a\in\cC(\mathcal{S})$. Proposition~\ref{per_rays_land} allows us to define a landing map $L_a:\mathrm{Per}(\rho)\to \Gamma_a$ that associates to every (pre)periodic angle $\theta$ (under $\rho$) the landing point of the $\theta$-dynamical ray of $F_a$.

\begin{definition}[Pre-periodic lamination of $F_a$]\label{def_preper_lami}
For $a\in\cC(\mathcal{S})$, the preperiodic lamination of $F_a$ is defined as the equivalence relation on $\mathrm{Per}(\rho)\subset\R/\Z$ such that $\theta, \theta'\in\mathrm{Per}(\rho)$ are related if and only if $L_a(\theta)=L_a(\theta')$. We denote the preperiodic lamination of $F_a$ by $\lambda(F_a)$.
\end{definition}

The next definition and the subsequent proposition relates preperiodic laminations of $F_a$ to rational laminations of quadratic anti-polynomials (see \cite[\S 2.2]{LLMM2}). This connection will be crucial in \cite{LLMM2}, where we will establish a combinatorial bijection between the geometrically finite parameters of $\mathcal{S}$ and those of $\mathcal{L}$. Recall from Subsection~\ref{conjugacy_anti_doubling_rho_sec} that $\mathcal{E}$ is a circle homeomorphism that conjugates $\rho$ to $m_{-2}$.

\begin{definition}[Push-forward/pullback of laminations]\label{push_lami}
The push-forward $\mathcal{E}_{\ast}(\lambda(F_a))$ of the preperiodic lamination of $F_a$ is defined as the image of $\lambda(F_a)\subset\mathrm{Per}(\rho)\times\mathrm{Per}(\rho)$ under $\mathcal{E}\times\mathcal{E}$. Clearly, $\mathcal{E}_{\ast}(\lambda(F_a))$ is an equivalence relation on $\Q/\Z$. Similarly, the pullback $\mathcal{E}^{\ast}(\lambda(f_c))$ of the rational lamination of a quadratic anti-polynomial $f_c$ is defined as the preimage of $\lambda(f_c)\subset\Q/\Z\times\Q/\Z$ under $\mathcal{E}\times\mathcal{E}$.
\end{definition}

\begin{proposition}[Properties of preperiodic laminations]\label{prop_preper_lami}
Let $a\in\cC(\mathcal{S})$, and $\lambda(F_a)$ be the preperiodic lamination associated with $F_a$. Then, $\lambda(F_a)$ satisfies the following properties.
\begin{enumerate}
\item $\lambda(F_a)$ is closed in $\mathrm{Per}(\rho)\times\mathrm{Per}(\rho)$.

\item Each $\lambda(F_a)$-equivalence class $A$ is a finite subset of $\mathrm{Per}(\rho)$.

\item If $A$ is a  $\lambda(F_a)$-equivalence class, then $\rho(A)$ is also a $\lambda(F_a)$-equivalence class.

\item If $A$ is a  $\lambda(F_a)$-equivalence class, then $A\mapsto\rho(A)$ is consecutive reversing.

\item $\lambda(F_a)$-equivalence classes are pairwise unlinked.
\end{enumerate}

Consequently, the push-forward $\mathcal{E}_{\ast}(\lambda(F_a))$ is a formal rational lamination (under $m_{-2}$).
\end{proposition}

\begin{proof}
The proof of the properties of $\lambda(F_a)$ are analogous to the proof of the corresponding statements for rational laminations of polynomials (see \cite[Theorem~1.1]{kiwi}). The last statement follows from the fact that $\mathcal{E}$ is a homeomorphism of the circle that conjugates $\rho$ to $m_{-2}$.  
\end{proof}

\subsubsection{Cantor dynamics}\label{cantor_sec} We will now study the topological structure of non-escaping sets for maps $F_a$ outside the connectedness locus $\cC(\mathcal{S})$.

\begin{proposition}[Cantor outside $\cC(\mathcal{S})$]\label{cantor_outside}
If $a\notin\cC(\mathcal{S})$, then $K_a$ is a Cantor set.
\end{proposition}

\begin{proof}
The idea of the proof is similar to that of \cite[Expos{\'e}~III, \S 1, Proposition~1]{orsay}. However, lack of uniform contraction of the inverse branches of $F_a$ (around the singular points) adds some subtlety to the situation.

Let $a\notin\cC(\mathcal{S})$ and $n(a)$ be the depth of $a$ (see Definition~\ref{def_depth}). Then, some tile of rank $(n(a)+1)$ contains the critical point $0$, and disconnects $K_a$ (compare Figure~\ref{various_limit_sets} (right)). This is equivalent to saying that $\widehat{\C}\setminus E_a^{n(a)+1}$ has two connected components (where $E^k_a$ is the union of tiles of rank $\leq k$). $F_a$ maps each of these two connected components injectively onto $\widehat{\C}\setminus E_a^{n(a)}$. Let us denote the inverse branches of $F_a$ defined on $\widehat{\C}\setminus E_a^{n(a)}$ by $g_1$ and $g_2$. Clearly, every connected component of $K_a$ is the nested intersection of sets of the form $(g_{i_1}\circ g_{i_2}\circ\cdots\circ g_{i_k})(\widehat{\C}\setminus E_a^{n(a)})$, where $(i_1,i_2,\cdots)\in\{1,2\}^{\N}$.

The fact that the connected component of $K_a$ containing $\frac{1}{4}$ is a singleton follows from Corollary~\ref{iterated_pre_image_cardioid}, and the same statement about $\alpha_a$ follows from shrinking of petals in Relation~(\ref{petals_shrink}). Hence, every connected component of $K_a$ containing a point of $\bigcup_{k\geq 0}F_a^{-k}(\{1/4,\alpha_a\})$ is a singleton.

On $\Int{(\widehat{\C}\setminus E_a^{n(a)})}$, the sequence of anti-conformal maps $\{g_{i_1}, g_{i_1}\circ g_{i_2},\cdots\}$ forms a normal family with a constant limit function. It follows that a connected component of $K_a$ not intersecting $\frac14,\alpha_a$ and their preimages is also a singleton. 

This proves that every connected component of $K_a$ is a singleton; i.e. $K_a$ is a Cantor set. 
\end{proof}

\begin{proof}[Proof of Theorem~\ref{non_escaping_connected_intro}]
This follows from Proposition~\ref{schwarz_group} and Proposition~\ref{cantor_outside}.
\end{proof}

\section{Geometrically finite maps in the circle-and-cardioid family}\label{geom_fin_sec}

In this section, we will take a closer look at the dynamical properties of a simpler subclass of maps in $\mathcal{S}$. More precisely, we will focus on geometrically finite maps; i.e. maps with an attracting or parabolic cycle, and maps for which the critical point is non-escaping and strictly preperiodic.

Recall that for any $a\in\C\setminus(-\infty,-\frac{1}{12})$, the map $F_a$ has two fixed points $\alpha_a$ and $\frac{1}{4}$ such that $F_a$ does not admit an anti-holomorphic extension in a neighborhood of these points. We analyzed the dynamical behavior of $F_a$ near these fixed points in Proposition~\ref{cusp_asymp} and Proposition~\ref{double_asymp}. In what follows, we will refer to these fixed points as \emph{singular points}. On the other hand, the terms \emph{cycle/periodic orbit} will be reserved for periodic points of period greater than one (these are contained in $\Omega_a$).

\subsection{Hyperbolic and parabolic maps}\label{hyp_para_sec}

$F_a$ is called hyperbolic (respectively, parabolic) if it has a (super-)attracting (respectively, parabolic) cycle. It is easy to see that a (super-)attracting cycle of $F_a$ belongs to the interior of $K_a$, and a parabolic cycle lies on the boundary of $K_a$ (e.g. by \cite[\S 5, Theorem~5.2]{M1new}). Moreover, a parabolic periodic point necessarily lies on the boundary of a Fatou component (i.e. a connected component of $\Int{K_a}$) that contains an attracting petal of the parabolic germ such that the forward orbit of every point in the component converges to the parabolic cycle through the attracting petals.

\subsubsection{First properties and examples}\label{first_prop_examples_sec_1}

By Corollary~\ref{att_para_count}, a hyperbolic or parabolic map $F_a$ has a unique (super-)attracting or parabolic cycle, and the basin of attraction of this cycle coincides with all of $\Int{K_a}$. Moreover, Proposition~\ref{fatou_critical} implies that $a\in\cC(\mathcal{S})\setminus\{-\frac{1}{12}\}$.

\noindent Examples. i) The simplest hyperbolic map in $\cC(\mathcal{S})$ is $a=0$. The critical orbit of the corresponding map is given by $0\leftrightarrow\infty$, so the map has a super-attracting cycle. This map is the analogue of the Basilica anti-polynomial $\overline{z}^2-1$.

ii) The map $F_a$ with $a=\frac{3}{16}$ (respectively, $a=\frac{16\sqrt{3}-3}{132}$) has a super-attracting (respectively, parabolic) $3$-cycle (by Equation~(\ref{schwarz_cardioid})). This map is the analogue of the Airplane anti-polynomial $\overline{z}^2-1.7549\cdots$ (respectively, the parabolic Airplane anti-polynomial $\overline{z}^2-\frac74$).

\subsubsection{Radial limit set of hyperbolic/parabolic maps}\label{radial_hyp_para_sec} The following notion, which plays an important role in our study, was first introduced in \cite{L1} and later formalized in \cite{Urb1,Mcm}. (The concept was tacitly used in the proof of Lemma~\ref{deltoid_lc}.)

\begin{definition}[Radial limit set]\label{def_radial}
A point $w\in \Gamma_a$ is called a \emph{radial point} if there exists $\delta > 0$ and an infinite sequence of positive integers $\lbrace n_k \rbrace$ such that there exists a well-defined inverse branch of $F_a^{\circ n_k}$ defined on $B(F_a^{\circ n_k}(w), \delta)$ sending $F_a^{\circ n_k}(w)$ to $w$ for all $k$. The set of all radial points of $\Gamma_a$ is called the \emph{radial limit set} of $F_a$.
\end{definition}

It is not hard to see that for a radial point $w$, the sequence of inverse branches of $F_a^{\circ n_k}$ mapping $F_a^{\circ n_k}(w)$ to $w$ are eventually contracting; in other words, we have $\displaystyle\lim_{k\rightarrow \infty}(F_a^{\circ n_k})'(w)=\infty$. Thus, $F_a$ has a strong expansion property at the radial points of $\Gamma_a$. Furthermore, one can use the inverse branches of $F_a^{\circ n_k}$ to transfer information from moderate scales to microscopic scales at radial points.

The next proposition gives an explicit description of the radial limit set of hyperbolic and parabolic maps.

\begin{proposition}\label{radial_parabolic}
1) Let $F_a$ have a (super-)attracting cycle. Then, the radial limit set of $F_a$ is equal to $\Gamma_a\setminus\displaystyle  \bigcup_{k=0}^{\infty} F_a^{-k}(\{1/4,\alpha_a\})$.

2) Let $F_a$ have a parabolic cycle $\mathcal{O}$. Then, the radial limit set of $F_a$ is equal to $\Gamma_a\setminus \displaystyle\bigcup_{k=0}^{\infty} F_a^{-k}(\mathcal{O}\cup\{1/4,\alpha_a\})$. 
\end{proposition}
\begin{proof}
1) If $w$ is an iterated preimage of a singular point, then it clearly does not belong to the radial limit set. 

Let $w\in\Gamma_a\setminus\bigcup_{k=0}^{\infty} F_a^{-k}(\{1/4,\alpha_a\})$. Recall that by Proposition~\ref{wedge_at_cardioid_sing}, the limit set is locally contained in the ``repelling petals" at $\frac14$ and $\alpha_a$. Hence, no point in $\Gamma_a$ can non-trivially converge to one of the singular points. Our assumption on $w$ now implies that there exists $\delta>0$ such that infinitely many forward iterates $F_a^{\circ n_k}(w)$ stay outside $B(\{\frac{1}{4},\alpha_a\},2\delta)$. Since the critical orbit of $F_a$ converges to an attracting cycle (i.e. the post-critical set stays at a positive distance away from $\Gamma_a$), it follows $B(F_a^{\circ n_k}(w),\delta)\subset\Omega_a\setminus\left\{F_a^{\circ r}(0)\right\}_{r\geq0}$ for all $k\geq1$ (possibly after choosing a smaller $\delta>0$). Therefore, for each $k\geq1$, there is a well-defined inverse branch $F_a^{-n_k}:B(F_a^{\circ n_k}(w),\delta)\to\C$ that sends $F_a^{\circ n_k}(w)$ to $w$.

2) In the parabolic situation, the only difference is that the critical orbit converges to the parabolic cycle $\mathcal{O}$ which lies on the limit set. In particular, there is no uniform neighborhood of $\mathcal{O}$ where infinitely many inverse branches of $F_a^{\circ n}$ are defined. Hence, the radial limit set is disjoint from $\bigcup_{k=0}^{\infty} F_a^{-k}(\mathcal{O}\cup\{1/4,\alpha_a\})$.

On the other hand, if $w\in\Gamma_a\setminus\bigcup_{k=0}^{\infty} F_a^{-k}(\mathcal{O}\cup\{1/4,\alpha_a\})$, then its forward orbit can neither converge to the parabolic cycle nor converge to the singular points. Hence, infinitely many forward iterates of $F_a^{\circ n_k}(w)$ stay away from the fundamental tile and the closure of the post-critical set. Therefore, there exists $\delta>0$ and well-defined inverse branches $F_a^{-n_k}:B(F_a^{\circ n_k}(w),\delta)\to\C$ that send $F_a^{\circ n_k}(w)$ to $w$ (for all $k\geq1$).
\end{proof}

\subsubsection{Topological and analytic properties}\label{top_anal_sec_1} Note that the Julia set of a hyperbolic or parabolic polynomial has zero area. Using the description of radial limit sets given in Proposition~\ref{radial_parabolic}, we can now prove the same statement for hyperbolic and parabolic maps in $\mathcal{S}$. 

\begin{corollary}\label{meas_zero}
If $F_a$ is hyperbolic or parabolic, then $\Gamma_a$ has zero area.
\end{corollary}
\begin{proof}
One can follow the arguments of \cite[Proposition~25.23]{L6} or \cite[Theorem~3.8]{Urb1} to show that no radial limit point is a point of Lebesgue density for $\Gamma_a$. The description of the radial limit set of $F_a$ given in Proposition~\ref{radial_parabolic} now yields the result.
\end{proof}

Recall that the limit set of a hyperbolic or parabolic map is connected. The following proposition shows that the limit set of such maps is also locally connected.

\begin{proposition}\label{geom_fin_local_conn}
If $F_a$ is hyperbolic or parabolic, then $\Gamma_a$ is locally connected.
\end{proposition}
\begin{proof}
The proof of \cite[Expos{\'e}~X, Theorem~1]{orsay} can be adapted for our situation, we only indicate the necessary modifications.

Recall that since $F_a$ is hyperbolic or parabolic, the non-escaping set $K_a$ is connected and every tile is unramified. In particular, there is a unique tile of rank $0$, and $3\cdot2^{n-1}$ tiles of rank $n$ (for $n\geq1$). We denote the union of all tiles of rank$~\leq n$ by $E_a^n$. Then $\partial E_a^n$ is a closed curve contained in the union of the tiling set $T_a^\infty$ and the iterated preimages of the singular points $\frac14$ and $\alpha_a$. Moreover, $F_a^{\circ 2}$ is a four-to-one covering from $\partial E_a^{n+2}$ onto $\partial E_a^n$. Choose a parametrization $\gamma_1:\mathbb{T}\to\partial E_a^1$ such that $\gamma_1(0)=\frac14$, and $\gamma_1(\theta)=\gamma_1(\theta')$ for $\theta\neq\theta'$ only if $\gamma_1(\theta)$ is a singular point of $T_a$ or one of its preimages. Now define parametrizations $\gamma_n:\mathbb{T}\to\partial E_a^{2n-1}$ inductively by lifting $\gamma_{n-1}:\mathbb{T}\to\partial E_a^{2n-3}$ (by the covering maps $F_a^{\circ 2}:\partial E_a^{2n-1}\to\partial E_a^{2n-3}$ and $\rho^{\circ 2}:\mathbb{T}\to\mathbb{T}$) such that $\gamma_{n}(0)=\frac14$.

Let $B$ be a neighborhood of the attracting cycle $A_{-}$ of $F_a$ (if $F_a$ is hyperbolic) or the union of attracting petals at the parabolic cycle $A_0$ of $F_a$ (if $F_a$ is parabolic) such that the post-critical set $\{F_a^{\circ n}(0)\}_{n\geq0}$ is contained in $B$. We define a compact set $X:=\widehat{\C}\setminus\left(\Int{E_a^1}\cup B\right).$ Let $Y$ be the set of singular points on $\partial E_a^1$; i.e. $Y:= \{\frac14,\alpha_a,\sigma^{-1}(\alpha_a),\sigma_a^{-1}(\frac14)\}$. Then, $X$ satisfies the following properties (compare \cite[Expos{\'e}~X, Proposition~1]{orsay}).

\begin{enumerate}
\item $X\cap\{F_a^{\circ n}(0)\}_{n\geq0}=\emptyset$, and in the hyperbolic case, $X\cap A_{-}=\emptyset$;

\item $A_0\cup Y\subset\partial X$ (where $A_0=\emptyset$ in the hyperbolic case);

\item Each of the sets $\Gamma_a, F_a^{-2}(X),$ and $\partial E_a^n$ ($n\geq1$) is contained in $\Int{X}\cup A_0\cup Y$ (where $A_0=\emptyset$ in the hyperbolic case).
\end{enumerate}

Let $U:=\Int{X}$, and $\mu_U$ the hyperbolic metric on each connected component of $U$. Fixing $\epsilon>0$ sufficiently small and $M>0$ sufficiently large, we define a metric $\mu$ on $U\cup N_\epsilon(A_0\cup\{\frac14,\alpha_a\})$ as follows
\begin{equation*}
\mu:=\left\{\begin{array}{ll}
                     \mathrm{inf}(\mu_U, M\vert dw\vert) & \mbox{on}\ U\cap N_\epsilon(A_0\cup\{\frac14,\alpha_a\}), \\
                      \mu_U & \mbox{on}\ U\setminus N_\epsilon(A_0\cup\{\frac14,\alpha_a\}),\\
                      M\vert dw\vert & \mbox{on}\ N_\epsilon(A_0\cup\{\frac14,\alpha_a\})\setminus U.
                                          \end{array}\right. 
\end{equation*}

Finally, let us set $X':=F_a^{-2}(X)$. The arguments of \cite[Expos{\'e}~X, Propositions~4,5]{orsay} now apply verbatim to show that if $M$ is large enough, then $\{\gamma_n\}_n$ is a Cauchy sequence in $C(\mathbb{T},X')$ equipped with the metric of uniform convergence for the distance $d_\mu$ on $X'$ associated with the Riemannian metric $\mu$. It follows that the sequence converges uniformly for $d_\mu$, and hence also for the Euclidean distance (as $X'$ is compact). The limit of the sequence $\{\gamma_n\}_n$ is the Caratheodory loop of $\Gamma_a$, proving local connectivity of $\Gamma_a$.
\end{proof}

The following improvement will be important in the mating discussion in the sequel.

\begin{theorem}\label{tiling_set_John_1}
If $F_a$ is hyperbolic, then the tiling set $T_a^\infty$ is a John domain. Hence, $\Gamma_a$ is conformally removable.
\end{theorem}
\begin{proof}
Since the post-critical set of a hyperbolic map $F_a$ is bounded away from its limit set, the proof is completely analogous to that of Theorem~\ref{tiling_set_John}.
\end{proof}

\subsection{Misiurewicz maps}\label{misi_maps}

A parameter $a\in\C\setminus(-\infty,-\frac{1}{12})$ is called Misiurewicz if the critical point $0$ of $F_a$ is non-escaping and strictly preperiodic.

\subsubsection{First properties}\label{first_prop_sec_2}

\begin{proposition}\label{misi_first_prop}
Let $F_a$ be a Misiurewicz map. 

1) The critical point of $F_a$ eventually lands on a repelling cycle or on one of the singular points $\frac14, \alpha_a$. Consequently, $a\in\cC(\mathcal{S})\setminus\{-\frac{1}{12}\}$.

2) The non-escaping set $K_a$ of $F_a$ has empty interior; i.e. $K_a=\Gamma_a$. 
\end{proposition}
\begin{proof}
1) It follows from Proposition~\ref{fatou_critical} that such a map cannot have a (super-)attracting or neutral cycle. Hence, the critical point $0$ of a Misiurewicz map must eventually land on a repelling cycle or on one of the singular points $\frac14, \alpha_a$. In particular, the entire forward orbit of $0$ (which is finite) stays in $K_a$; and hence $a\in\cC(\mathcal{S})\setminus\{-\frac{1}{12}\}$. 

2) Since the critical point of $F_a$ is strictly preperiodic, Proposition~\ref{fatou_classification} and Proposition~\ref{fatou_critical} rule out the existence of interior components of $K_a$; hence $\Int{K_a}=\emptyset$. Thus, $K_a=\Gamma_a$. 
\end{proof}

\subsubsection{Radial limit set of Misiurewicz maps}\label{radial_sec_misi} The next proposition describes the radial limit set (see Definition~\ref{def_radial}) for Misiurewicz maps. 

\begin{proposition}\label{radial_misi}
Let $a$ be a Misiurewicz parameter. The radial limit set of $F_{a}$ is equal to $\Gamma_a\setminus\displaystyle  \bigcup_{k=0}^{\infty} F_a^{-k}(\{0,1/4,\alpha_a\})$.
\end{proposition}
\begin{proof}
Since critical values are obstructions to the existence of inverse branches of anti-holomorphic maps (on a full disk neighborhood), and $F_a$ does not have anti-holomorphic extensions in neighborhoods of the singular points, it follows that neither a precritical point nor an iterated preimage of a singular point can be a radial limit point of $F_a$.

Since the critical value $\infty$ eventually falls on a repelling cycle, points in the forward orbit of $\infty$ evidently satisfy the radial limit set condition. 

Finally, if $w\in\Gamma_a$ does not belong to the grand orbit of the critical point and the singular points, then infinitely many forward iterates of $w$ stay away from the singular points and the post-critical set. The arguments of Proposition~\ref{radial_parabolic} can now be applied verbatim to conclude that such a point belongs to the radial limit set.
\end{proof}

\subsubsection{Topological and analytic properties}\label{top_anal_sec_2} We now record a couple of basic topological and analytic properties of the non-escaping set of a Misiurewicz map.

\begin{proposition}\label{misi_measure_zero}
Let $a$ be some Misiurewicz parameter. Then, the area of $K_a$ is zero.
\end{proposition}
\begin{proof}
Note that since $K_a$ has empty interior, we only need to show that $\Gamma_a$ has measure zero. This follows from the description of the radial limit set of $F_a$ given in Proposition~\ref{radial_misi}, and the fact that no radial limit point is a point of Lebesgue density for $\Gamma_a$ (compare Proposition~\ref{meas_zero}).
\end{proof}

\begin{proposition}\label{misi_loc_conn}
Let $a$ be some Misiurewicz parameter. Then, $\Gamma_a$ is locally connected. Consequently, all external dynamical rays of $F_a$ land on $\Gamma_a$. Moreover, $\Gamma_a$ is a dendrite.
\end{proposition}
\begin{proof}
The proof is similar to that of Proposition~\ref{geom_fin_local_conn} with one important modification (since $\Gamma_a$ contains the critical orbit).

Let $X:=\widehat{\C}\setminus\Int{E_a^1}$. Then $X$ satisfies the following properties
\begin{enumerate}
\item $\{\frac14,\alpha_a,\sigma^{-1}(\alpha_a),\sigma_a^{-1}(\frac14)\}\subset\partial X$; 

\item $\{F_a^{\circ n}(0)\}_{n\geq0}\subset\Int{X}$,

\item $\Gamma_a,\gamma_n, F_a^{-2}(X)\subset\Int{X}\cup\{\frac14,\alpha_a,\sigma^{-1}(\alpha_a),\sigma_a^{-1}(\frac14)\}$.
\end{enumerate}

Let us set $U:=\Int{X}$, and define a function $v:U\to\mathbb{Z}^+$ as
 $$v(w)=\left\{\begin{array}{ll}
                    2 & \mbox{if}\ w\in\{F_a^{\circ k}(0)\}_{k\geq1}, \\
                    1 & \mbox{otherwise}.
                                          \end{array}\right.$$ 
            
Then, $v(F_a(w))=\mathrm{deg}_w(F_a)\cdot v(w)$ for each $w$, where $\mathrm{deg}_w(F_a)$ denotes the local degree of $F_a$ at $w$. Let $U^*$ be a ramified covering of $U$, with ramification degree equal to $v(w)$ for every point above $w$, and $\widetilde{U}$ the universal covering of $U^*$. Let $\pi:\widetilde{U}\to U$ be the projection map. 

Denote by $\mu_{\widetilde{U}}$ the hyperbolic metric on $\widetilde{U}$ and $\mu_U$ the admissible Riemannian metric on $U$ such that $\pi:\widetilde{U}\to U$ is a local isometry. Now, fixing $\epsilon>0$ sufficiently small and $M>0$ sufficiently large, we define a metric $\mu$ on $U\cup N_\epsilon(\{\frac14,\alpha_a\})$ as follows
\begin{equation*}
\mu:=\left\{\begin{array}{ll}
                     \mathrm{inf}(\mu_U, M\vert dw\vert) & \mbox{on}\ U\cap N_\epsilon(\{\frac14,\alpha_a\}), \\
                      \mu_U & \mbox{on}\ U\setminus N_\epsilon(\{\frac14,\alpha_a\}),\\
                      M\vert dw\vert & \mbox{on}\ N_\epsilon(\{\frac14,\alpha_a\})\setminus U.
                                          \end{array}\right. 
\end{equation*}

Once again, the arguments of \cite[Expos{\'e}~X, Propositions~4,5]{orsay} show that if $M$ is large enough, then $\{\gamma_n\}_n$ (see the proof of Proposition~\ref{geom_fin_local_conn} for the definition of the curves $\gamma_n$) is a Cauchy sequence in $C(\mathbb{T},X')$ (where $X':=F_a^{-2}(X)$) equipped with the metric of uniform convergence for the distance $d_\mu$ on $X'$ associated with the Riemannian metric $\mu$. It follows that the sequence converges uniformly for $d_\mu$, and hence also for the Euclidean distance (as $X'$ is compact). The limit of the sequence $\{\gamma_n\}_n$ is the Caratheodory loop of $\Gamma_a$, proving local connectivity of $\Gamma_a$.

Landing of dynamical rays follows from local connectivity of $\Gamma_a$ and Caratheodory's theorem. The last statement follows from the fact that $K_a$ is full and has empty interior (see Proposition~\ref{misi_first_prop}).
\end{proof}

Theorem~\ref{geom_finite_limit_set}, which was announced in the introduction, follows from our analysis of geometrically finite maps.

\begin{proof}[Proof of Theorem~\ref{geom_finite_limit_set}]
1) Local connectivity of the limit sets of geometrically finite maps follows from Propositions~\ref{geom_fin_local_conn} and~\ref{misi_loc_conn}. The statement on zero area of limit sets follows from Proposition~\ref{meas_zero} and~\ref{misi_measure_zero}.

2) This follows from Propositions~\ref{density_loc_conn},~\ref{geom_fin_local_conn} , and~\ref{misi_loc_conn}.
\end{proof}

\subsubsection{Examples of Misiurewicz maps}\label{examples_sec_2}

\noindent Examples. i) The simplest Misiurewicz parameter in $\cC(\mathcal{S})$ is $a=\frac{1}{4}$. The critical orbit of the corresponding map is $0\mapsto\infty\mapsto\frac{1}{4}\circlearrowleft$. Note that $F_a:\left[-\infty,\frac14\right]\to\left[-\infty,\frac14\right]$ is a two-to-one surjective map. Hence, $\left[-\infty,\frac14\right]$ is a completely invariant subset of $\Gamma_a$.

By Propositions~\ref{density_loc_conn} and~\ref{misi_loc_conn}, the preimages of $\frac14$ are dense in $\Gamma_a$. It follows that $\Gamma_a=\left[-\infty,\frac{1}{4}\right]$. This map is the analogue of the Chebyshev anti-polynomial $\overline{z}^2-2$ whose Julia set is the interval $\left[-2,2\right]$.

ii) For $a=\frac{5}{36}$, the critical orbit of $F_a$ is $0\mapsto\infty\mapsto\frac{5}{36}\mapsto-\frac{3}{4}\circlearrowleft$ (by Relation~\eqref{schwarz_cardioid}). So, $\frac{5}{36}$ is a Misiurewicz parameter of $\cC(\mathcal{S})$. This map is the analogue of the anti-polynomial $\overline{z}^2-1.543689\cdots$, whose critical point also lands at the separating fixed point in three steps.

\section{An example of mating in the circle-and-cardioid family}\label{basilica_mating_sec}

Recall that in Section~\ref{deltoid_reflection}, we showed that the Schwarz reflection map of the deltoid is the unique conformal mating of the anti-polynomial $\overline{z}^2$ and the reflection map $\rho$. In \cite{LLMM2}, we will carry out a systematic study of the parameter space of the circle-and-cardioid family, and show that each geometrically finite Schwarz reflection map in the family $\mathcal{S}$ is a conformal mating of a quadratic anti-polynomial and the reflection map $\rho$. We conclude this paper by illustrating this mating phenomenon with a simple example.

We set $a_0=0$. The map $F_{a_0}$ has a super-attracting $2$-cycle $0\leftrightarrow\infty$. The dynamical $\frac13$ and $\frac23$-rays of $F_{a_0}$ land at $\alpha_a$, which is the unique common boundary point of the periodic super-attracting Fatou components. By \cite[Proposition 8.6]{LLMM2}, iterated pull-backs of the leaf connecting $\frac13$ and $\frac23$ (in $\overline{\D}$) under $\rho$ are pairwise disjoint and their closure in $\Q/\Z$ is the preperiodic lamination $\lambda(F_{a_0})$.

On the other hand, for $c_0=-1$, the Basilica anti-polynomial $f_{c_0}(z)=\overline{z}^2-1$ has a super-attracting $2$-cycle $0\leftrightarrow-1$. The dynamical $\frac13$ and $\frac23$-rays of $f_{c_0}$ land at the $\alpha$-fixed point of $f_{c_0}$, which is the unique common boundary point of the periodic super-attracting Fatou components.  By \cite[\S~25.7]{L6} (also compare \cite[Proposition~6.8]{kiwi}), iterated pull-backs of the leaf connecting $\frac13$ and $\frac23$ (in $\overline{\D}$) under $m_{-2}$ are pairwise disjoint and their closure in $\Q/\Z$ is the rational lamination $\lambda(f_{c_0})$.

Since $\mathcal{E}$ fixes $\frac13$ and $\frac23$, it follows that $\mathcal{E}_{\ast}(\lambda(F_{a_0}))=\lambda(f_{c_0})$. By \cite[Proposition~11.1]{LLMM2}, there exists a topological conjugacy between $F_{a_0}:K_{a_0}\to K_{a_0}$ and $f_{c_0}:\mathcal{K}_{c_0}\to\mathcal{K}_{c_0}$ that is conformal on the interior (compare Figure~\ref{schwarz_basilica}).

\begin{figure}[ht!]
\centering
\includegraphics[width=0.32\linewidth]{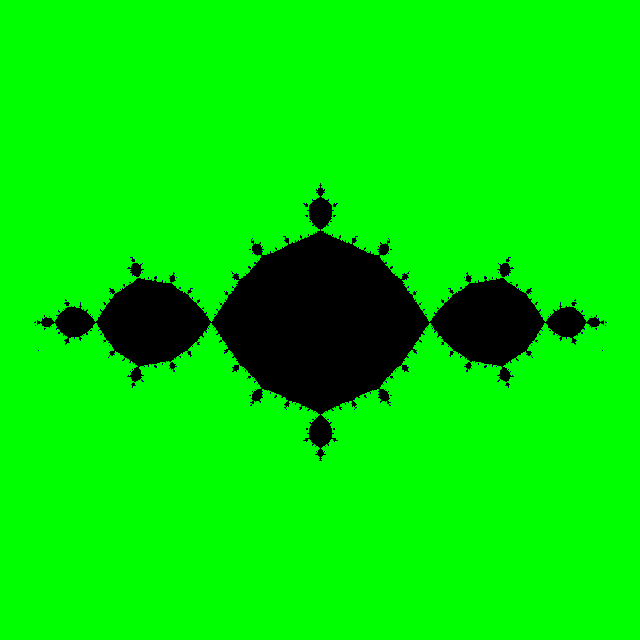}\ \includegraphics[width=0.32\linewidth]{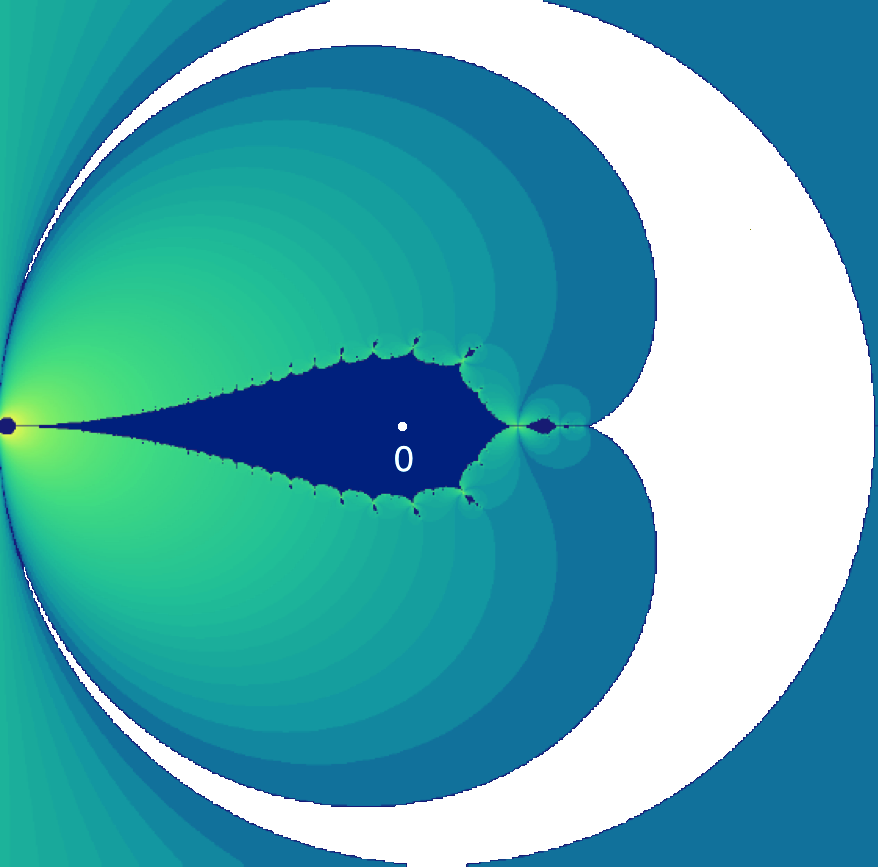}
\caption{The filled Julia set of the map $\overline{z}^2-1$ and the part of the non-escaping set of $F_0$ (in dark blue) inside the cardioid are shown. Both maps have a super-attracting $2$-cycle.}
\label{schwarz_basilica}
\end{figure}

Let us now consider the two conformal dynamical systems $f_{c_0}:\mathcal{K}_{c_0}\to\mathcal{K}_{c_0}$ and $\rho:\overline{\mathbb{D}}\setminus\Int{\Pi}\to\overline{\mathbb{D}}$. We use the mating tool $\xi:=\phi_{c_0}^{-1}\circ\mathcal{E}:\mathbb{T}\to\mathcal{J}_{c_0}$ (which semi-conjugates $\rho$ to $f_{c_0}$) to glue $\overline{\mathbb{D}}$ outside $\mathcal{K}_{c_0}$. 

Set $X~=~\overline{\mathbb D}~\vee_{\xi}~\mathcal{K}_{c_0}$, and $Y=X\setminus\Int{\Pi}.$ (This is a slight abuse of notation. We have denoted the image of $\Int{\Pi}\subset\D$ in $X$ under the gluing by $\Int{\Pi}$.)
We will argue that $X$ is a topological sphere. Since $\mathcal{K}_{c_0}$ is homeomorphic to $\overline{\mathbb{D}}/\lambda_{\R}(f_{c_0})$, it follows that $X$ is topologically the quotient of the $2$-sphere by a closed equivalence relation such that all equivalence classes are connected and non-separating, and not all points are equivalent. It follows by Moore's theorem that $X$ is a topological $2$-sphere \cite[Theorem~25]{Moore}. Moreover, $Y$ is the union of two closed Jordan disks with a single point of intersection in $X$.

The well-defined topological map $\eta~\equiv~\rho~\vee_{\xi}~f_{c_0}:~ Y\to X$
is the topological mating between $\rho$ and $f_{c_0}$.

The conjugacy between $F_{a_0}:K_{a_0}\to K_{a_0}$ and $f_{c_0}:\mathcal{K}_{c_0}\to\mathcal{K}_{c_0}$ and the conjugacy between $F_{a_0}:T_{a_0}^\infty\setminus T_{a_0}^0\to T_{a_0}^\infty$ and $\rho:\D\setminus\Pi\to\D$ obtained in Proposition~\ref{schwarz_group} glue together to produce a homeomorphism 
$H:  (X,Y) \rightarrow  (\widehat{\C}, \overline{\Omega}_{a_0}) $
which is conformal outside $H^{-1}(\Gamma_{a_0})$, and which conjugates $\eta$ to $F_{a_0}$.
It endows $X$ with a conformal structure compatible with the one on $X\setminus H^{-1}(\Gamma_{a_0})$
that turns $\eta$ into an anti-holomorphic map conformally conjugate to $F_{a_0}$. 
In this way, the mating $\eta$ provides us with a model for $F_{a_0}$.   

Moreover, there is only one conformal structure on $X$ compatible with the standard structure on $X\setminus H^{-1}(\Gamma_{a_0})$. Indeed, another structure would result in a non-conformal homeomorphism from $\widehat \C$ to itself which is conformal outside $\Gamma_{a_0}$, contradicting the conformal removability of $\Gamma_{a_0}$ (see Theorem~\ref{tiling_set_John_1}). 
In this sense, the map $F_{a_0}$ is the unique conformal mating of the map $\rho$ arising from the ideal triangle group and the Basilica anti-polynomial $\overline{z}^2-1$.

\end{document}